\documentclass[10pt]{amsart}
\usepackage{verbatim}
\usepackage{graphicx}
\usepackage{svg}
\usepackage{subfigure}
\usepackage{makecell}
\usepackage[export]{adjustbox}
\usepackage{amssymb}
\usepackage{tikz, tikz-cd}
\usetikzlibrary{decorations.markings, knots, decorations.pathreplacing}
\usetikzlibrary{positioning, fit, backgrounds,through, calc}
\usepackage{fullpage}
\usepackage[all]{xy}
\usepackage[margin=1in, headheight=25pt, headsep=20pt]{geometry}
\usepackage{amsmath}
\usepackage{amsthm}
\usepackage{amsfonts}
\usepackage{comment}
\usepackage{amsthm}
\usepackage{array} 
\usepackage{color}
\usepackage[utf8]{inputenc}
\usepackage[english]{babel}
\usepackage{hyperref}
\usepackage{pdfpages}
\usepackage{textcomp}
\usepackage{slashbox}
\usepackage{listings}
\usepackage{caption}
\usepackage{quiver}
\hypersetup{
    colorlinks,
    linkcolor={red!50!black},
    citecolor={blue!50!black},
    urlcolor={blue!80!black}
    }

\newcolumntype{L}{>{$}l<{$}}

\DeclareMathSymbol{\shortminus}{\mathbin}{AMSa}{"39}

\newtheorem{theorem}{Theorem}[section]
\newtheorem{lemma}[theorem]{Lemma}

\newtheorem{corollary}[theorem]{Corollary}
\newtheorem{conjecture}[theorem]{Conjecture}
\newtheorem{proposition}[theorem]{Proposition}
\theoremstyle{definition}  
\newtheorem{definition} [theorem] {Definition} 

\newtheorem{example} [theorem] {Example}
\newtheorem{remark} [theorem] {Remark}
\newtheorem{question} [theorem] {Question}
\newtheorem{construction} [theorem] {Construction}

\theoremstyle{definition}

\newtheorem{defin}[theorem]{Definition}

\newcommand{\Q}{{\mathbb{Q}}}

\newcommand{\Z}{{\mathbb{Z}}}

\newcommand{\A}{\mathcal{A}}
\newcommand{\Lk}{\mathcal{L}}
\newcommand{\fil}{{\mathcal{F}}}
\newcommand{\Mt}{{\mathcal{M}}}
\newcommand{\Pt}{{\mathcal{P}}}
\newcommand{\Ch}{{\mathcal{C}}}

\newcommand{\To}{{\mathfrak{T}_o}}
\newcommand{\Res}{\text{Res}}
\newcommand{\im}{\text{im}}

\tikzset{red dot/.style={thick}}

\DeclareMathOperator{\CKh}{CKh}
\DeclareMathOperator{\CLee}{CLee}
\DeclareMathOperator{\CST}{CST}
\DeclareMathOperator{\Kh}{Kh}

\title[A spanning tree model for Khovanov homology]{A spanning tree model for Khovanov homology, Rasmussen's s-invariant and exotic discs in the $4$-ball}
\author[Banerjee]{Aninda Banerjee}
\address[]{IAI, TCG CREST, Kolkata, India}
\address{NIT, DURGAPUR}

\email{anindabanerjee24@gmail.com}

\author[Chakraborty]{Apratim Chakraborty}

\address[]{IAI, TCG CREST, Kolkata, India}
\address[]{Academy of Scientific and Innovative Research (AcSIR), Ghaziabad- 201002, India}

\email{apratimn@gmail.com}

\author[Das]{Swarup Kumar Das}
\address[]{IAI, TCG CREST, Kolkata, India}
\address{NIT, DURGAPUR}

\email{swarupdas.math@gmail.com}

\begin{document}


\begin{abstract}
	The checkerboard coloring of knot diagrams offers a graph-theoretical approach to address topological questions. Champanerkar and Kofman defined a complex generated by the spanning trees of a graph obtained from the checkerboard coloring whose homology is the reduced Khovanov homology. Notably, the  differential in their chain complex was not explicitly defined. We explicitly define the combinatorial form of the differential within the spanning tree complex. We additionally provide a description of Rasmussen's $s$-invariant within the context of the spanning tree complex. Applying our techniques, we identify a new infinite family of knots where each of them bounds a set of exotic discs within the 4-ball.
\end{abstract}

\maketitle

\setlength\parindent{0pt}


\begin{section}{Introduction}
		
	Khovanov homology is an invariant of links in $S^3$ that categorifies the Jones polynomial. Champanerkar and Kofman \cite{kh_st} and Wehrli \cite{wehrli} introduced the spanning tree model for the Khovanov homology. In their construction, the chain complex is generated by the spanning trees of the Tait graph associated with a given knot diagram. The homology of this spanning tree complex is isomorphic to the reduced Khovanov homology. However, the exact combinatorial form of the differential remained unknown in the bigraded case.\\

  Given a connected link diagram $\Lk$, let $G_{\Lk}$ denote its associated Tait graph. Throughout this paper, we explore Khovanov homology through the perspective of the Tait graph. In Section 2, we establish a correspondence between these two viewpoints. Under this framework, the generators of the Khovanov chain complex, which are enhanced states, correspond to enhanced spanning subgraphs of the Tait graph.\\
    
    We then provide an acyclic matching within the reduced Khovanov chain complex, denoted as $(\CKh^+(\Lk), \partial_{\Kh}^+)$. From the acyclic matching, we obtain the bigraded Morse complex $(\CST^+(G_{\Lk}), \partial_{ST}^+)$, which is generated by the spanning trees of $G_{\Lk}$. The differential operator $\partial_{ST}^+$ is entirely determined by the graph-theoretic properties of the Tait graph $G_{\Lk}$. A similar technique has been applied by the authors and Paul in \cite{STM-chromatic} in order to derive a spanning tree model for chromatic homology which categorifies the chromatic polynomial. As a consequence of the fundamental theorem of algebraic discrete Morse theory, we present the following theorem, providing a spanning tree model for reduced Khovanov homology where the differential can be explicitly computed using the graph theoretic information of the Tait graph:
	
	\begin{theorem}\label{maintheorem}
		For a connected link diagram $\Lk$, the cohomology of the complex $(\CST^+(G_{\Lk}), \partial_{ST}^+)$ is isomorphic to the reduced Khovanov homology ${\Kh}^+(\Lk)$ (Also denoted by $\widetilde{\Kh}(\Lk)$).
	\end{theorem}
	
	We extend our construction of acyclic matching to the unreduced Khovanov chain complex of an arbitrary connected link diagram $\Lk$, denoted as $(\CKh(\Lk), \partial_{\Kh})$, in order to obtain the Morse complex $(\CST(G_{\Lk}), \partial_{ST})$. This complex is generated by the signed spanning trees of the Tait graph $G_{\Lk}$, and the differential can be described from the graph theoretic information of the Tait graph as well.
	
	\begin{theorem}\label{unreduced maintheorem}
		For a link diagram $\Lk$, the cohomology of the complex $(\CST(G_{\Lk}), \partial_{ST})$ is isomorphic to the Khovanov homology ${Kh}(\Lk)$.
	\end{theorem}
	
	As an application,  we deduce the following result about existence of $\mathbb{Z}_2$-torsion in Khovanov homology of alternating links, originally established independently by Asaeda and Przytycki \cite{Asaeda}, and Shumakovitch \cite{Shumakovitch}.
	
	\begin{theorem}\label{torsion theorem}
		Every non-split alternating link except the trivial knot, Hopf link and their connected sums have a $\Z_2-$torsion in their Khovanov homology.
	\end{theorem}
	
	\begin{figure}[ht]
		\centering
		\includesvg[width=0.8\textwidth]{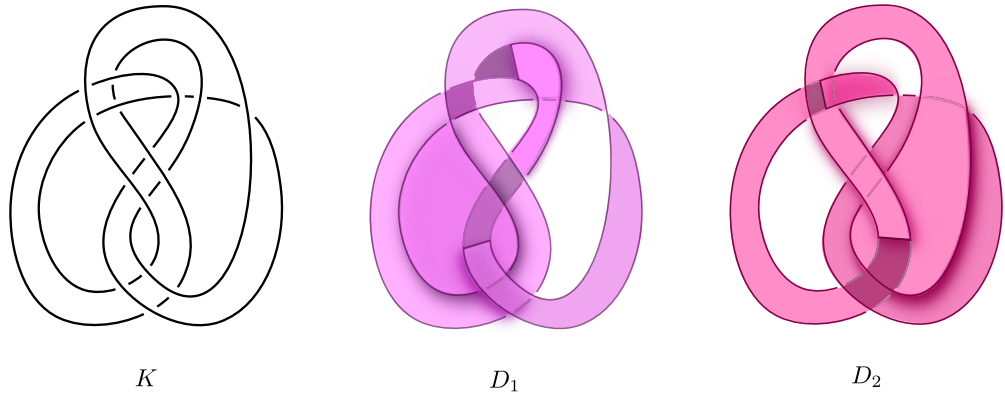}
		\caption{The slice disks $D_1$ and $D_2$ bounded by the knot $K$ forms an exotic pair detected by the distinguished oriented resolution cycle under Khovanov  cobordism map}
		\label{exotic_disks}
	\end{figure}
	
	We also study the spanning tree model for Lee's complex. In the spanning tree model of Lee's complex we have a filtration grading of the homology classes denoted by $s_{ST}$. We introduce the notion of \textbf{orientation preserving tree} $\mathbf{\mathfrak{T}_o}$ which can be thought of as spanning tree version of the oriented resolution.  ${\mathfrak{T}_o}^{\pm}$ is a cycle in the unreduced spanning tree complex for a suitable ordering. We prove that Rasmussen's $s$-invariant can be recovered from the filtration gradings of homology classes of signed orientation preserving trees ${\mathfrak{T}_o}^+$ and  ${\mathfrak{T}_o}^-$.
	
	\begin{theorem}\label{s-invariant thm}
		
		Let $\mathfrak{T}_o$ be the orientation preserving tree in the Tait graph $G_{\Lk}$ of a knot $K$ then,
		\[
		s(K) = \frac{s_{ST}([\mathfrak{T}_o^{+}])+ s_{ST}([\mathfrak{T}_o^{-}])}{2}.  
		\]
	\end{theorem}
	
	As a corollary, we obtain the following lower bound in terms of quantum grading ($j-$grading) of the signed tree ${\mathfrak{T}_o}^+$.

	\begin{corollary} \label{s-inv lower bound}
    Let ${\mathfrak{T}_o}$ be the orientation preserving tree in the Tait graph $G_{K}$ of a knot $K$.  Then, the Rasmussen's $s$-invariant of the knot $K$ gets a lower bound in terms of the quantum grading of $\mathfrak{T}_o^{+}$ i.e.
    
     \[s(K) \geq j(\mathfrak{T}_o^{+})-1. \] 
	\end{corollary}

We observe that this lower bound for a given Tait graph  $G_{K}$ of a knot $K$  agrees with the following lower bound provided by Lobb (See Remark \ref{lobb remark}) which is stronger than the slice-Bennequin bound \cite{Lobb}.
	
	\begin{equation}\label{lobb inequality}
		\begin{aligned}
			2 g^*(K) & \geq s(K) \geq U(D)-2 \Delta(D) \\
			& \geq w(D)-\# \operatorname{nodes}(\mathrm{~T}(\mathrm{D}))+2 \# \operatorname{components}\left(\mathrm{T}^{+}(\mathrm{D})\right)-1.
		\end{aligned}
	\end{equation}
	
	The cycle $[{\mathfrak{T}_o}^{+}]$ corresponds to a cycle $[f({\mathfrak{T}_o}^{+})]$ in Khovanov complex which has interesting properties under Khovanov cobordism map. Hayden and Sunderberg \cite{kh-isaac-hayden} used Khovanov homology to provide examples of knots bounding  an exotic pair of ribbon discs in the four ball. The cycles used to distinguish the ribbon discs are obtained from enhanced states of the oriented resolution. Our prescription of orientation preserving tree provides a cycle $[f({\mathfrak{T}_o}^{+})]$ in the Khovanov complex which we call a \textbf{distinguished oriented resolution cycle} as it is a sum of enhanced states of the oriented resolution. We observe that the distinguished oriented resolution cycle generalizes the construction of distinguishing cycles in the work of  Hayden-Sunderberg\cite{kh-isaac-hayden}. Moreover, the orientation preserving tree $\mathfrak{T}_o$  also motivates the construction of knot diagrams admitting a pair of ribbon disc.  As an application, we find a new infinite family of knots that bound pairs of exotic ribbbon discs in the $4$-ball [See Figure \ref{exotic_disks}].
	
	\begin{theorem}\label{application in exotic disks}
	There exists an infinitely many knots in  $S^3$ that bound a pair of smooth, orientable discs $D, D^{\prime} \subset B^4$ that are topologically isotopic rel boundary yet they are not smoothly isotopic. The distinguished oriented resolution cycle in Khovanov homology obstructs smooth isotopy in these examples.
	\end{theorem}

	In a different direction, Baldwin and Levine \cite{BALDWIN} defined the spanning tree model for $\delta$-graded knot Floer homology with explicitly defined combinatorial differential. Roberts \cite{roberts}, defined a deformation of Khovanov homology and provided a spanning tree model for it with explicit combinatorial differential. Later, Jaeger \cite{jaeger} showed that Robert's deformation agrees with $\delta$-graded reduced characteristic-2 Khovanov homology.
	It is natural to ask if analogous theories can be defined for bigraded knot Floer complex with graph theoretic differential.
	
	\begin{question}
	Can one define an analogous spanning tree model for knot Floer homology (bi-graded) where the differential is explicitly computed using the graph theoretic information of the Tait graph
	\end{question}

	The paper is organized as follows:  \hyperlink{section.2}{Section 2} introduces algebraic Morse theory, basic concepts of Tait graphs and Khovanov homology. In  \hyperlink{section.3}{Section 3} , we define a chain complex generated by the spanning trees and an explicit differential, the cohomology of which yields the reduced Khovanov homology. In  \hyperlink{section.4}{Section 4}, we extend the spanning tree model for the unreduced Khovanov chain complex.  \hyperlink{section.5}{Section 5} provides a description of Rasmussen's s-invariant within the context of the spanning tree complex. Finally, in  \hyperlink{section.6}{Section 6}, we present a new, infinite family of examples of knots with pairs of ribbon disks that are not smoothly isotopic.

\subsection*{Acknowledgements}
We would like to thank Abhijit Champanerkar for his helpful suggestions and encouragement.

\end{section}


\begin{section}{Preliminaries}
\label{prelim}
    \begin{subsection}{Algebraic discrete Morse Theory}
	
	For a given regular CW-complex, Forman \cite{Forman1998} developed Discrete Morse theory which studies a new homotopy-equivalent CW-complex which has fewer cells than the original one. Using this concept, Kozlov \cite{Kozlov2005} and Sk\"{o}ldberg \cite{sk-berg} independently developed an algebraic analogue of this theory which is also known as the Algebraic discrete Morse theory. In simple terms,this theory takes a freely generated chain complex $(C_\star, \partial_\star)$, provides an acyclic matching on the corresponding Hasse diagram $\mathcal{H}(C_\star, \partial_\star)$, a directed weighted graph representing $(C_\star, \partial_\star)$, and as it turns out that $(C_\star, \partial_\star)$ and the Morse complex of $\mathcal{H}(C_\star, \partial_\star)$ are quasi-isomorphic. Thus, we get a combinatorial description of the homology of $(C_\star, \partial_\star)$. For our purpose, we define this theory following \cite{Min-Res} for a freely generated chain complex where each $C_i$ is a finitely generated free module over a commutative ring with unity. An analogous theory shall also follow for general freely generated chain and co-chain complexes.\\
	
	Let $\mathcal{R}$ be a commutative ring with $1$, and $C_\star =(C_i,\partial_i)_{i\geq 0}$ be a freely generated chain complex of finitely generated free $\mathcal{R}$-modules. Fix a basis $B=\bigcup_{i \geq 0} B_i$ such that 
	
	$$
      C_i \backsimeq \bigoplus_{c \in B_i} \mathcal{R}c
	$$
	
	and let the differential $\partial_i : C_i \rightarrow C_{i-1}$ be described as
	
	$$
	 \partial_i(c) = \sum_{c' \in B_{i-1}} [c:c'].c'.
	$$
	
	Here, $[c:c']$ denotes the incidence number or the coefficient between $c$ and $c'$.\\
	
	Given the complex $(C_\star,\partial_\star)$, we construct the Hasse diagram corresponding to $(C_\star,\partial_\star)$, a directed, weighted graph $\mathcal{H}(C_\star,\partial_\star)=(V,E)$ such that $V=B$, and the elements of $E$ are triples $(v,v',w)$, where $v,v' \in V$, $v$ is directed towards $v'$ (denoted by $v \rightarrow v'$) and $w \in \mathcal{R}$ is the weight, are defined as follows:
	
	\[
	 (\{c,c'\},[c:c']) \in E \Longleftrightarrow c \in B_i, c' \in B_{i-1}, \text{ and } [c:c'] \neq 0  
	\]
	
	Thus for such a given graph, we now define the meaning of an \textit{acyclic matching} on $\mathcal{H}_\Mt(C_\star,\partial_\star)$.
	
	\begin{definition}
		A subset $\Mt \subset E$ is called an \textit{acyclic matching} if it satisfies the following three conditions:
		
		\begin{enumerate}
			\item Each vertex $v\in V$ lies in at most one edge $e \in \Mt$.
			\item For all edges $(\{c,c'\},[c:c']) \in \Mt$, the weight $[c:c']$ is a unit in $\mathcal{R}$.
			\item $\mathcal{H}_\Mt(C_\star,\partial_\star)$ has no directed cycles, where $V(\mathcal{H}_\Mt(C_\star,\partial_\star))=V$ and $E(\mathcal{H}_\Mt(C_\star,\partial_\star))$ is given by 
			
			$$
			  E(\mathcal{H}_\Mt(C_\star,\partial_\star)):=(E\setminus \Mt) \cup \left\{\left(c',c,-\frac{1}{[c:c']}\right) \bigg|\: (c,c',[c:c']) \in \Mt \right\}
			$$
		\end{enumerate}
		
	\end{definition}
	
	Given an acyclic matching $\Mt$ on $\mathcal{H}(C_\star,\partial_\star)$, we now have a couple of notations associated to it,
	
	\begin{enumerate}
		\item A vertex $v \in V$ is called \textit{critical} with respect to $\Mt$ if $v \not\in e$ for all $e \in \Mt$. For a given chain group $C_i$, we have the following set:
		 
		$$
		 B_i^\Mt:=\{c \in C_i \mid c \text{ is critical }\}
		$$
		
		\item Let $\A(c,c')$ denote the set of directed paths from $c$ to $c'$ in $\mathcal{H}_\Mt(C_\star,\partial_\star)$.
		
		\item For a given edge $e=(c \rightarrow c') \in E(\mathcal{H}_\Mt(C_\star,\partial_\star))$, the weight can be redefined as:
		
		$$
		 w(c \rightarrow c') = \begin{cases}
		 	[c:c'], & \text{ if } e \in E \cap E(\mathcal{H}_\Mt(C_\star,\partial_\star)) \\
		 	
		 	-\frac{1}{[c:c']}, & \text{ otherwise }
		 \end{cases}
		$$
		
		\item For $\Pt \in \A(c,c')$, where $\Pt=(c=c_1 \rightarrow c_2 \rightarrow \cdots \rightarrow c_n=c')$, the weight of $\Pt$ is given by
		 
		$$
		 w(\Pt):=\prod_{i=1}^{n-1} w(c_i \rightarrow c_{i+1})
		$$
		
		\item $\Gamma(c,c')=\sum_{\Pt \in \A(c,c')}w(\Pt)$ denotes the sum of the weights of all directed paths from $c$ to $c'$.
	\end{enumerate}
	
	We are now ready to define the \textit{Morse complex} associated to $(C_\star,\partial_\star)$ and an acyclic matching $\Mt$ on its Hasse diagram.
	
	\begin{definition}
		Let $(C_\star,\partial_\star)$ be a free chain complex with an acyclic matching $\Mt$ in $\mathcal{H}(C_\star,\partial_\star)$. The \textit{Morse complex} $\left(C_\star^{\Mt},\partial_\star^{\Mt}\right)$, with respect to $\Mt$, is defined to be the freely generated chain complex $C_\star^{\Mt} = \bigoplus_{i \geq 0} C_i^\Mt$, where each $C_i^\Mt$ is a finitely generated $\mathcal{R}$-module with 
		
		$$
		 C^\Mt_i := \mathcal{R}[B_i^\Mt]
		$$
		
		and $\partial_\star^\Mt$ is defined for each $i \geq 0$, $\partial_i^\Mt: C_i^\Mt \rightarrow C_{i-1}^\Mt$, as follows:
		
		$$
		 \partial_i^\Mt(c) := \sum_{c' \in B^\Mt_{i-1}} \Gamma(c,c').c'.
		$$ 
	\end{definition}
	
	\begin{lemma}
		$\partial_\star^\Mt$ is a differential.
	\end{lemma}
	
	Associated to a given freely generated chain complex $(C_\star,\partial_\star)$ and an acyclic matching $\Mt$ on its Hasse diagram, we have the following maps:
	
	\begin{equation} \label{f map in dmt}
		\begin{aligned}
			& f: C_\star^{\Mt} \rightarrow C_\star \\
			& f_i: C_i^\Mt \rightarrow C_i, &  f_i(c) := \sum_{c' \in B_i} \Gamma(c,c').c',
		\end{aligned}
	\end{equation}
	
	\begin{equation} \label{g map in dmt}
		\begin{aligned}
			& g: C_\star \rightarrow C_\star^\Mt \\
			& g_i: C_i \rightarrow C_i^\Mt, &  g_i(c) := \sum_{c' \in B_i^\Mt} \Gamma(c,c').c',
		\end{aligned}
	\end{equation}
	
	\begin{equation}
		\begin{aligned}
			& \chi: C_\star \rightarrow C_\star \\
			& \chi_i: C_i \rightarrow C_{i+1}, &  \chi_i(c) := \sum_{c' \in B_{i+1}} \Gamma(c,c').c',
		\end{aligned}
	\end{equation}
	
	\begin{lemma} \label{chain homotopic maps}
		For the maps defined above, we have the following:
		\begin{enumerate}
			\item $f$ and $g$ are chain maps.
			
			\item $f$ and $g$ define a chain homotopy. In particular, we have
			  \begin{enumerate}
			 	\item $f_i \circ g_i - id_{C_i} = \partial \circ \chi_{i+1} + \chi_i \circ \partial$,
			 	
			 	\item $g_i \circ f_i - id_{C_i^\Mt} = 0$.
			  \end{enumerate}
		\end{enumerate}
	\end{lemma}
	 
	 Finally, as a consequence of the above lemma, we have the following central result of algebraic discrete Morse theory,
	
	\begin{theorem}[Theorem 2.2, \cite{Min-Res}]\label{dmt maintheorem}
		Let $(C_\star,\partial_\star)$ be a freely generated chain complex and $\Mt$ be an acyclic matching on $\mathcal{H}(C_\star,\partial_\star)$. Then, we the following homotopy equivalence of chain complexes:
		
		$$
		 (C_\star,\partial_\star) \backsimeq \left(C_\star^\Mt,\partial_\star^\Mt\right)
		$$
		
		Moreover, for all $i \geq 0$, we have the isomorphism,
		
		$$
		H_i(C_\star) \cong H_i(C_\star^\Mt)
		$$
	\end{theorem}
	
\end{subsection}

\begin{subsection}{Links and Tait graph}
    
Given a connected link diagram $\Lk$, one can view it as a connected 4-regular planar graph $G$, where the vertices of $G$ are crossings of $\Lk$ and there is an edge between two such vertices if there is an arc in $\Lk$ between the two corresponding crossings. One can show that the dual graph $G^*$ is a bipartite graph and hence, we can use two colors (black and white) to color the vertices of each partition. Thus, one can color the faces of $\Lk$ using two colors such that no two adjacent faces (faces sharing a common arc) have same color. This is also known as \textit{checkerboard coloring} of $\Lk$. \\  
Given a checkerboard coloring of $\Lk$, one can associate a graph to it in which the vertices represent the black regions of the diagram, and the edges represent the common crossings shared by the black regions. Additionally, each edge can be assigned a sign ($+$ or $-$) based on whether the corresponding crossing connects or separates two black regions when a $A-$smoothing is performed on it.  This resulting signed graph is called the \textit{Tait graph} of $\Lk$, which we denote by $G_\Lk$ (See Figure \ref{Three crossing trefoil}).   
		
		\begin{figure}[ht]
			\centering
			\subfigure{\includesvg[width=2.5cm,inkscape=force]{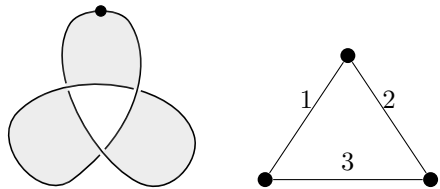}}
			\qquad
			\subfigure{\begin{tikzpicture}[vertex/.style={circle,fill,inner sep=2pt},scale=0.55]
					\node[vertex] at (0,0) (a) {};
					\node[vertex] at (2,3) (b) {};
					\node[vertex] at (4,0) (c) {};
					\draw (a)--(b) node[midway,above] {$1$};
					\draw (b)--(c) node[midway,above] {$2$};
					\draw (c)--(a) node[midway,above] {$3$};
			\end{tikzpicture}}
			\caption{Right-handed trefoil diagram $D$ and its Tait graph $G_D$}
			\label{Three crossing trefoil}
		\end{figure}

  \begin{figure}[ht]
	\centering
	\includesvg[width=0.3\textwidth]{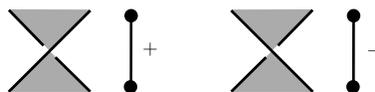}
	\caption{Sign of an edge of the Tait graph based on the crossing}	
 \end{figure} 

Interestingly, it turns out that one can express the Jones polynomial, a classical invariant of a link, in terms of some combinatorial information of the Tait graph. Thistlethwaite \cite{th_jp} showed that one can assign certain monomials to each of the spanning trees of the Tait graph in order to expand the Jones polynomial of the link in terms of the spanning trees. More precisely, each spanning tree of $G_\Lk$ can be assigned an \textit{activity word} based on a fixed ordering on the edges of $G_\Lk$.\\
For a given spanning tree $T$ of $G_\Lk$ and for each edge $e \in G_\Lk$, we have the following two definitions:

$$
\begin{aligned}
    \text{For } e \in T, \quad & \emph{cut}(T,e)=\{e' \in G_\Lk \mid e' \text{ connects } T \setminus e\} \\
    \text{For } e \not\in T, \quad  & \emph{cyc}(T,f)=\{f' \in G_\Lk \mid f' \text{ belongs to the unique cycle in } T \cup f\}
\end{aligned}
$$

Now let us fix an ordering on the edges of $G_\Lk$. Given an edge $e \in T$, if $e$ is the smallest edge in $\emph{cut}(T,e)$, then $e$ is said to be an \textit{internally active (live)} edge of the tree $T$ else it is said to be an \textit{internally inactive (dead)} edge. Similarly, if $e \not\in T$ and $e$ is the smallest edge in $\emph{cyc}(T,e)$, then $e$ is said to be \textit{externally active (live)} edge of $T$ or else it is said to be \textit{externally inactive (dead)} edge of T. In other words, for a positive edge $e \in G_\Lk$, the activity of $e$ is the following function:

     $$
     \begin{aligned}
      & a_T : E^+(G_\Lk) \rightarrow \{L,D,l,d\} \\
      & a_T(e) = \begin{cases}
          L & \text{if $e$ is internally active} \\
          D & \text{if $e$ is internally inactive} \\
          l & \text{if $e$ is externally active} \\
          d & \text{if $e$ is externally inactive}
      \end{cases}
     \end{aligned} 
     $$
		
   Where $E^+(G_\Lk)$ is the set of positive edges of $G_\Lk$. Similarly, if the edge is negative, the co-domain for $a_T$ becomes $\{\Bar{L},\Bar{D},\Bar{l},\Bar{d}\}$. Hence, the activity word $W(T)$ for a given spanning tree is defined as 

   $$
     W(T)=a_T(e_1)a_T(e_2)\cdots a_T(e_n)
   $$

   Table \ref{Three trefoil table} lists all the spanning trees of the Tait graph of the right-handed trefoil and their corresponding activity words for a given ordering in Figure \ref{Three crossing trefoil}.\\
		
	\begin{table}[ht]
			\centering
			\renewcommand{\arraystretch}{1.5}
			\begin{tabular}{|c|c|c|c|}
				\hline
				Spanning trees & \adjustbox{valign=m}{\Gape[0.2cm][0cm]{
						\begin{tikzpicture}[vertex/.style={circle,fill,inner sep=2pt},baseline]
							\node[vertex] at (0,0) (a) {};
							\node[vertex] at (0.8,1) (b) {};
							\node[vertex] at (1.5,0) (c) {};
							\draw (a)--(b);
							\draw (a)--(c);
							\node [below=0.2cm, align=flush center] at (0.8,0) {$T_1$};
				\end{tikzpicture}}} & 
				\adjustbox{valign=m}{
					\Gape[0.2cm][0cm]{
						\begin{tikzpicture}[vertex/.style={circle,fill,inner sep=2pt},baseline]
							\node[vertex] at (0,0) (a) {};
							\node[vertex] at (0.8,1) (b) {};
							\node[vertex] at (1.5,0) (c) {};
							\draw (a)--(b);
							\draw (b)--(c);
							\node [below=0.2cm, align=flush center] at (0.8,0) {$T_2$};
				\end{tikzpicture}}} &  
				\adjustbox{valign=m}{
					\Gape[0.2cm][0cm]{
						\begin{tikzpicture}[vertex/.style={circle,fill,inner sep=2pt},baseline]
							\node[vertex] at (0,0) (a) {};
							\node[vertex] at (0.8,1) (b) {};
							\node[vertex] at (1.5,0) (c) {};
							\draw (b)--(c);
							\draw (a)--(c);
							\node [below=0.2cm, align=flush center] at (0.8,0) {$T_3$};
				\end{tikzpicture}}} \\
				\hline
				Activity word & $LdD$ & $LLd$  & $lDD$\\
				\hline
			\end{tabular}
			\vspace{0.2cm}
			\caption{Spanning trees and activity word of the right-handed trefoil}
			\label{Three trefoil table}
		\end{table}

\end{subsection}

\begin{subsection}{Spanning tree and Skein resolution tree}

A \textit{twisted unknot} is a knot diagram which is isotopic to the round unknot only through Reidemeister I moves. A \textit{skein resolution tree} $\mathcal{T}$ corresponding to $\Lk$ is a rooted tree whose leaves are twisted unknots and the root is $\Lk$. To construct $\mathcal{T}$, fix an ordering on the edges of $G_\Lk$ and take the reverse ordering on the crossings of $\Lk$. A crossing is called \textit{nugatory} if either $A$ or $B$-smoothing on it disconnects the link diagram.  Now starting with $\Lk$ as the root of $\mathcal{T}$, resolve each crossing in the given order to form branches and leave nugatory crossings unsmoothed and move on to the next crossing. Continue until all the crossings are nugatory. Observe that a knot diagram is twisted unknot if and only if all of its crossings are nugatory. Thus, the leaves of $\mathcal{T}$ are twisted unknots. One can associate a unique spanning tree of $G_\Lk$ to each of these twisted unknots of $\mathcal{T}$ (See Figure \ref{skein tree}). In order to get a twisted unknot $U(T)$ from a spanning tree $T$, smooth the crossings of $\Lk$ corresponding to the dead edges of $T$ using Table \ref{smoothing table}. 

\setlength{\tabcolsep}{5mm}
\def\arraystretch{1.75}
    
\begin{table}[ht]
    \centering
    \begin{tabular}{|c|c|c|c|}
     \hline
    \begin{tabular}{cc}
       $L$  & $D$  
    \end{tabular} & 
    \begin{tabular}{cc}
        $l$ & $d$  
    \end{tabular} &
    \begin{tabular}{cc}
        $\Bar{L}$ & $\Bar{D}$  
    \end{tabular} &
    \begin{tabular}{cc}
        $\Bar{l}$ & $\Bar{d}$  
    \end{tabular} \\
    \hline
    \begin{tabular}{cc}
       $B$  & $A$  
    \end{tabular} & 
    \begin{tabular}{cc}
        $A$ & $B$  
    \end{tabular} &
    \begin{tabular}{cc}
        $A$ & $B$  
    \end{tabular} &
    \begin{tabular}{cc}
        $B$ & $A$  
    \end{tabular} \\
    \hline
    \includesvg[scale=0.8]{st-1.svg} &
    \includesvg[scale=0.75]{st-2.svg} &
    \includesvg[scale=0.83]{st-3.svg} &
    \includesvg[scale=0.77]{st-4.svg} \\
    \hline
    \end{tabular} 
    \caption{Smoothings corresponding to edge activities}
    \label{smoothing table}
\end{table}

\begin{lemma}
   For a given fixed ordering on the edges of $G_\Lk$, there is a bijection between the set of spanning trees of $G_\Lk$ and the twisted unknots of $\mathcal{T}$ obtained from this ordering.
\end{lemma}
		
\begin{proof}
    Consider a regular neighborhood of $T$ in the plane and put a crossing for each live edge $e \not \in T$. This is the only crossing in $\emph{cyc}(T,e)$ and hence the round unknot differs only by Reidemeister I moves. Similarly, for edges $e \in T$, do the above procedure in the dual tree $T^*$ and use the fact that $\emph{cut}(T,e) = \emph{cyc}(T^*,e^*)$.\\
    Conversely, for a given twisted unknot $U \in \mathcal{T}$, we consider the coloring on $U$ induced from $\Lk$ and keep all such edges of $G_\Lk$ which are either contained in the black region of $U$ or a crossing of $U$ corresponding to an edge separates two black regions of $U$ when it is smoothed according to Table \ref{smoothing table}.
\end{proof}

\begin{figure}[ht]
    \centering
    \includesvg[width=0.45\textwidth]{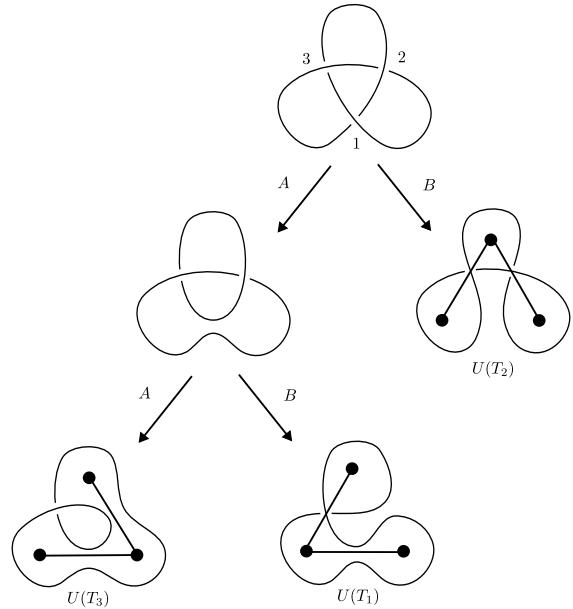}
    \caption{Skein resolution tree of the left-handed trefoil}
    \label{skein tree}
\end{figure}

Each twisted unknot $U(T)$ is a partial smoothing of $\Lk$ which comes from Table \ref{smoothing table} and we label the unsmoothed crossings of $U(T)$ by $\star$ and denote this partial smoothing by a tuple $(x_1,x_2,\cdots,x_n)$. 
			
\begin{definition}\label{partial order on spanning trees defn}
     Let $\Lk$ be a connected link diagram with $n$ ordered crossings. For any two spanning trees $T,T'$ of $G_\Lk$, let $(x_1,\cdots,x_n)$ and $(y_1,\cdots,y_n)$ be the corresponding partial smoothings of $\Lk$ giving $U(T)$ and $U(T')$ respectively. Define a relation $T > T'$ if and only if 
				
		$$
		    (x_1,\cdots,x_n) > (y_1,\cdots,y_n)
		$$
				
     if for each $i$, $y_i=A$ implies $x_i=A$ or $\star$, and there exists $i$ such that $x_i=A$ and $y_i=B$. This relation is not transitive in general and thus taking a transitive closure gives a partial order due to Proposition \ref{partially ordered trees are distinct}.
\end{definition}
			
\begin{proposition}[\cite{kh_st}]\label{partially ordered trees are distinct}
	If $T > \cdots > T'$ then $T \neq T'$.
\end{proposition}			
			
  \begin{example}
      For the right-handed trefoil from Figure \ref{Three crossing trefoil},
				
		\begin{table}[ht]
			\centering
			\renewcommand{\arraystretch}{1.3}
			\begin{tabular}{|c|c|c|}
				\hline
				$T_1$ & $T_2$ & $T_3$\\
				\hline
				$\star BA$ & $\star \star B$ & $\star AA$\\
				\hline
			\end{tabular}
			\label{partial ordering on trefoil}
		\end{table}		
				
	We get only one sequence: $T_3 > T_1 > T_2$. 
				
\end{example}

In the next subsection, we introduce Khovanov homology of a connected link diagram $\Lk$ in terms of the underlying Tait graph $G_\Lk$ and we observe that the Khovanov chain complex of $\Lk$ is ``disjoint union" of the Khovanov chain complexes of each twisted unknot $U(T)$ for each spanning tree $T$ of $G_\Lk$.
			
\end{subsection}
    
\begin{subsection}{Khovanov homology}
        
Let $\Lk$ be a connected link diagram and $G_\Lk$ be its Tait graph. Let $H$ be a spanning subgraph of $G_\Lk$ and $C(H)$ be the set of connected components of $H$. Denote $F(C)$ to be the set of faces of the subgraph $C \in C(H)$ and let,
		
		$$
		   F(H) = \bigcup_{C \in C(H)} F(C)
		$$

   We now define the generators of the Khovanov chain complex in terms of these spanning subgraphs and their faces.
		
	\begin{definition}
		An \textit{enhanced spanning subgraph} of $G_\Lk$ is a pair $(H,\epsilon)$, where $H$ is a spanning subgraph of $G_\Lk$ and $\epsilon: F(H) \rightarrow \{1,x\}$ is a set-theoretic map.
	\end{definition}
		
		For a given enhanced spanning subgraph $(H,\epsilon)$, we define a bigrading for it,
		
		$$
		 \begin{aligned}
		 	& i((H,\epsilon)) = \#(+ve \text{ edges not in }H) + \#(-ve \text{ edges in }H) - n_-\\
		 	& j((H,\epsilon)) = i((H,\epsilon)) + \left(|\epsilon^{-1}(1)| - |\epsilon^{-1}(x)|\right) + n_+ - n_-.
		 \end{aligned}
		$$

        where, $n_+$ and $n_-$ are number of positive and negative crossings of $\Lk$ respectively. Now we have the following bigraded $\Z$-module:
		
		$$
		     \CKh^{i,j}(\Lk) = \Z[(H,\epsilon) \mid i((H,\epsilon)) = i, j((H,\epsilon)) = j ]
		$$
		
		The Khovanov chain complex is defined as the following bigraded complex,
		
		$$
		   \CKh(\Lk) = \bigoplus_{i,j} \CKh^{i,j}(\Lk)
		$$
		
      For $(H,\epsilon) \in \CKh^{i,j}(\Lk)$, let $A((H,\epsilon)) = \{+ve \text{ edges in } H\} \cup \{-ve \text{ edges not in }H\}$. Then for $e \in A((H,\epsilon))$, Let $H'_e$ be the subgraph obtained from $H$ by either removing $e$ from $H$ if $e \in H$ and $e$ is positive or inserting $e$ into $H$ if $e \not\in H$ and $e$ is negative. Hence, we have $|F(H)| - |F(H'_e)| = \pm 1$. Thus, either two faces in $F(H)$ merges into one or a face splits into two. For each case, we obtain a new set of enhanced spanning subgraphs $(H'_e,\epsilon')$:\\
	    
	    \textbf{case 1:} Suppose two faces $F_1, F_2 \in F(H)$ merges (denoted by $m$) into a single $F_{12} \in F(H'_e)$, then for a face $F \in F(H'_e)$ $\epsilon'$ is given by,
	    
	$$
	     \epsilon'(F) = \begin{cases}
	     	\epsilon(F) & F \neq F_{12}\\
	        m(\epsilon(F_1), \epsilon(F_2)) & F=F_{12}
	     \end{cases} 
	$$

        where, $m(1,1)=1, m(1,x) = m(x,1) = x, m(x,x)=0$.\\
	    
	    \textbf{case 2:} Suppose a face $F_{12} \in F(H)$ splits (denoted by $\Delta$) into two faces $F_1, F_2 \in F(H'_e)$, then we have enhanced subgraphs based on $\epsilon(F_{12})$. If $\epsilon(F_{12})=1$, then we have two enhanced spanning subgraphs $(H_1',\epsilon'_1)$ and $(H_2',\epsilon'_2)$ with $H'_1 = H'_2 = H'_e$ and 
	    
	    $$
	      \epsilon'_1(F) = \begin{cases}
	      	\epsilon(F) & F \neq F_1,F_2\\
	      	 1 & F = F_1\\ 
	      	 x & F= F_2
	      \end{cases}
	    $$ 
	    
	    and 
	    
	    $$
	    \epsilon'_2(F) = \begin{cases}
	    	\epsilon(F) & F \neq F_1,F_2\\
	    	x & F = F_1\\ 
	    	1 & F= F_2
	    \end{cases}
	    $$
	    
	    If $\epsilon(F_{12})=x$, then we have exactly one enhanced spanning subgraph $(H'_e,\epsilon')$ with $\epsilon'$ given by,
	    
	    $$
	    \epsilon'(F) = \begin{cases}
	    	\epsilon(F) & F \neq F_1,F_2\\
	    	x & F = F_1,F_2
	    \end{cases}
	    $$
	    
	    Fix an ordering on the set of edges of $G_\Lk$ and for each $(H,\epsilon)$ and $e \in A((H,\epsilon))$, define $w(e) = \#\{f \in E(G_\Lk) \mid f < e \text{ and } f \not\in A((H,\epsilon))\}$. The differential on the Khovanov chain complex is thus defined as follows:
	    
	    $$
	    \begin{aligned}
	    	&\partial_{\Kh}^{i,j} : \CKh^{i,j}(\Lk) \rightarrow \CKh^{i+1,j}(\Lk) \\
	    	& \partial_{\Kh}^{i,j}((H,\epsilon)) = \sum_{e \in A((H,\epsilon))} (-1)^{w(e)}.(H'_e,\epsilon')
	    \end{aligned}
	    $$

  The fact that the above definition of Khovanov homology coincides with the usual definition in terms of resolved states of the link diagram can be seen through the following observation:
  
  \begin{proposition} \label{span graph and res state}
	 There is an one-one correspondence between the sets:
	
       $$
	   \{F(H) \mid H \text{ is a spanning subgraph of } G_\Lk\} \Longleftrightarrow \{\text{ Resolved states of } \Lk \:\}
	$$
  \end{proposition}
		
   \begin{proof}
      For a given spanning subgraph $H \subset G_\Lk$ and $C \in C(H)$, diagramatically $F(C)$ looks like a regular neighborhood of $C$ in the plane. This regular neighborhood is our desired resolved state of $\Lk$, since if we smooth crossings of $\Lk$ in such a way that $C$ is contained inside the black region of $\Lk$, then we obtain this neighborhood. \\
	 Conversely, for a given resolved state $S$ of $\Lk$, we can make a spanning subgraph $H(S)$ of $G_\Lk$ out of $S$. An edge $e \in G_\Lk$ belongs to $H(S)$ if the black regions due to its end vertices are joined in $S$, otherwise we remove it. Thus, for each circle component in $S$, we have a face corresponding to it in $F(H(S))$. Also see Table \ref{res-edge table} which we will be using often to mention the term ``resolution of an edge" (will be denoted by $\Res(e)$) instead of resolution of a crossing.
   \end{proof} 

   \begin{table}[ht]
   \setlength{\tabcolsep}{2mm}
		\def\arraystretch{1.5}
       \centering
       \begin{tabular}{|c|c|c|}
        \hline
          \backslashbox{$e$}{$\Res(e)$} & $A$ & $B$ \\
          \hline
          Positive  & $e \in H$ & $e \not\in H$\\
          \hline
          Negative & $e\not\in H$ & $e \in H$ \\
          \hline   
       \end{tabular}
       \caption{The correspondence between the resolution of a crossing and the presence of an edge in a spanning subgraph of the Tait graph.}
       \label{res-edge table}
   \end{table}
  
  \begin{lemma}\label{kh complex is dis union of twisted unknots}
    Given a connected link diagram $\Lk$ and its Khovanov chain complex $(\CKh(\Lk),\partial_{\Kh})$, consider its Hasse diagram $\mathcal{H}((\CKh(\Lk),\partial_{\Kh}))$. Then we the following equality:
   
   $$
    \mathcal{H}((\CKh(\Lk),\partial_{\Kh})) = \bigsqcup_{T}  \mathcal{H}((\CKh(U(T),\partial_{\Kh}))
   $$
  \end{lemma}

  \begin{proof}
    All the resolved states of $\Lk$ can be obtained by resolving all the nugatory crossings of each twisted unknot and the twisted unknots were obtained by constructing the skein resolution tree $\mathcal{T}$ by resolving crossings of $\Lk$ in a given fixed order. Moreover, $\CKh(U(T)) \cap \CKh(U(T')) = \phi$ for $T \neq T'$ because there exists edges $e,e' \in G_\Lk$ such that $e \in cyc(T,e')$ and $e' \in cyc(T',e)$, now without loss of generality if $e < e'$, then $a_T(e') \in \{d,\Bar{d}\}$ and $a_{T'}(e') \in \{D,\Bar{D}\}$ and these activities have opposite smoothing according to Table \ref{smoothing table}.
  \end{proof}

\end{subsection}

\begin{subsection}{Lee cohomology and \texorpdfstring{$s$}{s}-invariant} 

  Lee \cite{LEE2005554} deformed the differential maps in Khovanov complex to get a new boundary map which did not behave so well in terms of preservation of the $j$--grading. To be precise, the co-chain complex over here, denoted by $\CLee(\Lk)$, is same as the Khovanov complex but the differential maps are defined as follows:\\

   \textbf{case 1:} When two faces merge, the only change in this case is the definition of the map $m$, where $m(1,1)=m(x,x)=1$ and $m(1,x)=m(x,1)=x$.

    \textbf{case 2:} When a face $F$ splits and if $\epsilon(F)=1$ then $\epsilon_1'$ and $\epsilon_2'$ remains the same but when $\epsilon(F)=x$ then we will have two enhanced spanning subgraphs $(H_1',\epsilon_1')$ and $(H_2',\epsilon_2')$, where 

    $$
	    \epsilon'_1(F) = \begin{cases}
	    	\epsilon(F) & F \neq F_1,F_2\\
	    	x & F = F_1,F_2
	    \end{cases}
	    $$

        and 

          \begin{equation}\label{lee diff split}
              \epsilon'_2(F) = \begin{cases}
	    	\epsilon(F) & F \neq F_1,F_2\\
	    	1 & F = F_1,F_2
	    \end{cases}
          \end{equation}

      Observe that, this modified maps increases the $j$--grading by 4, if the $i$--grading increases by 1. This immediately provides a filtration on the Lee complex and as a consequence we have a spectral sequence converging to $\Q \oplus \Q$ as shown in \cite{LEE2005554}. The homology group is fairly simple which is not so interesting but the states which generates this summand are of significant importance as both shown by Lee and Rasmussen \cite{rus}. Lee replaces the generating set $\{1,x\}$ by $\{x+1,x-1\}$ which not only represents the boundary maps in a simplified form but it helps to represent the generators which generates $\Q \oplus \Q$.\\
      Let $\Lk$ be an $n$--component link, then there are $2^n$ possible orientations on $\Lk$. Lee showed that for each orientation there is exactly one generator for $H^*(\CLee(\Lk))$. In fact, if $o$ be an orientation of $\Lk$, and let $D_o$ correspond to the oriented resolution of $\Lk$. We label the circles in $D_o$ with $(x+1)$ and $(x-1)$ according to the following rule:\\
	For each circle $S$ in $D_o$ we assign a mod-2 invariant, which is the mod-2 number of circles in $D_o$ that separates $S$ from infinity. (Draw a ray from $S$ to infinity, count the number of other times it intersects the other circles, mod 2). We add 1 if $S$ has the counterclockwise orientation and add 0 if it has a clockwise orientation. Label $S$ by $(x+1)$ if the resulting mod invariant is 0, and label it by $(x-1)$ if it is 1. The resulting state is denoted by $\mathfrak{s}_o$.\\

      Rasmussen showed that the generators $\mathfrak{s}_o$ for each orientation $o$ of a link $\Lk$ behave well under cobordisms and induces non-trivial maps on Lee cohomology. For a given knot $K$, Rasmussen uses the induced $j-$grading of the generators $\mathfrak{s}_o$ and $\mathfrak{s}_{\Bar{o}}$ to define a knot invariant $s(K)$ which gives a lower bound for the slice genus of $K$

      $$
       |s(K)| \leq 2g_4(K)
      $$
\end{subsection}

\end{section}


\begin{section}{Reduced spanning tree complex}
	\label{reduced}
	 Let $K$ be a connected knot diagram. Then the \textit{reduced} Khovanov homology of $K$ was also defined in \cite{Kh} which is a knot invariant defined similarly where the differential maps remains the same but the generators were special kinds of enhanced spanning subgraphs. In order to define it, first let us mark a fixed arc on $K$ by placing a dot on it. Now we consider all such enhanced spanning subgraphs $(H,\epsilon)$ as the generators of the reduced Khovanov complex of $K$ where, $\epsilon(F) = 1$ if $F$ contains the dotted arc. We denote this complex by $(\CKh^+(K),\partial_{\Kh}^+)$ and its homology group by $\Kh^+(K)$ Similarly, one can define the other version $(\CKh^-(K),\partial_{\Kh}^-)$, where $\epsilon(F)=x$ for $F$ containing the dotted arc and the homology group by $\Kh^-(K)$. \footnote{In fact, $\CKh^-(K)$ is a subcomplex of $\CKh(K)$ while $\CKh^+(K)$ is not.} Note that we are being very specific about the fact that reduced Khovanov homology is a knot invariant because in case of a link $\Lk$ (knot with atleast two components) the reduced Khovanov homology of $\Lk$ differs if we change the position of the dot from one component to other. It only becomes a link invariant if we fix the dotted arc. But this should not be a problem while developing the spanning tree model for a connected link diagram since, we will consider the dotted arc to be fixed on a chosen component.\\

     In order to define the \textit{reduced spanning tree complex} of a given connected link diagram $\Lk$ with a fixed dotted arc, our initial goal would be to provide an acyclic matching on the Hasse diagram $\mathcal{H}((\CKh^+(\Lk), \partial^+_{\Kh}))$ and use algebraic discrete Morse theory to construct a Morse complex consisting of \textit{critical enhanced spanning subgraphs} and we will use the differential map of the Morse complex in order to define the differential map of the reduced spanning tree complex. Now Lemma \ref{kh complex is dis union of twisted unknots} implies that it is enough to provide an acyclic matching on each of the Hasse diagram $\mathcal{H}((\CKh^+(U(T)), \partial_{\Kh}^+))$ for all spanning trees $T \subset G_\Lk$. More specifically, we would like to provide a near-perfect acyclic matching on $\mathcal{H}((\CKh^+(U(T)), \partial_{\Kh}^+))$. The near-perfect condition is essential since then we would have a unique enhanced spanning subgraph, which we denote by $S_c^{T^+}$ representing a unique spanning tree $T$. We then take the union of all these acyclic matchings to obtain an acyclic matching $\Mt^{\Kh}_\Lk$ on $\mathcal{H}((\CKh^+(\Lk), \partial^+_{\Kh}))$.
	
	\begin{definition}\label{reduced spanning tree complex defn}
		Let $\Lk$ be a connected link diagram and $G_\Lk$ be its ordered Tait graph. For any spanning tree $T$ of $G_\Lk$, we define bigradings of $T^+$, $i(T^+)=i_{\Kh}(S_c^{T^+})$ and $j(T^+)=j_{\Kh}(S_c^{T^+})$. Then we have the following bigraded groups,
		$$
		 {\CST^+}^{i,j}(\Lk) = \Z \langle T^+ \subset G_\Lk \: | \: i(T^+)=i, j(T^+)=j \rangle
		$$
		Define $\CST^+(\Lk)=\oplus_{i,j} {\CST^+}^{i,j}(\Lk)$ to be the \textit{reduced spanning tree complex}. Here, $i_{\Kh}$ and $j_{\Kh}$ refers to the bigradings in Khovanov chain complex.
	\end{definition}
	
	We later on define the differential map on $\CST^+(\Lk)$.

\begin{subsection}{Acyclic matching on reduced Khovanov complex}
		
  In order to provide a near-perfect acyclic matching in $\mathcal{H}((\CKh^+(U(T)),\partial^+_{\Kh}))$, we first construct a tree $G(T)$ representing the twisted unknot $U(T)$ in a way such that all possible resolved states of $U(T)$ are in one-one correspondence with all possible spanning subgraphs of $G(T)$. Thus, we will represent each enhanced spanning subgraph of $\CKh^+(U(T))$ as a pair $(H,\epsilon)$ where $H$ is a spanning subgraph of $G(T)$ and $\epsilon : C(H) \rightarrow \{1,x\}$.

\begin{construction}
     Given a spanning tree $T$ of $G_\Lk$, insert blue edges in $T$ for every externally active edges of $T$. Now contract all the internally inactive edges of $T$. Hence, we get a graph with some blue loops together with internally active edges of $T$. Now convert every blue loop into a graph as shown in Figure \ref{blue loop conversion}. After converting every such blue loop, we have the tree $G(T)$, which represents all the live edges of $T$ as its edges and as a result represents the twisted unknot $U(T)$. In order to represent the dotted arc in $U(T)$ coming from $\Lk$ in $G(T)$, choose the vertex in $G(T)$ which corresponds to the circle component which contains the dotted arc after we resolve the crossing of $U(T)$ according to Table \ref{smoothing table}. Consider this vertex as the root of $G(T)$ and label it as $v_d$.
	
	\begin{figure}[ht]
		\centering
		\includesvg[width=6cm]{gt.svg}
		\caption{}
		\label{blue loop conversion}
	\end{figure}
	
\end{construction}


\begin{theorem}\label{matching_theorem}
    Let $T$ be a spanning tree of the Tait graph $G_\Lk$ associated with a connected link diagram $\Lk$, then there exists an acyclic near-perfect matching on $\mathcal{H}((\CKh^+(U(T)),\partial_{\Kh}^+))$, where $(\CKh^+(U(T)),\partial_{\Kh}^+)$ is the reduced Khovanov complex of $U(T)$. Furthermore, by taking the union of these matchings over all spanning trees of $G_\Lk$, we get an acyclic matching $\Mt^{\Kh}_\Lk$ on $\mathcal{H}((\CKh^+(\Lk),\partial_{\Kh}^+))$. 
\end{theorem}

\begin{proof}
	
  We provide the matching on $\mathcal{H}((\CKh^+(U(T)),\partial_{\Kh}^+))$ by applying induction on the number of edges in $G(T)$. Let $n$ be the number of edges in $G(T)$. The base case, $n=1$, is shown Figure \ref{base case matching}.
	
	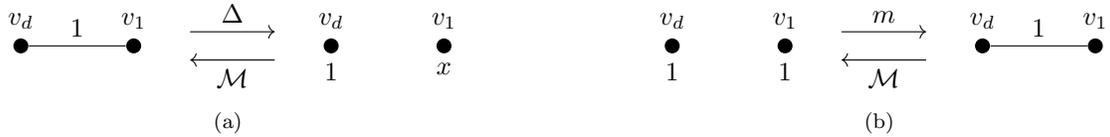
\begin{figure}[ht]
		\centering
		\subfigure[]{
			\begin{tikzpicture}[vertex/.style={circle,fill,inner sep=2pt},scale=0.75]
				\node[vertex] [label=above:$v_d$] at (0,0) (a) {};
				\node[vertex] [label=above:$v_1$] at (2,0) (b) {};
				\draw (a) -- (b) node[midway,above] {$1$};
				\draw[->] (3,0.25) -- (4.5,0.25) node[midway,above] {$\Delta$};
				\draw[<-] (3,-0.25) -- (4.5,-0.25) node[midway,below] {$\Mt$};
				\node[vertex] [label=above:$v_d$,label=below:$1$] at (5.5,0) (c) {};
				\node[vertex] [label=above:$v_1$,label=below:$x$] at (7.5,0) (d) {};
		   \end{tikzpicture}	
			}
		\hspace{2cm}
		\subfigure[]{
			\begin{tikzpicture}[vertex/.style={circle,fill,inner sep=2pt},scale=0.75]
				\node[vertex] [label=above:$v_d$,label=below:$1$] at (0,0) (a) {};
				\node[vertex] [label=above:$v_1$,label=below:$1$] at (2,0) (b) {};
				\draw[->] (3,0.25) -- (4.5,0.25) node[midway,above] {$m$};
				\draw[<-] (3,-0.25) -- (4.5,-0.25) node[midway,below] {$\Mt$};
				\node[vertex] [label=above:$v_d$] at (5.5,0) (c) {};
				\node[vertex] [label=above:$v_1$] at (7.5,0) (d) {};
				\draw (c) -- (d) node[midway,above] {$1$};
			\end{tikzpicture}
		}
		\caption{(a) $e_1$ is a negative twist (b) $e_1$ is a positive twist}
		\label{base case matching}
	\end{figure}
	
	Assume that $G(T)$ is a tree with edges $n > 1$. Let $e_n$ be the least ordered leaf of $G(T)$ and $v_n \neq v_d$ be the vertex adjancent to $e_n$ whose degree is one. Let $\mathcal{B}_n$ be the set of all enhanced spanning subgraphs $(H,\epsilon)$ of $G(T)$, where the connected component $C_{v_n}$ containing $v_n$ has more than one vertex. To define $\Mt_n$, we first pair elements of $\mathcal{B}_n$ based on the type of twist corresponding to edge $e_n$.\\
	
	\textbf{Negative twist}: We pair $(H,\epsilon) \in \mathcal{B}_n$ with $(H',\epsilon')$ where $H'=H - e_n$ and $\epsilon'$ is defined as follows:
	
	$$
	 \epsilon'(C) = \begin{cases}
	 	 \epsilon(C_{v_n}) & \text{ if } C=C_{v_n} - v_n \\
	 	 x & \text{ if } C=\{v_n\} \\
	 	 \epsilon(C) & \text{ if } C \in C(H) \setminus \{C_{v_n}\}
	 \end{cases}
	$$	
	
	\textbf{Positive twist}: We pair $(H,\epsilon) \in \mathcal{B}_n$ with $(H',\epsilon')$ where $H'=H - e_n$ and $\epsilon'$ is defined as follows:
	
	$$
	\epsilon'(C) = \begin{cases}
		\epsilon(C_{v_n}) & \text{ if } C=C_{v_n} - v_n \\
		1 & \text{ if } C=\{v_n\} \\
		\epsilon(C) & \text{ if } C \in C(H) \setminus \{C_{v_n}\}
	\end{cases}
	$$	
	
	The schematic description of $\Mt_n$ for each type of twist is shown in Figure \ref{pos-cross} and \ref{neg-cross}. Now we provide matching on set of the rest of enhanced spanning subgraphs other than in $\mathcal{B}_n$, which we denote by $\mathcal{B}_n'$. We use induction hypothesis to pair elements of $\mathcal{B}_n'$ using the previous matching $\Mt_{n-1}$, the acyclic matching corresponding to the graph $G(T) \setminus v_n$ having $n-1$ edges. The left over unmatched enhanced spanning subgraphs $(H,\epsilon)$ of $\mathcal{B}_n'$ have $v_n$ as an isolated vertex with $\epsilon(\{v_n\}) = 1$ when $e_n$ is a negative twist and $\epsilon(\{v_n\})= x$ when $e_n$ is a positive twist. So, the enhanced spanning subgraphs of $\mathcal{B}_n'$ are in bijection with the enhanced spanning subgraphs of $G(T) \setminus \{v_n\}$. The bijective correspondence is given by the following map:
	
	$$
	 \Psi((H,\epsilon)) = (H \setminus \{v_n\}, \epsilon \mid_{H \setminus \{v_n\}})	
	$$
	
	We now pair of $(H,\epsilon)$ with $(H',\epsilon')$ if and only if $(\Psi(H,\epsilon),\Psi(H',\epsilon')) \in \Mt_{n-1}$ and $\epsilon(\{v_n\})=\epsilon'(\{v_n\})$.

	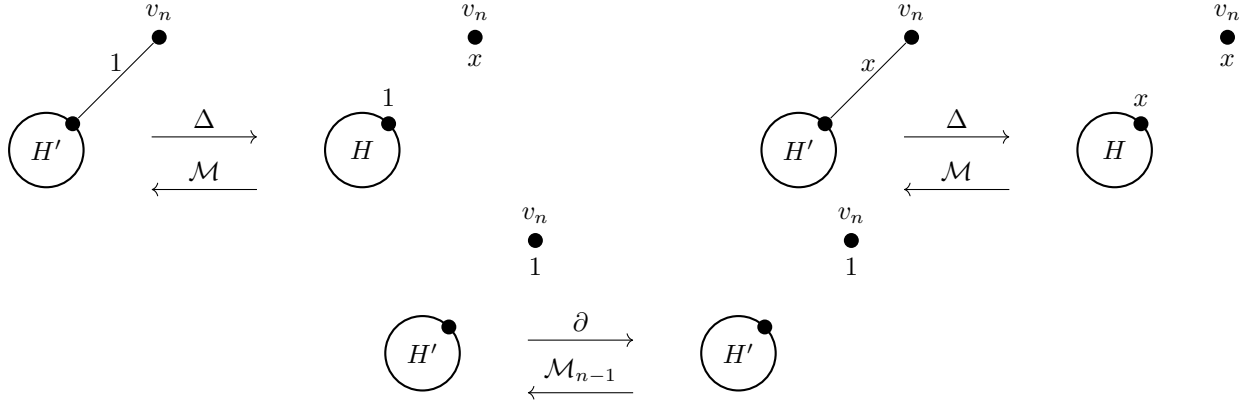
\begin{figure}[ht]
		\centering
		\begin{tikzpicture}[vertex/.style={circle,fill,inner sep=2pt},scale=0.7]
			\node[vertex] at (0.5,0.5) (a) {};
			\node[vertex] [label=above:$v_n$] (b) [above right=of a] {};
			\node [draw,thick,circle through=(a)] at (0,0) {};
			\node at (0,0) {$H'$} (a);
			\draw[->] (2,0.25) -- (4,0.25) node[midway,above] {$\Delta$};
			\draw[<-] (2,-0.75) -- (4,-0.75) node[midway,above] {$\Mt$};
			\node at (6,0) {$H$} (c);
			\node[vertex] [label=above:$1$] at (6.5,0.5) (c) {};
			\node[vertex] [label=above:$v_n$,label=below:$x$] (d) [above right=of c] {};
			\draw (a) -- (b) node[midway,above] {$1$};
			\node [draw,thick,circle through=(c)] at (6,0) {};
		\end{tikzpicture}
		\hfill
		\begin{tikzpicture}[vertex/.style={circle,fill,inner sep=2pt},scale=0.7]
			\node[vertex] at (0.5,0.5) (a) {};
			\node[vertex] [label=above:$v_n$] (b) [above right=of a] {};
			\node [draw,thick,circle through=(a)] at (0,0) {};
			\node at (0,0) {$H'$} (a);
			\draw[->] (2,0.25) -- (4,0.25) node[midway,above] {$\Delta$};
			\draw[<-] (2,-0.75) -- (4,-0.75) node[midway,above] {$\Mt$};
			\node at (6,0) {$H$} (c);
			\node[vertex] [label=above:$x$] at (6.5,0.5) (c) {};
			\node[vertex] [label=above:$v_n$,label=below:$x$] (d) [above right=of c] {};
			\draw (a) -- (b) node[midway,above] {$x$};
			\node [draw,thick,circle through=(c)] at (6,0) {};
		\end{tikzpicture}
		\begin{tikzpicture}[vertex/.style={circle,fill,inner sep=2pt},scale=0.7]
			\node[vertex] at (0.5,0.5) (a) {};
			\node[vertex] [label=below:$1$,label=above:$v_n$] (b) [above right=of a] {};
			\node [draw,thick,circle through=(a)] at (0,0) {};
			\node at (0,0) {$H'$} (a);
			\draw[->] (2,0.25) -- (4,0.25) node[midway,above] {$\partial$};
			\draw[<-] (2,-0.75) -- (4,-0.75) node[midway,above] {$\Mt_{n-1}$};
			\node[vertex] at (6.5,0.5) (c) {};
			\node[vertex] [label=above:$v_n$,label=below:$1$] (d) [above right=of c] {};
			\node [draw,thick,circle through=(c)] at (6,0) {};
			\node at (6,0) {$H'$} (c);
		\end{tikzpicture}
		\caption{$e_n$ is negative twist}
		\label{pos-cross}
	\end{figure}
	
	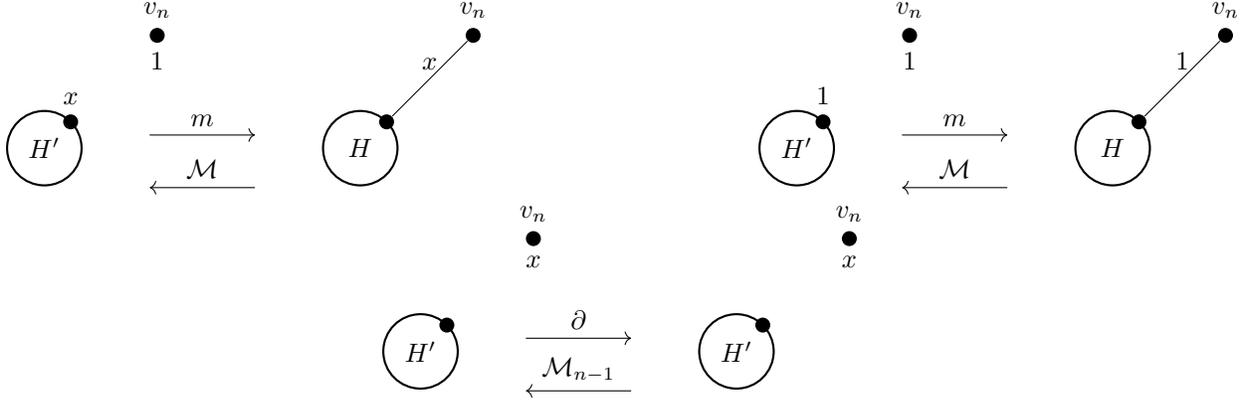
\begin{figure}[ht]
		\centering
		\begin{tikzpicture}[vertex/.style={circle,fill,inner sep=2pt},scale=0.7]
			\node[vertex] [label=above:$x$] at (0.5,0.5) (a) {};
			\node[vertex] [label=below:$1$,label=above:$v_n$] (b) [above right=of a] {};
			\node [draw,thick,circle through=(a)] at (0,0) {};
			\node at (0,0) {$H'$} (a);
			\draw[->] (2,0.25) -- (4,0.25) node[midway,above] {$m$};
			\draw[<-] (2,-0.75) -- (4,-0.75) node[midway,above] {$\Mt$};
			\node[vertex] at (6.5,0.5) (c) {};
			\node[vertex] [label=above:$v_n$] (d) [above right=of c] {};
			\draw (c) -- (d) node[midway,above] {$x$};
			\node [draw,thick,circle through=(c)] at (6,0) {};
			\node at (6,0) {$H$} (c);
		\end{tikzpicture}
		\hfill
		\begin{tikzpicture}[vertex/.style={circle,fill,inner sep=2pt},scale=0.7]
			\node[vertex] [label=above:$1$] at (0.5,0.5) (a) {};
			\node[vertex] [label=below:$1$,label=above:$v_n$] (b) [above right=of a] {};
			\node [draw,thick,circle through=(a)] at (0,0) {};
			\node at (0,0) {$H'$} (a);
			\draw[->] (2,0.25) -- (4,0.25) node[midway,above] {$m$};
			\draw[<-] (2,-0.75) -- (4,-0.75) node[midway,above] {$\Mt$};
			\node[vertex] at (6.5,0.5) (c) {};
			\node[vertex] [label=above:$v_n$] (d) [above right=of c] {};
			\draw (c) -- (d) node[midway,above] {$1$};
			\node [draw,thick,circle through=(c)] at (6,0) {};
			\node at (6,0) {$H$} (c);
		\end{tikzpicture}
		\begin{tikzpicture}[vertex/.style={circle,fill,inner sep=2pt},scale=0.7]
			\node[vertex] at (0.5,0.5) (a) {};
			\node[vertex] [label=below:$x$,label=above:$v_n$] (b) [above right=of a] {};
			\node [draw,thick,circle through=(a)] at (0,0) {};
			\node at (0,0) {$H'$} (a);
			\draw[->] (2,0.25) -- (4,0.25) node[midway,above] {$\partial$};
			\draw[<-] (2,-0.75) -- (4,-0.75) node[midway,above] {$\Mt_{n-1}$};
			\node[vertex] at (6.5,0.5) (c) {};
			\node[vertex] [label=above:$v_n$,label=below:$x$] (d) [above right=of c] {};
			\node [draw,thick,circle through=(c)] at (6,0) {};
			\node at (6,0) {$H'$} (c);
		\end{tikzpicture}
		\caption{$e_n$ is positive twist}
		\label{neg-cross}
	\end{figure} 

   For each $n$, let the set of unmatched enhanced spanning subgraphs be denoted by $\mathcal{B}_n^c$. We call the elements of $\mathcal{B}_n^c$ to be \textit{critical enhanced spanning subgraphs}. For the base case $n=1$, we clearly have exactly one critical enhanced spanning subgraph which is a graph with two isolated vertex $v_d$ and $v_1$ with $\epsilon(\{v_1\})=1$ if $e_1$ is a negative twist and $\epsilon(\{v_1\})=x$ if $e_1$ is a positive twist. Thus, $\Mt_1$ is near-perfect. Suppose by induction hypothesis, assume that $\Mt_{n-1}$ is near-perfect with the unique critical enhanced spanning subgraph $(H^{n-1}_c,\epsilon^{n-1}_c)$. Without loss of generality assume that $e_n$ is a negative twist. Observe that, $\mathcal{B}_n^c$ does not contain elements of $\mathcal{B}_n$ since every enhanced spanning subgraph of $\mathcal{B}_n$ are paired. Thus, $\mathcal{B}_n^c \subset \mathcal{B}_n'$ and since it uses $\Mt_{n-1}$ so by induction hypothesis, we have exactly one critical enhanced spanning subgraph left out by $\Mt_n$ which is $(H^n_c,\epsilon^n_c)$, where $H^n_c = H^{n-1}_c \sqcup v_n$ and 
   
   $$
    \epsilon^n_c(C) = \begin{cases}
    	\epsilon^{n-1}_c(C) & \text{ if } C \neq \{v_n\}\\
    	1 & \text{ if } C=\{v_n\}
    \end{cases}
   $$
   
   Similary, when $e_n$ is a positive twist, then the only difference being $\epsilon^n_c(\{v_n\})=x$. Thus, the matching is near-perfect.\\
   
   Clearly, the matching in case $n=1$ is acyclic. Suppose the matching $\Mt_{n-1}$ is acyclic and in contrary let us assume there is an alternating cycle $\A$ in $\mathcal{H}_\Mt((\CKh^+(U(T)),\partial_{\Kh}^+))$ with atleast one pair of enhanced spanning subgraphs $(H,\epsilon) \leftrightarrows (H',\epsilon')$ in $\A$, where $(H,\epsilon) \in \mathcal{B}_n$. Observe that the critical enhanced state in $B_n^c$ cannot be present in $\A$. Without of loss of generality assume that $e_n$ is a negative twist, then there is only one possible enhanced spanning subgraph $(H'',\epsilon'')$ which can occur in $\A$ before $(H',\epsilon')$ (See Figure \ref{acyclic case}):
   
   \begin{figure}[ht]
   	\centering
   		\begin{tikzpicture}[vertex/.style={circle,fill,inner sep=2pt},scale=0.7]
   			\node[vertex] [label=above:$p$] at (0.5,0.5) (a) {};
   			\node[vertex] [label=above:$v_n$,label=below:$x$] (b) [above right=of a] {};
   			\node [draw,thick,circle through=(a)] at (0,0) {};
   			\node at (0,0) {$H''$} (a);
   			\draw[->] (2,0) -- (4,0) node[midway,above] {$\partial$};
   			\node at (6,0) {$H'$} (c);
   			\node[vertex] [label=above:$p$] at (6.5,0.5) (c) {};
   			\node[vertex] [label=above:$v_n$,label=below:$x$] (d) [above right=of c] {};
   			\node [draw,thick,circle through=(c)] at (6,0) {};
   		\end{tikzpicture}
   		\hfill
   		\begin{tikzpicture}[vertex/.style={circle,fill,inner sep=2pt},scale=0.7]
   			\node[vertex] at (0.5,0.5) (a) {};
   			\node[vertex] [label=above:$v_n$] (b) [above right=of a] {};
   			\node [draw,thick,circle through=(a)] at (0,0) {};
   			\node at (0,0) {$H''$} (a);
   			\draw[->] (2,0) -- (4,0) node[midway,above] {$\partial$};
   			\node at (6,0) {$H'$} (c);
   			\node[vertex] at (6.5,0.5) (c) {};
   			\draw (a)--(b) node[midway,above] {$p$};
   			\draw (c)--(d) node[midway,above] {$p$};
   			\node[vertex] [label=above:$v_n$] (d) [above right=of c] {};
   			\node [draw,thick,circle through=(c)] at (6,0) {};
   		\end{tikzpicture}
   	\caption{ The differential occurs due to a removal or addition of an edge other than $e_n$ for both type of twist of the corresponding crossing due to $e_n$}
   	\label{acyclic case}
   \end{figure}
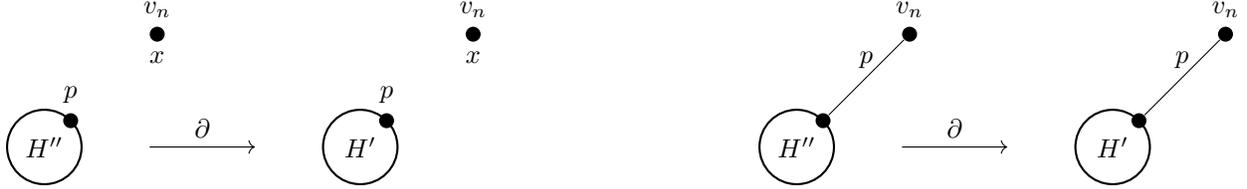
   
   But then there is no enhanced spanning subgraph occuring before $(H'',\epsilon'')$, since we cannot use $\Mt_{n-1}$ to pair $(H'',\epsilon'')$ if $\epsilon''(\{v_n\}) = x$, which implies that for $\A$ to be a cycle, any enhanced spanning subgraph $(H,\epsilon)$ cannot occur in $\A$ where $H$ has the connected component $C_{v_n}$ which means $\A$ can be seen as  a cycle in $\mathcal{H}_\Mt((\CKh^+(U(T)),\partial_{\Kh}^+))$ where all the matched pairs are from $\Mt_{n-1}$ which leads to a contradiction.

\end{proof}

\begin{corollary}\label{determination of aw and res from matching theorem}
	The activity word for a spanning tree $T$ completely determines the critical enhanced spanning subgraph $S_c^{T^+}$ in $\mathcal{H}_\Mt((\CKh^+(U(T)),\partial_{\Kh}^+))$ and the resolution of its corresponding resolved state.
\end{corollary}

\begin{proof}
    We have previously observed in Theorem \ref{matching_theorem} that the critical enhanced spanning subgraph is the unique spanning subgraph $S_c^{T^+}=(H_c,\epsilon_c)$ with $H_c$ being the collection of all isolated vertices of $G(T)$ and
	
    $$
    \epsilon_c(\{v_i\})  = \begin{cases}
    	1 & \text{ if } e_i \text{ is a negative twist}\\
    	x & \text{ if } e_i \text{ is a positive twist}
    \end{cases}
    $$
    
    Now a negative twisted edge in $G(T)$ denotes a $L$ and $\Bar{l}$ in $W(T)$, while a positive twisted edge denotes $\Bar{L}$ and $l$ in $W(T)$. To obtain the resolution, we replace $W(T)$ according to the Table \ref{smoothing table}. 
    
	This table is justified by the fact that for an internally positive live edge, $L$, we apply $B-$smoothing to increase the number of disconnected components and the same reasoning applies for the rest of the edges which as a result gives $H_c$.
\end{proof}

\begin{defin}
    Given a spanning tree $T$, the above matching in $\mathcal{H}_\Mt((\CKh^+(U(T)),\partial_{\Kh}^+))$ can be represented in the form of a word by concatenating the edges of $G(T)$ in order of the inductive matching sequence in Theorem \ref{matching_theorem} and suffixing each live edge with its corresponding order. We call this the \textit{matching word} for $T$ denoted by $M(T)$.  
\end{defin}

    \tikzset{diag1/.pic={
       \node[vertex] at (0,0) (a) {};
       \node[vertex] at (2,0) (b) {};
       \node[vertex,fill=red,color=red] at (1,1) (c) {};
       \draw (a)--(c);
       \draw (b)--(c);
       \node [label=below:$1$] at (1,0.5) (d) {};
       \node [draw, color=gray, thick, shape=rectangle, minimum width=2.5cm, minimum height=1.5cm, anchor=center] at (1,0.5) {};
    }}

    \tikzset{diag2/.pic={
       \node[vertex] [label=above:$1$] at (0,0) (a) {};
       \node[vertex] at (2,0) (b) {};
       \node[vertex,fill=red,color=red] at (1,1) (c) {};
       \draw (b)--(c) node[midway,right] {$1$};
       \node [draw, color=gray, thick, shape=rectangle, minimum width=2.5cm, minimum height=1.5cm, anchor=center] at (1,0.5) {};
    }}

    \tikzset{diag3/.pic={
       \node[vertex] [label=above:$x$] at (0,0) (a) {};
       \node[vertex] at (2,0) (b) {};
       \node[vertex,fill=red,color=red] at (1,1) (c) {};
       \draw (b)--(c) node[midway,right] {$1$};
       \node [draw, color=gray, thick, shape=rectangle, minimum width=2.5cm, minimum height=1.5cm, anchor=center] at (1,0.5) {};
    }}

    \tikzset{diag4/.pic={
       \node[vertex] at (0,0) (a) {};
       \node[vertex] [label=above:$1$] at (2,0) (b) {};
       \node[vertex,fill=red,color=red] at (1,1) (c) {};
       \draw (a)--(c) node[midway,left] {$1$};
       \node [draw, color=gray, thick, shape=rectangle, minimum width=2.5cm, minimum height=1.5cm, anchor=center] at (1,0.5) {};
    }}

    \tikzset{diag5/.pic={
       \node[vertex] at (0,0) (a) {};
       \node[vertex] [label=above:$x$] at (2,0) (b) {};
       \node[vertex,fill=red,color=red] at (1,1) (c) {};
       \draw (a)--(c) node[midway,left] {$1$};
       \node [draw, color=gray, thick, shape=rectangle, minimum width=2.5cm, minimum height=1.5cm, anchor=center] at (1,0.5) {};
    }}
    
    \tikzset{diag6/.pic={
       \node[vertex] [label=above:$1$] at (0,0) (a) {};
       \node[vertex] [label=above:$1$] at (2,0) (b) {};
       \node[vertex,fill=red,color=red] [label=above:$1$] at (1,1) (c) {};
       \node [draw, color=gray, thick, shape=rectangle, minimum width=2.5cm, minimum height=2cm, anchor=center] at (1,0.7) {};
    }}

    \tikzset{diag7/.pic={
       \node[vertex] [label=above:$1$] at (0,0) (a) {};
       \node[vertex] [label=above:$x$] at (2,0) (b) {};
       \node[vertex,fill=red,color=red] [label=above:$1$] at (1,1) (c) {};
       \node [draw, color=gray, thick, shape=rectangle, minimum width=2.5cm, minimum height=2cm, anchor=center] at (1,0.7) {};
    }}

    \tikzset{diag8/.pic={
       \node[vertex] [label=above:$x$] at (0,0) (a) {};
       \node[vertex] [label=above:$1$] at (2,0) (b) {};
       \node[vertex,fill=red,color=red] [label=above:$1$] at (1,1) (c) {};
       \node [draw, color=gray, thick, shape=rectangle, minimum width=2.5cm, minimum height=2cm, anchor=center] at (1,0.7) {};
    }}

    \tikzset{diag9/.pic={
       \node[vertex] [label=above:$x$] at (0,0) (a) {};
       \node[vertex] [label=above:$x$] at (2,0) (b) {};
       \node[vertex,fill=red,color=red] [label=above:$1$] at (1,1) (c) {};
       \node [draw, color=gray, thick, shape=rectangle, minimum width=2.5cm, minimum height=2cm, anchor=center] at (1,0.7) {};
    }}
    
    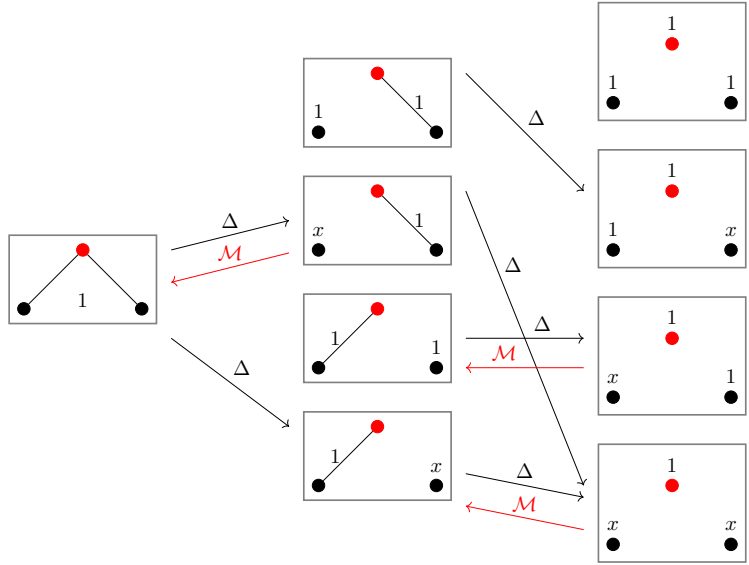
\begin{figure}[ht]
        \centering
        \resizebox{0.6\textwidth}{!}{
        \begin{tikzpicture}[vertex/.style={circle, draw, inner sep=0pt, minimum size=6pt, fill=black}]
            \pic (A) at (0,0) {diag1};
            \pic (B) at (5,3) {diag2};
            \pic (C) at (5,1) {diag3};
            \pic (D) at (5,-1) {diag4};
            \pic (E) at (5,-3) {diag5};
            \pic (F) at (10,3.5) {diag6};
            \pic (G) at (10,1) {diag7};
            \pic (H) at (10,-1.5) {diag8};
            \pic (I) at (10,-4) {diag9};
            \draw[->] (2.5,1)--(4.5,1.5) node[midway,above] {$\Delta$};
            \draw[<-,red] (2.5,0.45)--(4.5,0.95) node[midway,above] {$\Mt$};
            \draw[->] (2.5,-0.5)--(4.5,-2) node[midway,above,xshift=0.2cm] {$\Delta$};
            \draw[->] (7.5,4)--(9.5,2) node[midway,above,xshift=0.2cm] {$\Delta$};
            \draw[->] (7.5,2)--(9.5,-3) node[midway,above,xshift=-0.2cm,yshift=1cm] {$\Delta$};
            \draw[->] (7.5,-0.5)--(9.5,-0.5) node[midway,above,xshift=0.3cm] {$\Delta$};
            \draw[<-,red] (7.5,-1)--(9.5,-1) node[midway,above,xshift=-0.35cm] {$\Mt$};
            \draw[->] (7.5,-2.8)--(9.5,-3.2) node[midway,above] {$\Delta$};
            \draw[<-,red] (7.5,-3.35)--(9.5,-3.75) node[midway,above] {$\Mt$};
        \end{tikzpicture}}
    \caption{$\mathcal{H}_\Mt((\CKh^+(U(T_2)),\partial_{\Kh}^+))$ with $M(T_2)=L_2L_1$, where $T_2$ is from Table \ref{Three trefoil table}}
    \label{hasse diagram example}
\end{figure}

\end{subsection}

\begin{subsection}{Reduced spanning tree complex}
	Given a connected link diagram $\Lk$, we thus have an acyclic matching due to Theorem \ref{matching_theorem} on $\mathcal{H}((\CKh^+(U(T)),\partial_{\Kh}^+))$ which gives us the following Morse complex:
    
	$$
	   {C^\Mt}^+(\Lk) = \bigoplus_{i,j} {({C^\Mt}^+)}^{i,j}(\Lk)
	$$
    
   where, ${({C^\Mt}^+)}^{i,j}(\Lk) = \Z \langle (H_c,\epsilon_c) \in \CKh^+(\Lk) \: \mid \: i_{\Kh}((H_c,\epsilon_c)) = i, j_{\Kh}((H_c,\epsilon_c)) = j \rangle$ and the differential on ${C^\Mt}^+(\Lk)$ is defined as follows:
	
	$$
	\begin{aligned}
		&\partial_\Mt^+ : {({C^\Mt}^+)}^{i,j}(\Lk^+) \rightarrow {({C^\Mt}^+)}^{i+1,j}(\Lk^+) \\
		& (H_c,\epsilon_c) \mapsto \sum_{i_{\Kh}((H_c,\epsilon_c))=i+1} \Gamma((H_c,\epsilon_c),(H_c',\epsilon_c')).(H_c',\epsilon_c') 
	\end{aligned}
	$$
	
   where, $\Gamma((H_c,\epsilon_c),(H_c',\epsilon_c'))$ is the sum of weights of all alternating paths between $(H_c,\epsilon_c)$ and $(H_c',\epsilon_c')$ in $\mathcal{H}_\Mt((\CKh^+(U(T)),\partial_{\Kh}^+))$.\\
   Now each spanning tree of $G_\Lk$ can be represented by the unique critical enhanced spanning subgraph $(H_c,\epsilon_c)$ and hence, we have a $\Z-$module isomorphism of complexes between ${C^\Mt}^+(\Lk)$ and $\CST^+(\Lk)$. Thus, we now have the definition of $\partial_{ST}^+$ which was previously mentioned in the definition \ref{reduced spanning tree complex defn}.
	
	\begin{equation}\label{reduced differential}
		\begin{aligned}
		   &  \partial_{ST}^+: {\CST^+}^{i,j}(\Lk) \rightarrow {\CST^+}^{i+1,j} (\Lk) \\
		   & T^+ \mapsto \sum_{i({T'}^+)=i+1}  \Gamma(T^+,{T'}^+).{T'}^+ 
		\end{aligned}
	\end{equation}
	
   where, the incidence number $\Gamma(T^+,{T'}^+)=\Gamma((H_c,\epsilon_c),(H_c',\epsilon_c'))$.
	
 \begin{proof}[Proof of Theorem \ref{maintheorem}]
     Theorem \ref{dmt maintheorem} and the definition of $\partial^+_{ST}$ in \ref{reduced differential} implies the following:
		$$
		   H^*(({\CST^+}(\Lk),\partial_{ST}^+)) \cong \Kh^+(\Lk)
		$$
 \end{proof}
	
Now our goal will be to provide a detailed combinatorial description of $\Gamma(T^+,{T'}^+)$ in terms of the activity word of the spanning trees of $G_\Lk$.  
	
\end{subsection}

\begin{subsection}{Alternating paths between critical enhanced subgraphs}\label{alternating path count subsection}
	
   An alternating path from $(H_c,\epsilon_c)$ to $(H_c',\epsilon_c')$ contains a collection of enhanced spanning subgraphs which belongs to a set of complexes $\{\CKh^+(U(T_i))\}_{i=1}^n$, where $T_1^+ = T^+$ and $T_n^+ = T'^+$. Now any alternating path in $\mathcal{H}_{\Mt}((\CKh^+(\Lk),\partial_{\Kh}^+))$ consists of non-zero $m$ and $\Delta$ maps defined in the Khovanov complex and as a result the collection of these trees $\{T_i^+\}$ follows the partial order defined in \ref{partial order on spanning trees defn}. In other words,
	
\begin{proposition}\label{partial order respects differential}
     The differential $\partial_{\Kh}^+$ in the reduced Khovanov complex $\CKh^+(\Lk)$ respects the partial order in Definition \ref{partial order on spanning trees defn} as follows:
		
	\begin{enumerate}
           \item If $[\partial_{\Kh}^+(x),y] \neq 0$ for any $x \in \CKh^+(U(T_1))$ and $y \in \CKh^+(U(T_2))$, then $T_1 > T_2$.
			
            \item If $T_1$ and $T_2$ are not comparable or $T_2 > T_1$, then $[\partial_{\Kh}^+(x),y] = 0$ for all $x \in \CKh^+(U(T_1))$ and $y \in \CKh^+(U(T_2))$.
	\end{enumerate}	
\end{proposition}	
	
\begin{proof}
   If $[\partial_{\Kh}^+(x),y] \neq 0$ then there is exactly one edge in $x$ with an $A$-marker which has changed to $B$-marker in $y$ and thus the partial smoothing that contains these states satisfies $(x_1,\cdots,x_n) > (y_1,\cdots,y_n)$. The second statement follows from the definition of $\partial_{\Kh}^+$.
\end{proof}

    First let us look at some properties of alternating paths within the Hasse diagram $\mathcal{H}_\Mt((\CKh^+(U(T)),\partial_{\Kh}^+))$ for a fixed spanning tree $T \subset G_\Lk$. From now on, we will use the notation $S_x^{T^+}$ in place of $(H_x,\epsilon_x)$ whenever we refer to any element of $\CKh^+(U(T))$.
   
   \begin{definition}
   	Given a spanning tree $T$, consider any two enhanced spanning subgraphs $S_x^{T^+}$ and $S_y^{T^+}$ in $\CKh^+(U(T))$. Then we define the following:
   	\begin{enumerate}
   		\item $\widetilde{\A^{\downarrow}}(S_x^{T^+},S_y^{T^+})$ := \# alternating paths from $S_x^{T^+}$ to $S_y^{T^+}$ in $\mathcal{H}_\Mt ((\CKh^+(U(T)),\partial_{\Kh}^+))$, where the paths begin with a non-matched edge.
   		
   		\item $\widetilde{\A^{\uparrow}}(S_x^{T^+},S_y^{T^+})$ := \# alternating paths from $S_x^{T^+}$ to $S_y^{T^+}$ in $\mathcal{H}_\Mt ((\CKh^+(U(T)),\partial_{\Kh}^+))$, where the paths begin with a matched edge.
   	\end{enumerate}
   \end{definition}
   
   \begin{proposition}
   	 $\widetilde{\A^{\downarrow}}(S_x^{T^+},S_y^{T^+}),\ \widetilde{\A^{\uparrow}}(S_x^{T^+},S_y^{T^+}) \in \{0,1\}$ for all $S_x^{T^+}, S_y^{T^+} \in \CKh^+(U(T))$.
   \end{proposition}
   
   \begin{proof}
      Assuming the statement is true for $\widetilde{\A^\downarrow}(S_x^{T^+},S_y^{T^+})$, one can easily show that it is also true for $\widetilde{\A^\uparrow}(S_x^{T^+},S_y^{T^+})$ since, there is an unique enhanced spanning subgraph $\partial_{\Kh}^+(S_x^{T^+})$ which is matched with $S_x^{T^+}$ and thus, $\widetilde{\A^\uparrow}(S_x^{T^+},S_y^{T^+}) = \widetilde{\A^\downarrow}(\partial_{\Kh}^+(S_x^{T^+}),S_y^{T^+})$.\\
   	 We prove our assumption by applying induction on the number of edges of $G(T)$. For $|E(G(T))|=1$, the statement is certainly true from Figure \ref{base case matching}. Suppose by induction hypothesis the statement is true for $|E(G(T))| < n$, then for the sake of contradiction, assume there are atleast two alternating paths $\Pt$ and $\Pt'$ from $S_x^{T^+}$ to $S_y^{T^+}$ and without loss of generality, assume $\Pt$ contains a directed edge $S_z^{T^+} \leftrightarrows \partial S_z^{T^+}$ which is matched due to the edge $e_n$, which occurs when defining $\Mt_n$ in the proof of Theorem \ref{matching_theorem}. Such a directed edge occurs due to the induction hypothesis. Now if $e_n$ is a positive twist then either $S_z^{T^+}=S_y^{T^+}$ or $S_y^{T^+}$ is the next enhanced spanning subgraph occuring after $S_z^{T^+}$ in $\Pt$. Now if $S_z^{T^+}=S_y^{T^+}$ then, either the $S_x^{T^+}$ occurs before $\partial S_z^{T^+}$ in which case $\Pt = \Pt'$ or there are two different alternating paths from $S_x^{T^+}$ to $\partial S_z^{T^+}$ which entirely consists of directed matching edges of $\Mt_{n-1}$ and thus, we have a contradiction. If $S_z^{T^+} \neq S_y^{T^+}$ then either the directed edge $S_z^{T^+} \leftrightarrows \partial S_z^{T^+}$ belongs to $\Pt'$ in which case $\Pt=\Pt'$ by a similar argument as discussed earlier or else $S_x^{T^+}=S_z^{T^+}$ which again implies $\Pt=\Pt'$. While if $e_n$ is a negative twist, then $\partial S_z^{T^+}$ is the enhanced spanning subgraph occuring just after $S_x^{T^+}$ in $\Pt$ and the rest of the argument follows similarly.
   \end{proof}

  Now since we want to count alternating paths between $S_c^{T^+}$ and $S_x^{T'^+}$, we first characterize  $S_x^{T^+} \in \CKh^+(U(T))$ and $S_y^{T^+} \in \CKh^+(U(T'))$ for which $\widetilde{\A^{\downarrow}}(S_c^{T^+},S_x^{T^+})=1$ and $\widetilde{\A^{\uparrow}}(S_y^{T^+},S_c^{T'^+})=1$.\\ 
  
  For a given state $S_x^{T^+}$, Let $(H_x \cap H_c, \epsilon_x \cap \epsilon_c)$ be the enhanced subgraph with $H_x \cap H_c$ be the subgraph with isolated vertices such that $C(H_x) = C(H_c)$ and $\epsilon_x \cap \epsilon_c : H_x \cap H_c \rightarrow \{1,x\}$ with $(\epsilon_x \cap \epsilon_c) (C) = \epsilon_x(C) = \epsilon_c(C)$. Now define $G(H_x)$ to be the maximal subgraph of $G(T)$ such that $V(H_x \cap H_c) = V(G(H_x))$.\\
  
  Let $\{e_1,e_2,\cdots,e_n\}$ be the set of leaves of $G(T)$ and let $P_{e_i}$ be the unique path in $G(T)$ from $v_d$ to $e_i$. Then, we have the following two definitions:

  \begin{definition}
  	A \textit{negative subpath} of $G(H_x)$ is a subgraph of $G(H_x)$ where each connected component is contained in some $P_{e_i}$, whose all edges have activity $\Bar{L}$ or $l$ and the initial (least distant from $v_d$ in $G(T)$) vertex $v_C$ of each component $C$ is not $v_d$, together with $\epsilon_x (\{v_C\}) = 1$.
  \end{definition}

  \begin{definition}
  	A \textit{positive subpath} of $G(H_x)$ is a subgraph of $G(H_x)$ where each connected component is contained in some $P_{e_i}$, whose all edges have activity $L$ or $\Bar{l}$ and the initial vertex $v_C$ of each component $C$ is not $v_d$, together with $\epsilon_x (\{v_C\}) = x$.
  \end{definition}

   \begin{proposition}\label{down path within tree}
   	For a given spanning tree $T$ of $G_\Lk$, the following two sets are in one-one correspondence:
   	
   	\[\begin{tikzcd}[ampersand replacement=\&,cramped]
   		{\left\{S_x^{T^+} \mid \widetilde{\A^\downarrow}(S_c^{T^+},S_x^{T^+})=1, i(S_c^{T^+})=i(S_x^{T^+})\right\}} \&\& {\left\{\text{Negative subpaths of }G(H_c)=G(T)\right\}}
   		\arrow[shorten <=4pt, shorten >=4pt, tail reversed, from=1-1, to=1-3]
   	\end{tikzcd}\]
   	
   \end{proposition}
   
   \begin{proof}
   	For a given negative subpath $H$, each connected component of $C \in C(H)$ can be observed as a subsequence of $M(T)$ and thus can be organized in the inductive matching order of $T$ as $C_1 C_2\cdots C_k$. We start with the initial component $C_1$ from $S_c^{T^+}$, where the initial vertex $v_{C_1}$ in $C_1$ has $\epsilon(\{v_{C_1}\})=1$ and rest vertices have $x$ as their enhancement in $H_c$. We can find an alternating path entirely consisting of merge maps where we end up with a state $S_x^{T^+}$ with $H_x = H_c$ and $\epsilon_x(\{v_{C_i}\})= \epsilon_x(\{v\}) = x$ for all $v \in V(C_i)$ except for the last vertex of each $C_i$ having enhancement $1$ (See Figure \ref{negative subpath example}). For the forward direction of the correspondence, we observe that any alternating path starting from $S_c^{T^+}$ has to begin with an edge with activity either $\Bar{L}$ or $l$ and has to have the previous edge as $L$ or $\Bar{l}$. To begin such a path, we choose any edge with activity $L$ or $\Bar{l}$ and continue until we hit another edge with activity $L$ or $\Bar{l}$ following the order in $M(T)$. We repeat this process until we reach the end of $M(T)$ in order to look for such an edge. This provides us with a negative subpath. 
   \end{proof}   

   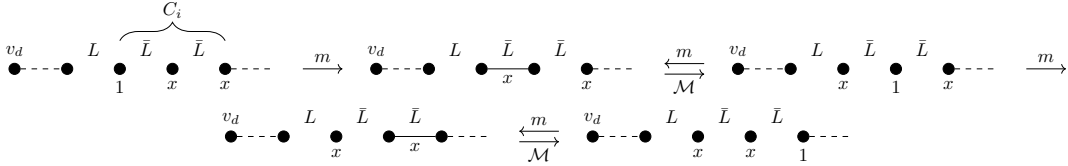
\begin{figure}[ht]
       \centering
       \scalebox{0.7}{
       \begin{tikzpicture}[vertex/.style={circle, draw, inner sep=0pt, minimum size=6pt, fill=black},baseline]
       \node[vertex] [label=above:$v_d$] at (0,0) (r) {};
       \node[vertex] at (1,0) (a) {};
       \node[vertex] [label=below:$1$] at (2,0) (b) {};
       \node[vertex] [label=below:$x$] at (3,0) (c) {};
       \node[vertex] [label=below:$x$] at (4,0) (d) {};
       \node at (5,0) (e) {};
       \draw[dashed] (r)--(a);
       \draw[dashed] (d)--(e);
       \node [label=above:$L$] at (1.5,0) {}; 
       \node [label=above:$\Bar{L}$] at (2.5,0) {}; 
       \node [label=above:$\Bar{L}$] at (3.5,0) {}; 
       \draw[decorate,decoration={brace,amplitude=10pt,raise=12pt}]
        (b.north) -- (d.north) node[midway,yshift=1cm]{$C_i$};
       \end{tikzpicture}
       \begin{tikzpicture}[baseline]
           \node at (0,0) (a) {};
           \node at (1,0) (b) {};
           \draw[->] (a)--(b) node[midway,above] {$m$};
       \end{tikzpicture}
       \begin{tikzpicture}[vertex/.style={circle, draw, inner sep=0pt, minimum size=6pt, fill=black},baseline]
       \node[vertex] [label=above:$v_d$] at (0,0) (r) {};
       \node[vertex] at (1,0) (a) {};
       \node[vertex] at (2,0) (b) {};
       \node[vertex] at (3,0) (c) {};
       \node[vertex] [label=below:$x$] at (4,0) (d) {};
       \node at (5,0) (e) {};
       \draw[dashed] (r)--(a);
       \draw[dashed] (d)--(e);
       \draw (b)--(c) node[midway,below] {$x$};
       \node [label=above:$L$] at (1.5,0) {}; 
       \node [label=above:$\Bar{L}$] at (2.5,0) {}; 
       \node [label=above:$\Bar{L}$] at (3.5,0) {}; 
       \end{tikzpicture}
       \begin{tikzpicture}[baseline]
           \node at (0,0.1) (a) {};
           \node at (1,0.1) (b) {};
           \node at (0,-0.1) (c) {};
           \node at (1,-0.1) (d) {};
           \draw[<-] (a)--(b) node[midway,above] {$m$};
           \draw[->] (c)--(d) node[midway,below] {$\Mt$};
       \end{tikzpicture}
       \begin{tikzpicture}[vertex/.style={circle, draw, inner sep=0pt, minimum size=6pt, fill=black},baseline]
       \node[vertex] [label=above:$v_d$] at (0,0) (r) {};
       \node[vertex] at (1,0) (a) {};
       \node[vertex] [label=below:$x$] at (2,0) (b) {};
       \node[vertex] [label=below:$1$] at (3,0) (c) {};
       \node[vertex] [label=below:$x$] at (4,0) (d) {};
       \node at (5,0) (e) {};
       \draw[dashed] (r)--(a);
       \draw[dashed] (d)--(e);
       \node [label=above:$L$] at (1.5,0) {}; 
       \node [label=above:$\Bar{L}$] at (2.5,0) {}; 
       \node [label=above:$\Bar{L}$] at (3.5,0) {}; 
       \end{tikzpicture}
       \begin{tikzpicture}[baseline]
           \node at (0,0) (a) {};
           \node at (1,0) (b) {};
           \draw[->] (a)--(b) node[midway,above] {$m$};
       \end{tikzpicture}}

       \scalebox{0.7}{
       \begin{tikzpicture}[vertex/.style={circle, draw, inner sep=0pt, minimum size=6pt, fill=black},baseline]
       \node[vertex] [label=above:$v_d$] at (0,0) (r) {};
       \node[vertex] at (1,0) (a) {};
       \node[vertex] [label=below:$x$] at (2,0) (b) {};
       \node[vertex] at (3,0) (c) {};
       \node[vertex] at (4,0) (d) {};
       \node at (5,0) (e) {};
       \draw[dashed] (r)--(a);
       \draw[dashed] (d)--(e);
       \draw (c)--(d) node[midway,below] {$x$};
       \node [label=above:$L$] at (1.5,0) {}; 
       \node [label=above:$\Bar{L}$] at (2.5,0) {}; 
       \node [label=above:$\Bar{L}$] at (3.5,0) {}; 
       \end{tikzpicture}
       \begin{tikzpicture}[baseline]
           \node at (0,0.1) (a) {};
           \node at (1,0.1) (b) {};
           \node at (0,-0.1) (c) {};
           \node at (1,-0.1) (d) {};
           \draw[<-] (a)--(b) node[midway,above] {$m$};
           \draw[->] (c)--(d) node[midway,below] {$\Mt$};
       \end{tikzpicture}
       \begin{tikzpicture}[vertex/.style={circle, draw, inner sep=0pt, minimum size=6pt, fill=black},baseline]
       \node[vertex] [label=above:$v_d$] at (0,0) (r) {};
       \node[vertex] at (1,0) (a) {};
       \node[vertex] [label=below:$x$] at (2,0) (b) {};
       \node[vertex] [label=below:$x$] at (3,0) (c) {};
       \node[vertex] [label=below:$1$] at (4,0) (d) {};
       \node at (5,0) (e) {};
       \draw[dashed] (r)--(a);
       \draw[dashed] (d)--(e);
       \node [label=above:$L$] at (1.5,0) {}; 
       \node [label=above:$\Bar{L}$] at (2.5,0) {}; 
       \node [label=above:$\Bar{L}$] at (3.5,0) {}; 
       \end{tikzpicture}}
       \caption{A negative subpath and its corresponding alternating path.}
       \label{negative subpath example}
   \end{figure}
   
   \begin{proposition}\label{up path within tree}
   	 For a given spanning tree $T$ of $G_\Lk$, the following two sets are in one-one correspondance:
   	 
   	 \[\begin{tikzcd}[ampersand replacement=\&,cramped]
   	 	{\left\{S_x^{T^+} \mid \widetilde{\A^\uparrow}(S_x^{T^+},S_c^{T^+})=1, i(S_c^{T^+})=i(S_x^{T^+})\right\}} \&\& {\left\{\text{Positive subpaths of }G(H_c)=G(T)\right\}}
   	 	\arrow[shorten <=4pt, shorten >=4pt, tail reversed, from=1-1, to=1-3]
   	 \end{tikzcd}\]
   \end{proposition}  
   
   \begin{proof}
   	The proof follows in similar manner. For a positive subpath, we move from $S_x^{T^+}$ to $S_c^{T^+}$ using an alternating path consisting entirely of split maps since it can only occur for an edge with activity $L$ or $\Bar{l}$ and can continue until we have another edge with those same activities (See Figure \ref{positive subpath example}). The only difference here being that the subsequence for a positive subpath in $M(T)$ runs from right to left in contrast to running from left to right for a negative subpath. 
   \end{proof}

   \begin{figure}[ht]
       \centering
       \scalebox{0.7}{
       \begin{tikzpicture}[vertex/.style={circle, draw, inner sep=0pt, minimum size=6pt, fill=black},baseline]
       \node[vertex] [label=above:$v_d$] at (0,0) (r) {};
       \node[vertex] at (1,0) (a) {};
       \node[vertex] [label=below:$1$] at (2,0) (b) {};
       \node[vertex] [label=below:$1$] at (3,0) (c) {};
       \node[vertex] [label=below:$x$] at (4,0) (d) {};
       \node at (5,0) (e) {};
       \draw[dashed] (r)--(a);
       \draw[dashed] (d)--(e);
       \node [label=above:$\Bar{L}$] at (1.5,0) {}; 
       \node [label=above:$L$] at (2.5,0) {}; 
       \node [label=above:$L$] at (3.5,0) {};
       \draw[decorate,decoration={brace,amplitude=10pt,raise=12pt}]
        (b.north) -- (d.north) node[midway,yshift=1cm]{$C_i$};
       \end{tikzpicture}
       \begin{tikzpicture}[baseline]
           \node at (0,0.1) (a) {};
           \node at (1,0.1) (b) {};
           \node at (0,-0.1) (c) {};
           \node at (1,-0.1) (d) {};
           \draw[<-] (a)--(b) node[midway,above] {$\Delta$};
           \draw[->] (c)--(d) node[midway,below] {$\Mt$};
       \end{tikzpicture}
       \begin{tikzpicture}[vertex/.style={circle, draw, inner sep=0pt, minimum size=6pt, fill=black},baseline]
       \node[vertex] [label=above:$v_d$] at (0,0) (r) {};
       \node[vertex] at (1,0) (a) {};
       \node[vertex] [label=below:$1$] at (2,0) (b) {};
       \node[vertex] at (3,0) (c) {};
       \node[vertex] at (4,0) (d) {};
       \node at (5,0) (e) {};
       \draw[dashed] (r)--(a);
       \draw[dashed] (d)--(e);
       \draw (c)--(d) node[midway,below] {$1$};
       \node [label=above:$\Bar{L}$] at (1.5,0) {}; 
       \node [label=above:$L$] at (2.5,0) {}; 
       \node [label=above:$L$] at (3.5,0) {}; 
       \end{tikzpicture}
       \begin{tikzpicture}[baseline]
           \node at (0,0.1) (a) {};
           \node at (1,0.1) (b) {};
           \draw[->] (a)--(b) node[midway,above] {$\Delta$};
       \end{tikzpicture}
       \begin{tikzpicture}[vertex/.style={circle, draw, inner sep=0pt, minimum size=6pt, fill=black},baseline]
       \node[vertex] [label=above:$v_d$] at (0,0) (r) {};
       \node[vertex] at (1,0) (a) {};
       \node[vertex] [label=below:$1$] at (2,0) (b) {};
       \node[vertex] [label=below:$x$] at (3,0) (c) {};
       \node[vertex] [label=below:$1$] at (4,0) (d) {};
       \node at (5,0) (e) {};
       \draw[dashed] (r)--(a);
       \draw[dashed] (d)--(e);
       \node [label=above:$\Bar{L}$] at (1.5,0) {}; 
       \node [label=above:$L$] at (2.5,0) {}; 
       \node [label=above:$L$] at (3.5,0) {}; 
       \end{tikzpicture}
        \begin{tikzpicture}[baseline]
           \node at (0,0.1) (a) {};
           \node at (1,0.1) (b) {};
           \node at (0,-0.1) (c) {};
           \node at (1,-0.1) (d) {};
           \draw[<-] (a)--(b) node[midway,above] {$\Delta$};
           \draw[->] (c)--(d) node[midway,below] {$\Mt$};
       \end{tikzpicture}}

       \scalebox{0.7}{
       \begin{tikzpicture}[vertex/.style={circle, draw, inner sep=0pt, minimum size=6pt, fill=black},baseline]
       \node[vertex] [label=above:$v_d$] at (0,0) (r) {};
       \node[vertex] at (1,0) (a) {};
       \node[vertex] at (2,0) (b) {};
       \node[vertex] at (3,0) (c) {};
       \node[vertex] [label=below:$1$] at (4,0) (d) {};
       \node at (5,0) (e) {};
       \draw[dashed] (r)--(a);
       \draw[dashed] (d)--(e);
       \draw (b)--(c) node[midway,below] {$1$};
       \node [label=above:$\Bar{L}$] at (1.5,0) {}; 
       \node [label=above:$L$] at (2.5,0) {}; 
       \node [label=above:$L$] at (3.5,0) {}; 
       \end{tikzpicture}
       \begin{tikzpicture}[baseline]
           \node at (0,0.1) (a) {};
           \node at (1,0.1) (b) {};
           \draw[->] (a)--(b) node[midway,above] {$\Delta$};
       \end{tikzpicture}
       \begin{tikzpicture}[vertex/.style={circle, draw, inner sep=0pt, minimum size=6pt, fill=black},baseline]
       \node[vertex] [label=above:$v_d$] at (0,0) (r) {};
       \node[vertex] at (1,0) (a) {};
       \node[vertex] [label=below:$x$] at (2,0) (b) {};
       \node[vertex] [label=below:$1$] at (3,0) (c) {};
       \node[vertex] [label=below:$1$] at (4,0) (d) {};
       \node at (5,0) (e) {};
       \draw[dashed] (r)--(a);
       \draw[dashed] (d)--(e);
       \node [label=above:$\Bar{L}$] at (1.5,0) {}; 
       \node [label=above:$L$] at (2.5,0) {}; 
       \node [label=above:$L$] at (3.5,0) {}; 
       \end{tikzpicture}}
       \caption{A positive subpath and its corresponding alternating path.}
       \label{positive subpath example}
   \end{figure}
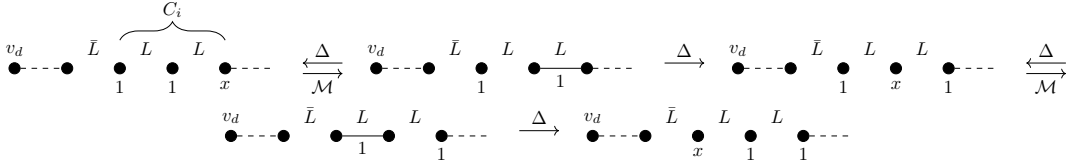

	Any alternating path through a fixed chain $\mathcal{C} = (T=T_1 > T_2 > \cdots > T_n= T')$ contains differential edges of Khovanov complex which corresponds to moving from $\CKh^+(U(T_i))$ to $\CKh^+(U(T_{i+1}))$ which in terms of activity word is a change of resolution of a crossing corresponding to a dead edge in $T_i$. There are only two following changes possible: 
	
	$$
	\Bar{d} \rightarrow \Bar{D} \quad D \rightarrow d
	$$
   
   Let $e_i$ be the dead edge whose resolution is changed while moving from $T_i$ to $T_{i+1}$. We remove an edge $f_i$ from $cyc(T_i,e_i)$ if $a_{T_i}(e_i)=\Bar{d}$ else we insert an edge in $cut(T_i,e_i)$ if $a_{T_i}(e_i)=D$ to obtain $T_{i+1}$. As a result, the activity of the edges in both $cyc(T_i,e_i)$ and $cut(T_{i+1},e_i)$ may change if $a_{T_i}(e_i)=\Bar{d}$ and similarly, the activity of the edges in both $cut(T_i,e_i)$ and $cyc(T_{i+1},f_i)$ may change if $a_{T_i}(e_i)=D$. We study the behavior of these changes in Proposition \ref{atmost one edge} in terms of the existence of an alternating path between $S_c^{T^+}$ and $S_c^{T'^+}$.

   \begin{definition}
      Let $\Pt$ be an alternating path between $S_c^{T^+}$ and $S_c^{T'^+}$ through a chain $\Ch = (T=T_1 > T_2 > \cdots > T_n= T')$. The \textit{ith broken path} is defined to be the subpath $\Pt \cap \mathcal{H}_\Mt ((\CKh^+(U(T_i)),\partial_{\Kh}^+))$.
   \end{definition}

   \begin{lemma} \label{unique states through a chain}
       Let $\Ch = (T_1 > T_2 > \cdots > T_n)$ be a chain of spanning trees and let $\Pt_1$ and $\Pt_2$ be two alternating paths from $S_c^{T_1^+}$ to $S_c^{T_n^+}$ with $P_i^j$ being the $ith$ broken path of $\Pt_j$ for $i=1,\cdots,n$ and $j=1,2$. Let $S^j_i$ and $E^j_i$ be the first and last enhanced spanning subgraphs in $P_i^j$, then $(S_i^1,E_i^1)=(S_i^2,E_i^2)$ as spanning subgraphs of $U(T_i)$ for all $i=1,\cdots,n$.
   \end{lemma}

   \begin{proof}
       Suppose for some $1 < i < n$, $S_i^1 = S_i^2$, but $E_i^1 \neq E_i^2$, choosing such an $i$ comes from the fact that $S_i^1 = E_i^1 = S_i^2 = E_i^2$ for $i=1,n$ which follows from Proposition \ref{down path within tree} and \ref{up path within tree}. This implies that there exists live edges $g_1,g_2 \in T_i$ such that $\Res(g_1) = A$, $\Res(g_2) = B$ at $E_i^1$ and vice versa for $E_i^2$. Now change the resolution at $e_i$ in $E_i^j$ for $j=1,2$ to obtain $S_{i+1}^1$ and $S_{i+1}^2$ respectively. It is clear that $S_{i+1}^1 \neq S_{i+1}^2$ and now we claim that $E_{i+1}^1 \neq E_{i+1}^2$. Now suppose if either $g_1$ or $g_2$ becomes dead in $T_{i+1}$, then it is obvious that $E_{i+1}^1 \neq E_{i+1}^2$. So, let us assume that $g_1$ and $g_2$ remains live in $T_{i+1}$ and $g_1$ occurs before $g_2$ in $M(T_{i+1})$. If $E_{i+1}^1 = E_{i+1}^2$ and without loss of generality assume that $\Res(g_1)=\Res(g_2)=A$ in $E_{i+1}^j$ and $g_2$ is a negative twist in $U(T_{i+1})$, then there is some state $y_j \in P_{i+1}^j$ such that $y_j$ is matched with $\partial y_j$ and $y_j \mapsto \partial y_j$ is due to $g_j$ and thus, in $y_j$, we have $\Res(g_j)=A$. Now, in order for the matching between $y_2$ and $\partial y_2$ to occur, we must have $\Res(g_2)=B$ in both $y_2$ and $\partial y_2$, but since $\Res(g_2)=A$ in $E_{i+1}^2$, so there must exist another state $y_2' \in P_{i+1}^2$ occuring before $\partial y_2$ such that $y_2' \mapsto \partial y_2'$ is non-matched and turns $\Res(g_2)$ from $A$ to $B$, but then one cannot further change $\Res(g_2)$ to $A$ because $M(T_{i+1})$ doesn't allow it. Thus, we can inductively conclude that $E_{n-1}^1 \neq E_{n-1}^2$ and hence, $S_n^1 \neq S_n^2$ which gives a contradiction.
   \end{proof}

   \begin{definition}
          Let $Com(T_i, T_{i+1})$ be the set of all edges in $G_\Lk$ which are either present or absent in both $T_i$ and $T_{i+1}$. Then, we define the following four sets:
   
   \begin{itemize}
   	 \item $I^1_i = \left\{ e \in Com(T_i,T_{i+1}) \mid a_{T_i}(e) = L/\Bar{l}, a_{T_{i+1}}(e) = D/\Bar{d} \right\}$.
   	 
   	 \item $I^2_i = \left\{ e \in Com(T_i,T_{i+1}) \mid a_{T_i}(e) = \Bar{L}/l, a_{T_{i+1}}(e) = \Bar{D}/d \right\}$.
   	 
   	 \item $I^3_i = \left\{ e \in Com(T_i,T_{i+1}) \mid a_{T_i}(e) = D/\Bar{d}, a_{T_{i+1}}(e) = L/\Bar{l} \right\}$.
   	 
   	 \item $I^4_i = \left\{ e \in Com(T_i,T_{i+1}) \mid a_{T_i}(e) = \Bar{D}/d, a_{T_{i+1}}(e) = \Bar{L}/l \right\}$.
   \end{itemize}
   \end{definition}
   
      \begin{proposition}\label{atmost one edge}
    	Let $P_i$ be the ith broken path for $\Pt$. Then,
   	\begin{enumerate}
   		\item every enhanced spanning subgraph in $P_i$ can have atmost one edge from $I^1_i$ with a $B$--marker.
   		
   		\item No enhanced spanning subgraph in $P_i$ can have any edge from $I_i^2$ with a $A$--marker. 
   		
   		\item No enhanced spanning subgraph in $P_{i+1}$ can have any edge from $I_i^3$ with a $B$-marker.
   		
   		\item every enhanced spanning subgraph in $P_{i+1}$ can have atmost one edge from $I^4_i$ with a $A$--marker.
   	\end{enumerate} 
   \end{proposition}
   
   \begin{proof}
   (1) $\implies$ Suppose there exists $S \in P_i$ with more than one $B$--marked edge coming from $I_i^1$ then, there is an enhanced spanning subgraph $S' \in P_i$ where these edges are $A$--marked. Since, there is an alternating path between $S$ and $S'$ where these $L/\Bar{l}$ edges are being matched in matching order of $T_i$ and thus one of them gets matched first to change from $B$ to $A$--marker but this prevents the rest of the live edges to change their resolution. Thus, we have a contradiction.	  
   (4) follows from a similar argument.\\
   	 
   (2) $\implies$ Let $S$ be an enhanced spanning subgraph in $P_i$ having an edge $e$ from $I_i^2$ with a $A$--marker such that $\partial S$ is obtained from $S$ due to a change in marker at $e$. Now in order to match $S$ with an enhanced spanning subgraph $S'$, we must perform the matching due to some live edge $f$. Now without loss of generality, assume that $f$ occurs before $e$ in $M(T_i)$, which means in order to match $\partial S$, we need to use a live edge which occurs after $e$ in $M(T_i)$ which is not possible since, all such edges after $e$ in $M(T_i)$ must be absent in $\partial S$ and the vertices between them must have same enhancement as $S_c^{T_i^+}$ in order for the matching (due to $f$) to exist. (3) follows from a similar argument.
   \end{proof}
   
   \begin{proposition} \label{combinatorial description for red complex}
   	 For a given fixed chain of spanning trees $\Ch= (T= T_1 > \cdots > T_n = T')$, $\Gamma(T^+,T'^+)$ can be computed entirely using $M(T_i)$ and $W(T_i)$ for all $i=1, \cdots, n$.
   \end{proposition}
   
   \begin{proof}
   	
    For a given fixed chain of spanning trees $\Ch= (T= T_1 > \cdots > T_n = T')$, let $(e_i,f_i)$ be the pair of edges which differs $T_i$ from $T_{i+1}$ and as before let $e_i$ be the dead edge whose change in resolution results in moving from the $ith$ broken path to $(i+1)th$. The initial necessary step to find an alternating path is to find the pairs $(S_i,E_i)$ for each broken path. We get $(S_1,E_1)$ and $(S_n,E_n)$ from Proposition \ref{down path within tree} and $\ref{up path within tree}$ respectively. Now we change the marker at $e_1$ to get $S_2$ from $E_1$ and assume that we have found $S_i$ by changing the marker at $e_{i-1}$ at $E_{i-1}$, then we look at all the live edges in $T_i$ which have a $B$-marker in $S_i$ (say we have $k_i$ many) and since $i(E_i)=i(S_i)-1$, so we have $k_i$ many choices for $E_i$ and now for each such choice, we change the marker at $e_i$ to get $k_i$ many choices of $S_{i+1}$. We continue this process inductively until we are able to reach $S_n$. For each inductive step, we take the help of the combinatorial information coming from $W(T_i),W(T_{i+1})$, Proposition \ref{atmost one edge}, resolution of $f_i$, $M(T_i),M(T_{i+1})$ for finding these pairs which reach upto $S_n$. If we find one such collection of pairs $(S_i,E_i)$, then Lemma \ref{unique states through a chain} asserts that these are unique pairs for any alternating path.\\
    Now for finding the enhancements of these pairs, we first observe that $\epsilon(S_1),\epsilon(E_1)$ and $\epsilon(S_n),\epsilon(E_n)$ comes from Proposition \ref{down path within tree} and \ref{up path within tree} respectively. Now inductively assume that we have $\epsilon(S_i)$ and we know $E_i$, then use $M(T_i)$ to find all the possible enhanced subgraphs whose resolution matches with $E_i$ and there is an alternating path to it from $S_i$ (the $ith$ broken path), if we find such an enhanced $E_i$, we change the marker at $e_i$ to obtain all possible enhancement of $S_{i+1}$. We continue this until we reach any of the desired $(S_n,\epsilon(S_n))$ which we have previously obtained.  

   \end{proof} 

   \begin{example}
			
			\begin{figure}[!tbph]
				\centering
				\subfigure[]{\includesvg[scale=0.65]{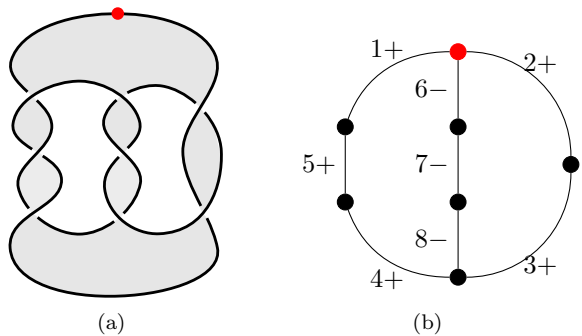}}
				\qquad 
				\subfigure[]{
					\begin{tikzpicture}[vertex/.style={circle, draw, inner sep=0pt, minimum size=6pt, fill=black}, edge/.style={draw}, scale=1]
						\node[vertex,fill=red,color=red] (a) at (0,2) {};
						\node[vertex] (b) at (0,1) {};
						\node[vertex] (c) at (0,0) {};
						\node[vertex] (d) at (0,-1) {};
						\node[vertex] (e) at (-1.5,1) {};
						\node[vertex] (f) at (-1.5,0) {};
						\node[vertex] (g) at (1.5,0.5) {};
						\path[edge] (a) to node[left] {$6-$} (b);
						\path[edge] (b) to node[left] {$7-$} (c);
						\path[edge] (c) to node[left] {$8-$} (d);
						\path[edge] (a) to[out=180,in=70] node[above] {$1+$} (e);
						\path[edge] (e) to node[left] {$5+$} (f);
						\path[edge] (f) to[out=290,in=180] node[below] {$4+$} (d);
						\path[edge] (a) to[out=0,in=90] node[above] {$2+$} (g);
						\path[edge] (g) to[out=270,in=0] node[below] {$3+$} (d);
				\end{tikzpicture}}
				\caption{$8_{20}$ and its Tait graph}
				\label{example 8_20}
			\end{figure}

      For the spanning tree complex and computations for the differential for the knot $8_{20}$ in Figure \ref{example 8_20}, see Section \ref{example computation} in the appendix. For the time being, we look at a specific incidence between two spanning trees $T_5^+$ and $T_{19}^+$ coming from Table \ref{chain complex table for 8_20} in order to understand the working principle of Proposition \ref{combinatorial description for red complex}.\\

      We label the edges according to their order (See Figure \ref{example 8_20}) as $e_1, e_2, \cdots e_8$. Here we have $W(T_9) = LLLdD \Bar{L} \Bar{d} \Bar{D}$ and one can easily verify that arriving at $T_{19}$ is only possible through the chain $\Ch = T_9 > T_4 > T_5 > T_{19}$. In order for the existence of an alternating path, we must have the existence of unique (without considering enhancements) spanning subgraphs $(S_i,E_i)$ for each $T_i$ for $i=4,5,9,19$. It clear that from Proposition \ref{down path within tree} and \ref{up path within tree} that $S_9 = E_9 = S_c^{T_9^+}$ and $S_{19} = E_{19} = S_c^{T_{19}^+}$ respectively. We change the marker at $e_7$ to obtain $S_4$ from $E_9$ and now $M(T_4)=L_2L_3L_1$ and all of these live edges have $B$-marker in $S_4$, so the possible choices of $E_4$ are the enhanced spanning subgraphs with resolution $(L_2L_3L_1)=(BBA),(BAB),(ABB)$. Now if we choose any one of them as $E_4$ except for $(ABB)$ and continue the process of finding $E_5$ from $S_5$, one can easily show that $S_{19}$ cannot be obtained. Thus, we get $E_4$ from $S_4$ by changing $e_2$ to $A$-marker at $S_4$ and then we change the marker at $e_5$ to obtain $S_5$. Now similarly, we have $M(T_5)=L_2L_3L_4L_1$ and the resolution of these live edges at $S_5$ is $(ABBB)$ and again we have three choices, out of which only $(AABB)$ is eligible for $E_5$ since, $I_5^1 = \{e_3\}$ and  Proposition \ref{atmost one edge} says that it has to be converted into an $A$-marker. Thus, finally changing the marker at $e_6$ gives us $S_{19}$. Now we only need to find all possible alternating paths through these pairs $(S_i,E_i)$ due to Lemma \ref{unique states through a chain} (See Figure \ref{alt path sample example}). \\
      Now we only need to find the enhancements of these states. $\epsilon(S_1)=\epsilon(E_1)=\epsilon(S_c^{T_9^+})$ comes from Proposition \ref{down path within tree} and $\epsilon(S_{19})=\epsilon(E_{19})=\epsilon(S_c^{T_{19}^+})$ comes from Proposition \ref{up path within tree}. Hence, we get $\epsilon(S_4)$ and $\epsilon(E_5)$. Now as $E_4$ is matched due to the edge $e_2$, we get all the enhancements of components of $E_4$ due to live edges $e_3$ and $e_1$ from $M(T_4)$. The component of $e_2$ has $1$ as its enhancement and thus, we can use $M(T_4)$ to obtain the unique path from $S_4$ to $E_4$. Now, we have two choices of $\epsilon(S_5)$ coming from $E_4$ due to the split map at $e_5$ but since, $E_5$ is matched due to $e_3$, so there is a unique choice for $S_5$ and thus, we get the unique path from $S_5$ to $E_5$ from $M(T_5)$ and thus, we have only one unique choice for $\epsilon(S_{19})$. Therefore, we have an unique alternating path from $T_9^+$ to $T_{19}^+$.
      
   \end{example}
   
   	\begin{table}[ht]
   	\centering
   	\setlength{\tabcolsep}{4mm}
   	\def\arraystretch{1.5}
   	\resizebox{0.8\textwidth}{!}{
   	\begin{tabular}{cccccc}

   		\begin{tikzpicture}[vertex/.style={circle, draw, inner sep=0pt, minimum size=6pt, fill=black}, edge/.style={draw}, baseline=0]
   			\node[vertex,fill=red,color=red] [label=above:$1$] (a) at (0,2) {};
   			\node[vertex] [label=below:$x$] (b) at (0,1) {};
   			\node[vertex] (c) at (0,0) {};
   			\node[vertex] [label=left:$1$] (d) at (0,-1) {};
   			\node[vertex] [label=left:$1$] (e) at (-1.5,1) {};
   			\node[vertex] (f) at (-1.5,0) {};
   			\node[vertex] [label=right:$1$] (g) at (1.5,0.5) {};
   			\path[edge,dashed] (a) to (b);
   			\path[edge] (c) to (d);
   			\path[edge,dashed] (a) to[out=180,in=70] (e);
   			\path[edge] (e) to (f);
   			\path[edge,dashed] (a) to[out=0,in=90] (g);
   			\path[edge,dashed] (g) to[out=270,in=0] (d);
   			\node at (0,-2) {$T_9$ ($LLLdD\Bar{L}\Bar{d}\Bar{D}$)};
   		\end{tikzpicture} & \begin{tikzpicture}
   			\draw[->,very thick,color=blue] (0,0) -- node[midway,above,color=black] {$m$} (1,0);
   		\end{tikzpicture} &
   		
   		\begin{tikzpicture}[vertex/.style={circle, draw, inner sep=0pt, minimum size=6pt, fill=black}, edge/.style={draw}, baseline=0]
   			\node[vertex,fill=red,color=red] [label=above:$1$] (a) at (0,2) {};
   			\node[vertex] (b) at (0,1) {};
   			\node[vertex] (c) at (0,0) {};
   			\node[vertex] [label=left:$x$] (d) at (0,-1) {};
   			\node[vertex] [label=left:$1$] (e) at (-1.5,1) {};
   			\node[vertex] (f) at (-1.5,0) {};
   			\node[vertex] [label=right:$1$] (g) at (1.5,0.5) {};
   			\path[edge] (b) to (c);
   			\path[edge] (c) to (d);
   			\path[edge,dashed] (a) to[out=180,in=70] (e);
   			\path[edge] (e) to (f);
   			\path[edge,dashed] (a) to[out=0,in=90] (g);
   			\path[edge,dashed,very thick,color=red] (g) to[out=270,in=0] (d);
   			\node at (0,-2) {$T_4$ ($LLLdD\Bar{d}\Bar{D}\Bar{D}$)};
   		\end{tikzpicture} & $\xrightarrow{\Mt}$ &
   		
   		\begin{tikzpicture}[vertex/.style={circle, draw, inner sep=0pt, minimum size=6pt, fill=black}, edge/.style={draw}, baseline=0]
   			\node[vertex,fill=red,color=red] [label=above:$1$] (a) at (0,2) {};
   			\node[vertex] (b) at (0,1) {};
   			\node[vertex] (c) at (0,0) {};
   			\node[vertex] (d) at (0,-1) {};
   			\node[vertex] [label=left:$1$] (e) at (-1.5,1) {};
   			\node[vertex] (f) at (-1.5,0) {};
   			\node[vertex] [label=right:$1$] (g) at (1.5,0.5) {};
   			\path[edge] (b) to (c);
   			\path[edge] (c) to (d);
   			\path[edge,dashed] (a) to[out=180,in=70] (e);
   			\path[edge] (e) to (f);
   			\path[edge,dashed] (a) to[out=0,in=90] (g);
   			\path[edge,very thick,color=red] (g) to[out=270,in=0] (d);
   			\node at (0,-2) {$T_4$};
   		\end{tikzpicture} & $\xrightarrow{\Delta}$ \\
   		
   		\begin{tikzpicture}[vertex/.style={circle, draw, inner sep=0pt, minimum size=6pt, fill=black}, edge/.style={draw}, baseline=0]
   			\node[vertex,fill=red,color=red] [label=above:$1$] (a) at (0,2) {};
   			\node[vertex] (b) at (0,1) {};
   			\node[vertex] (c) at (0,0) {};
   			\node[vertex] [label=left:$1$] (d) at (0,-1) {};
   			\node[vertex] [label=left:$1$] (e) at (-1.5,1) {};
   			\node[vertex] (f) at (-1.5,0) {};
   			\node[vertex] [label=right:$x$] (g) at (1.5,0.5) {};
   			\path[edge] (b) to (c);
   			\path[edge] (c) to (d);
   			\path[edge,dashed] (a) to[out=180,in=70] (e);
   			\path[edge] (e) to (f);
   			\path[edge,dashed,very thick,color=red] (a) to[out=0,in=90] (g);
   			\path[edge,dashed] (g) to[out=270,in=0] (d);
   			\node at (0,-2) {$T_4$};
   		\end{tikzpicture} & $\xrightarrow{\Mt}$ &
   		
   		\begin{tikzpicture}[vertex/.style={circle, draw, inner sep=0pt, minimum size=6pt, fill=black}, edge/.style={draw}, baseline=0]
   			\node[vertex,fill=red,color=red] [label=above:$1$] (a) at (0,2) {};
   			\node[vertex] (b) at (0,1) {};
   			\node[vertex] (c) at (0,0) {};
   			\node[vertex] [label=left:$1$] (d) at (0,-1) {};
   			\node[vertex] [label=left:$1$] (e) at (-1.5,1) {};
   			\node[vertex] (f) at (-1.5,0) {};
   			\node[vertex] (g) at (1.5,0.5) {};
   			\path[edge] (b) to (c);
   			\path[edge] (c) to (d);
   			\path[edge,dashed] (a) to[out=180,in=70] (e);
   			\path[edge] (e) to (f);
   			\path[edge,very thick,color=red] (a) to[out=0,in=90] (g);
   			\path[edge,dashed] (g) to[out=270,in=0] (d);
   			\node at (0,-2) {$T_4$};
   		\end{tikzpicture} & \begin{tikzpicture}
   			\draw[->,very thick,color=blue] (0,0) -- node[midway,above,color=black] {$\Delta$} (1,0);
   		\end{tikzpicture} &
   		
   		\begin{tikzpicture}[vertex/.style={circle, draw, inner sep=0pt, minimum size=6pt, fill=black}, edge/.style={draw}, baseline=0]
   			\node[vertex,fill=red,color=red] [label=above:$1$] (a) at (0,2) {};
   			\node[vertex] (b) at (0,1) {};
   			\node[vertex] (c) at (0,0) {};
   			\node[vertex] [label=below:$1$] (d) at (0,-1) {};
   			\node[vertex] [label=left:$1$] (e) at (-1.5,1) {};
   			\node[vertex] [label=left:$x$] (f) at (-1.5,0) {};
   			\node[vertex] (g) at (1.5,0.5) {};
   			\path[edge] (b) to (c);
   			\path[edge] (c) to (d);
   			\path[edge,dashed] (a) to[out=180,in=70] (e);
   			\path[edge,dashed,very thick,color=red] (f) to[out=290,in=180] (d);
   			\path[edge] (a) to[out=0,in=90] (g);
   			\path[edge,dashed] (g) to[out=270,in=0] (d);
   			\node at (0,-2) {$T_5$ ($LLLLd\Bar{d}\Bar{D}\Bar{D}$)};
   		\end{tikzpicture} & $\xrightarrow{\Mt}$ \\
   		
   		\begin{tikzpicture}[vertex/.style={circle, draw, inner sep=0pt, minimum size=6pt, fill=black}, edge/.style={draw}, baseline=0]
   			\node[vertex,fill=red,color=red] [label=above:$1$] (a) at (0,2) {};
   			\node[vertex] (b) at (0,1) {};
   			\node[vertex] (c) at (0,0) {};
   			\node[vertex] [label=below:$1$] (d) at (0,-1) {};
   			\node[vertex] [label=left:$1$] (e) at (-1.5,1) {};
   			\node[vertex] (f) at (-1.5,0) {};
   			\node[vertex] (g) at (1.5,0.5) {};
   			\path[edge] (b) to (c);
   			\path[edge] (c) to (d);
   			\path[edge,dashed] (a) to[out=180,in=70] (e);
   			\path[edge,very thick,color=red] (f) to[out=290,in=180] (d);
   			\path[edge] (a) to[out=0,in=90] (g);
   			\path[edge,dashed] (g) to[out=270,in=0] (d);
   			\node at (0,-2) {$T_5$};
   		\end{tikzpicture} & $\xrightarrow{\Delta}$ &
   		
   		\begin{tikzpicture}[vertex/.style={circle, draw, inner sep=0pt, minimum size=6pt, fill=black}, edge/.style={draw}, baseline=0]
   			\node[vertex,fill=red,color=red] [label=above:$1$] (a) at (0,2) {};
   			\node[vertex] (b) at (0,1) {};
   			\node[vertex] (c) at (0,0) {};
   			\node[vertex] [label=below:$x$] (d) at (0,-1) {};
   			\node[vertex] [label=left:$1$] (e) at (-1.5,1) {};
   			\node[vertex] [label=left:$1$] (f) at (-1.5,0) {};
   			\node[vertex] (g) at (1.5,0.5) {};
   			\path[edge] (b) to (c);
   			\path[edge] (c) to (d);
   			\path[edge,dashed] (a) to[out=180,in=70] (e);
   			\path[edge,dashed] (f) to[out=290,in=180] (d);
   			\path[edge] (a) to[out=0,in=90] (g);
   			\path[edge,dashed,very thick,color=red] (g) to[out=270,in=0] (d);
   			\node at (0,-2) {$T_5$};
   		\end{tikzpicture} & $\xrightarrow{\Mt}$ &
   		
   		\begin{tikzpicture}[vertex/.style={circle, draw, inner sep=0pt, minimum size=6pt, fill=black}, edge/.style={draw}, baseline=0]
   			\node[vertex,fill=red,color=red] [label=above:$1$] (a) at (0,2) {};
   			\node[vertex] (b) at (0,1) {};
   			\node[vertex] (c) at (0,0) {};
   			\node[vertex] (d) at (0,-1) {};
   			\node[vertex] [label=left:$1$] (e) at (-1.5,1) {};
   			\node[vertex] [label=left:$1$] (f) at (-1.5,0) {};
   			\node[vertex] (g) at (1.5,0.5) {};
   			\path[edge] (b) to (c);
   			\path[edge] (c) to (d);
   			\path[edge,dashed] (a) to[out=180,in=70] (e);
   			\path[edge,dashed] (f) to[out=290,in=180] (d);
   			\path[edge] (a) to[out=0,in=90] (g);
   			\path[edge,very thick,color=red] (g) to[out=270,in=0] (d);
   			\node at (0,-2) {$T_5$};
   		\end{tikzpicture}  & \begin{tikzpicture}
   			\draw[->,very thick,color=blue] (0,0) -- node[midway,above,color=black] {$\Delta$} (1,0);
   		\end{tikzpicture} \\
   		
   		\begin{tikzpicture}[vertex/.style={circle, draw, inner sep=0pt, minimum size=6pt, fill=black}, edge/.style={draw}, baseline=0]
   			\node[vertex,fill=red,color=red] [label=above:$1$] (a) at (0,2) {};
   			\node[vertex] (b) at (0,1) {};
   			\node[vertex] (c) at (0,0) {};
   			\node[vertex] (d) at (0,-1) {};
   			\node[vertex] [label=left:$1$] (e) at (-1.5,1) {};
   			\node[vertex] [label=left:$1$] (f) at (-1.5,0) {};
   			\node[vertex] (g) at (1.5,0.5) {};
            \node (h) at (0.75,0.5) {$x$};
   			\path[edge] (a) to (b);
   			\path[edge] (b) to (c);
   			\path[edge] (c) to (d);
   			\path[edge,dashed] (a) to[out=180,in=70] (e);		   	
   			\path[edge,dashed] (f) to[out=290,in=180] (d);		   
   			\path[edge] (g) to[out=270,in=0] (d);
   			\path[edge, color=blue] (a) to[out=0,in=90] (g);
   			\node at (0,-2) {$T_{19}$ ($LlDLd\Bar{D}\Bar{D}\Bar{D}$)};
   		\end{tikzpicture}
 
   	\end{tabular}}
   	
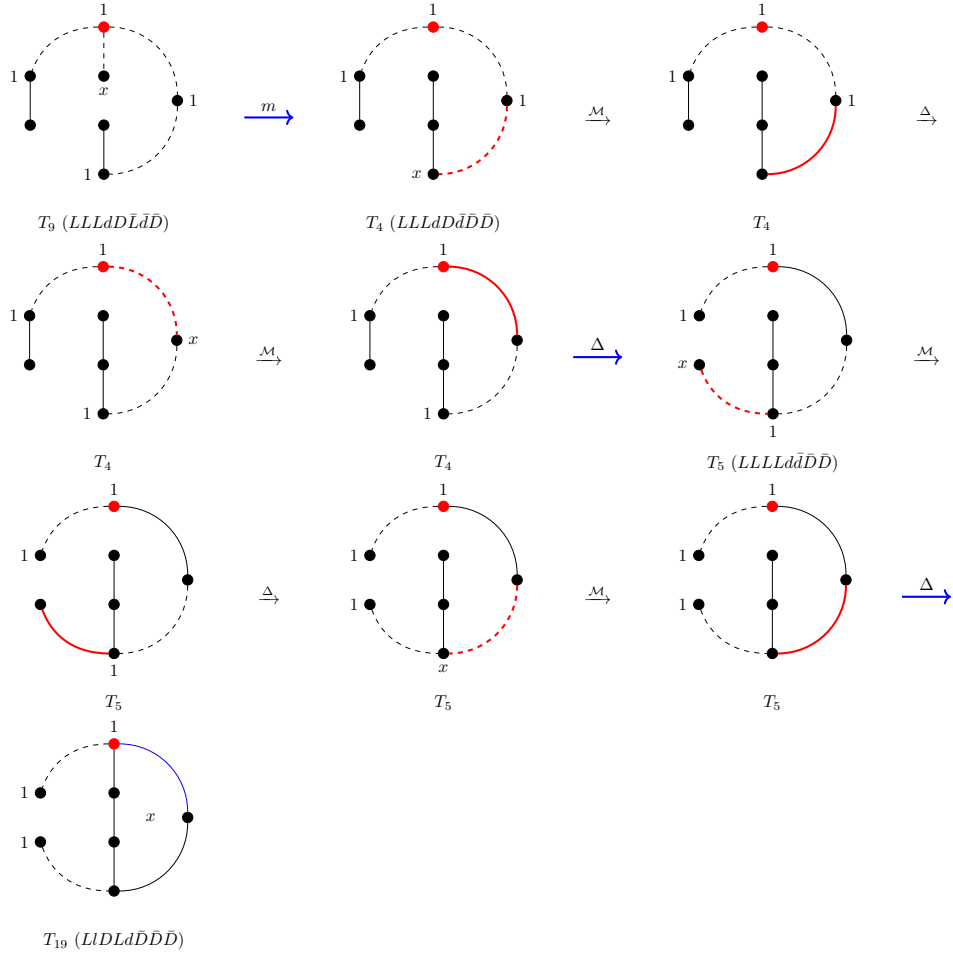
\captionof{figure}{An alternating path between $T_5^+$ and $T_{19}^+$ from example \ref{example computation}. The blue arrows indicates the change of complex from one twisted unknot of a spanning tree to the other.}
    \label{alt path sample example}
   \end{table}

\end{subsection}

\end{section}


\begin{section}{Unreduced spanning tree complex} \label{unreduced Kh complex section}
	
	\begin{subsection}{Acyclic matching on the unreduced Khovanov complex}
		
		In the case of reduced Khovanov complex, the enhanced spanning subgraphs $(H,\epsilon)$ had a restriction which was $\epsilon(C) = 1$ for $C \in C(H)$ with $v_d  \in C$. If we change this restriction to $\epsilon(C)=x$ then we have another complex $\CKh^-(U(T))$. The unreduced Khovanov complex for a twisted unknot $U(T)$ is given by 
		
		$$
		   \CKh(U(T)) = \CKh^+(U(T)) \oplus \CKh^-(U(T))
		$$
		
		Thus, for a given connected link diagram $\Lk$, the unreduced Khovanov complex of $\Lk$ is given by:
		
		$$
		   \CKh(\Lk) = \bigoplus_T \CKh(U(T))
		$$
		
		Observe that the matching in Theorem \ref{maintheorem} did not depend on the above mentioned restriction. Thus, one can easily define a similar near-perfect acyclic matching on $\mathcal{H}((\CKh^-(U(T)),\partial_{\Kh}^-))$. Thus, we do have an acyclic matching on $\mathcal{H}((\CKh(U(T)),\partial_{\Kh}))$ with exactly two critical enhanced spanning subgraphs which we denote by $S_c^{T^+}$ and $S_c^{T^-}$.
		
		\begin{theorem}\label{matching theorem in unreduced}
			Let $T$ be a spanning tree of the Tait graph $G_\Lk$ associated with a connected link diagram $\Lk$, then there exists an acyclic matching on $\mathcal{H}((\CKh(U(T)),\partial_{\Kh}))$, where $\mathcal{H}((\CKh(U(T)),\partial_{\Kh}))$ is the unreduced Khovanov complex of $U(T)$. Furthermore, by taking the union of these matchings over all spanning trees of $G_\Lk$, we get an acyclic matching $\Mt^{\Kh}_\Lk$ on $\mathcal{H}((\CKh(\Lk),\partial_{\Kh}))$.
		\end{theorem}
		
		\begin{proof}
			The existence of a matching in $\CKh(U(T))$ is clear from the above discussion. Now if a directed cycle exists in $\mathcal{H}_{\Mt}((\CKh(U(T)),\partial_{\Kh}))$, then it must contain a directed edge $(e \rightarrow e')$ such that $e \in \CKh^+(U(T))$ and $e' \in \CKh^-(U(T))$ but simultaneously there also must be a directed edge $(f \rightarrow f')$ such that $f \in \CKh^-(U(T))$ and $f' \in \CKh^+(U(T))$ but for such an edge to exist the component containing $v_d$ in $f$ with enhancement $x$ must either merge($m$) or split($\Delta$) to give $f$ and the component containing $v_d$ in $f$ must then $1$ as its enhancement, but the definition of the maps $m$ and $\Delta$ in Khovanov chain complex does not allow it. Thus, $\Mt_\Lk$ is an acyclic matching.  
		\end{proof}
	\end{subsection}
	
	\begin{subsection}{Spanning tree complex}
		For a given connected link diagram $\Lk$, we have previously defined the reduced spanning tree complex $\CST^+(\Lk)$ and its generators are the spanning trees of $G_\Lk$, where $v_d$ has the enhancement $1$ and thus, we denoted its generators by $T^+$. Similarly, we now have the counterpart where the enhancement of $v_d$ is $x$ and we denote the reduced counterpart by $\CST^-(\Lk)$ where the generators are the same spanning trees, but denoted by $T^-$. Thus we define the following total complex:
		
		\begin{definition}
			The \textit{unreduced spanning tree complex} or simply spanning tree complex is defined to be the bigraded complex,
			$$
			\CST(\Lk) = \bigoplus_{i,j}\CST^{i,j}(\Lk)
			$$
			where, 
			$$
			\CST^{i,j}(\Lk)= {\CST^+}^{i,j}(\Lk) + {\CST^-}^{i,j}(\Lk)
			$$ 
			with ${\CST^-}^{i,j}(\Lk)$ being defined similarly as ${\CST^+}^{i,j}(\Lk)$.
		\end{definition}
		
		We define the differential for the spanning tree complex $\partial_{ST} : \CST^{i,j}(\Lk) \rightarrow \CST^{i+1,j}(\Lk)$ as
		
		\begin{equation} \label{unreduced differential}
		   \partial_{ST}(T) := \left\{
		\begin{aligned}
			& \partial^+_{ST}(T) + \sum \Gamma(T^+,{T'}^-){T'}^- & \text{if } T=T^+ \\
			& \partial^-_{ST}(T) & \text{if } T=T^-
		\end{aligned}
		\right. 
		\end{equation}

        where, $\Gamma(T^+,T'^-) = \Gamma(S_c^{T^+},S_c^{T'^-})$ and $\partial_{ST}^-$ is defined similarly as in \ref{reduced differential}.

        \begin{proof}[Proof of Theorem \ref{unreduced maintheorem}]
            Theorem \ref{dmt maintheorem} and the definition of $\partial_{ST}$ in \ref{unreduced differential} implies the following:

            $$
             H^*((\CST(\Lk),\partial_{ST})) \cong \Kh(\Lk)
            $$
        \end{proof}
        
		All results proven in subsection \ref{alternating path count subsection} also holds for the complex $\CST^-(\Lk)$. Hence, in order to completely describe $\partial_{ST}$, we only need to look for the existence of an alternating path from $S_c^{T^+}$ to $S_c^{{T'}^-}$.

        \begin{definition}
   	Given a spanning tree $T$, consider any two enhanced spanning subgraphs $S_x^{T}$ and $S_y^{T}$ in $\CKh(U(T))$. Then we similarly define the following:
   	\begin{enumerate}
   		\item $\A^{\downarrow}(S_x^{T},S_y^{T})$ := \# alternating paths from $S_x^{T}$ to $S_y^{T}$ in $\mathcal{H}_\Mt ((\CKh(U(T)),\partial_{\Kh}))$, where the paths begin with a non-matched edge.
   		
   		\item $\A^{\uparrow}(S_x^{T^+},S_y^{T^+})$ := \# alternating paths from $S_x^{T}$ to $S_y^{T}$ in $\mathcal{H}_\Mt ((\CKh(U(T)),\partial_{\Kh}))$, where the paths begin with a matched edge.
   	\end{enumerate}
    \end{definition}
		
		\begin{definition}
			A \textit{rooted negative subpath} of $G(H_x)$ is subgraph of $G(H_x)$ which is itself a negative subpath and moreover the initial vertex $v_C$ of each connected component $C$ of this negative subpath can be $v_d$.
		\end{definition}
		
		\begin{definition}
			A \textit{rooted positive subpath} of $G(H_x)$ is subgraph of $G(H_x)$ which is itself a positive subpath and moreover the initial vertex $v_C$ of each connected component $C$ of this positive subpath can be $v_d$.
		\end{definition}
		
		\begin{proposition}\label{rooted negative subpaths}
		    For a given spanning tree $T$ of $G_\Lk$, the following two sets are in one-one correspondence:
			
			\[\begin{tikzcd}[ampersand replacement=\&,cramped]
				{\left\{S_x^{T} \mid \A^\downarrow(S_c^{T^+},S_x^{T})=1, i(S_c^{T^+})=i(S_x^{T})\right\}} \&\& {\left\{\text{Rooted negative subpaths of }G(H_c)=G(T)\right\}}
				\arrow[shorten <=4pt, shorten >=4pt, tail reversed, from=1-1, to=1-3]
			\end{tikzcd}\]
			
		\end{proposition}
		
	    \begin{proof}
	    	If $S_x^T \in \CKh^+(U(T))$ then it corresponds to negative subpaths of $G(T)$ as in Proposition \ref{down path within tree}. If we have a ``proper" rooted negative subpath $\Pt$, which means there is a connected component of $\Pt$ which contains an edge incident to $v_d$, then we can use that component to move from $\CKh^+(U(T))$ to $\CKh^-(U(T))$ using an alternating path consisting entirely of merge maps following $M(T)$ (See Figure \ref{rooted negative subpath example}). Thus, states $S_x^T \in \CKh^-(U(T))$ corresponds to proper rooted negative subpaths.
	    \end{proof}

   \begin{table}[ht]
     \centering
     \setlength{\tabcolsep}{4mm}
     \def\arraystretch{1.5}
     \resizebox{\textwidth}{!}{
         \begin{tabular}{ccccccccc}
           \begin{tikzpicture}[vertex/.style={circle, draw, inner sep=0pt, minimum size=6pt, fill=black}, edge/.style={draw}, baseline=0]
          \node[vertex,fill=red,color=red] [label=above:$1$] (a) at (0,2) {};
	   \node[vertex] [label=right:$x$] (b) at (0,1) {};
	   \node[vertex] [label=right:$x$] (c) at (0,0) {};
	   \node[vertex] (d) at (0,-1) {};
	   \node[vertex] [label=left:$1$] (e) at (-1.5,1) {};
	   \node[vertex] (f) at (-1.5,0) {};
	   \node[vertex] [label=right:$1$] (g) at (1.5,0.5) {};
	   \path[edge,dotted] (a) to (b);
	   \path[edge,dotted] (b) to (c);
	   \path[edge,dotted] (a) to[out=180,in=70] (e);
	   \path[edge] (e) to (f);
	   \path[edge] (f) to[out=290,in=180] (d);
	   \path[edge,dotted] (a) to[out=0,in=90] (g);
	\end{tikzpicture} & 
         \Large{$\xrightarrow{m}$} &
        \begin{tikzpicture}[vertex/.style={circle, draw, inner sep=0pt, minimum size=6pt, fill=black}, edge/.style={draw}, baseline=0]
          \node[vertex,fill=red,color=red] (a) at (0,2) {};
	   \node[vertex] [label=right:$x$] (b) at (0,1) {};
	   \node[vertex] [label=right:$x$] (c) at (0,0) {};
	   \node[vertex] (d) at (0,-1) {};
	   \node[vertex] [label=left:$1$] (e) at (-1.5,1) {};
	   \node[vertex] (f) at (-1.5,0) {};
	   \node[vertex] [label=right:$1$] (g) at (1.5,0.5) {};
	   \path[edge] (a) to (b);
	   \path[edge,dotted] (b) to (c);
	   \path[edge,dotted] (a) to[out=180,in=70] (e);
	   \path[edge] (e) to (f);
	   \path[edge] (f) to[out=290,in=180] (d);
	   \path[edge,dotted] (a) to[out=0,in=90] (g);
	\end{tikzpicture} & \Large{$\xrightarrow{\Mt}$} &
        \begin{tikzpicture}[vertex/.style={circle, draw, inner sep=0pt, minimum size=6pt, fill=black}, edge/.style={draw}, baseline=0]
          \node[vertex,fill=red,color=red] [label=above:$x$] (a) at (0,2) {};
	   \node[vertex] [label=right:$1$] (b) at (0,1) {};
	   \node[vertex] [label=right:$x$] (c) at (0,0) {};
	   \node[vertex] (d) at (0,-1) {};
	   \node[vertex] [label=left:$1$] (e) at (-1.5,1) {};
	   \node[vertex] (f) at (-1.5,0) {};
	   \node[vertex] [label=right:$1$] (g) at (1.5,0.5) {};
	   \path[edge,dotted] (a) to (b);
	   \path[edge,dotted] (b) to (c);
	   \path[edge,dotted] (a) to[out=180,in=70] (e);
	   \path[edge] (e) to (f);
	   \path[edge] (f) to[out=290,in=180] (d);
	   \path[edge,dotted] (a) to[out=0,in=90] (g);
	\end{tikzpicture} & \Large{$\xrightarrow{m}$} &
        \begin{tikzpicture}[vertex/.style={circle, draw, inner sep=0pt, minimum size=6pt, fill=black}, edge/.style={draw}, baseline=0]
          \node[vertex,fill=red,color=red] [label=above:$1$] (a) at (0,2) {};
	   \node[vertex] [label=right:$x$] (b) at (0,1) {};
	   \node[vertex] (c) at (0,0) {};
	   \node[vertex] (d) at (0,-1) {};
	   \node[vertex] [label=left:$1$] (e) at (-1.5,1) {};
	   \node[vertex] (f) at (-1.5,0) {};
	   \node[vertex] [label=right:$1$] (g) at (1.5,0.5) {};
	   \path[edge,dotted] (a) to (b);
	   \path[edge] (b) to (c);
	   \path[edge,dotted] (a) to[out=180,in=70] (e);
	   \path[edge] (e) to (f);
	   \path[edge] (f) to[out=290,in=180] (d);
	   \path[edge,dotted] (a) to[out=0,in=90] (g);
	\end{tikzpicture} & \Large{$\xrightarrow{\Mt}$} & 
        \begin{tikzpicture}[vertex/.style={circle, draw, inner sep=0pt, minimum size=6pt, fill=black}, edge/.style={draw}, baseline=0]
          \node[vertex,fill=red,color=red] [label=above:$x$] (a) at (0,2) {};
	   \node[vertex] [label=right:$x$] (b) at (0,1) {};
	   \node[vertex] [label=right:$1$] (c) at (0,0) {};
	   \node[vertex] (d) at (0,-1) {};
	   \node[vertex] [label=left:$1$] (e) at (-1.5,1) {};
	   \node[vertex] (f) at (-1.5,0) {};
	   \node[vertex] [label=right:$1$] (g) at (1.5,0.5) {};
	   \path[edge,dotted] (a) to (b);
	   \path[edge,dotted] (b) to (c);
	   \path[edge,dotted] (a) to[out=180,in=70] (e);
	   \path[edge] (e) to (f);
	   \path[edge] (f) to[out=290,in=180] (d);
	   \path[edge,dotted] (a) to[out=0,in=90] (g);
	\end{tikzpicture}
    \end{tabular}}
    
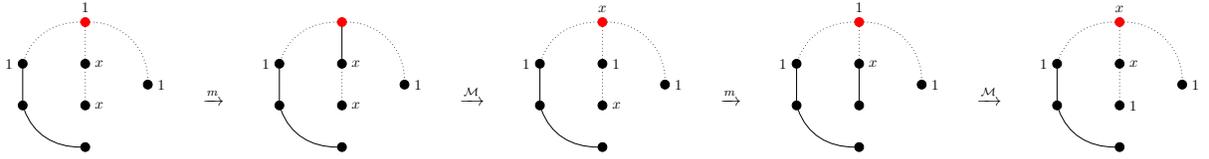
\captionof{figure}{A rooted negative subpath in $T_{13}(LLdDD\Bar{L}\Bar{L}\Bar{d})$ from Table \ref{spanning trees of 8_20}}
   \label{rooted negative subpath example}
   \end{table}
		   
		\begin{proposition}\label{rooted positive subpaths}
			For a given spanning tree $T$ of $G_\Lk$, the following two sets are in one-one correspondence:
			
			\[\begin{tikzcd}[ampersand replacement=\&,cramped]
				{\left\{S_x^{T} \mid \A^\uparrow(S_x^{T},S_c^{T^-})=1, i(S_x^{T})=i(S_c^{T^-})\right\}} \&\& {\left\{\text{Rooted positive subpaths of }G(H_c)=G(T)\right\}}
				\arrow[shorten <=4pt, shorten >=4pt, tail reversed, from=1-1, to=1-3]
			\end{tikzcd}\]
		\end{proposition}
		
		\begin{proof}
			A similar argument works over here. The states $S_x^T \in \CKh^+(U(T))$ corresponds to proper rooted positive subpaths while states $S_x^T \in \CKh^-(U(T))$ corresponds to positive subpaths as in Proposition \ref{up path within tree} (See Figure \ref{rooted positive subpath example}).
		\end{proof}

    \begin{table}[ht]
     \centering
     \setlength{\tabcolsep}{4mm}
     \def\arraystretch{1.5}
     \resizebox{\textwidth}{!}{
     \begin{tabular}{ccccccccc}
         \begin{tikzpicture}[vertex/.style={circle, draw, inner sep=0pt, minimum size=6pt, fill=black}, edge/.style={draw}, baseline=0]
	   \node[vertex,fill=red,color=red] [label=above:$1$] (a) at (0,2) {};
		\node[vertex] (b) at (0,1) {};
		\node[vertex] (c) at (0,0) {};
		\node[vertex] [label=left:$x$] (d) at (0,-1) {};
		\node[vertex] [label=left:$1$] (e) at (-1.5,1) {};
		\node[vertex] (f) at (-1.5,0) {};
		\node[vertex] [label=left:$1$] (g) at (1.5,0.5) {};
		\path[edge] (b) to (c);
		\path[edge] (c) to (d);
		\path[edge,dotted] (a) to[out=180,in=70] (e);
		\path[edge] (e) to (f);
		\path[edge,dotted] (a) to[out=0,in=90] (g);
		\path[edge,dotted] (g) to[out=270,in=0] (d);
	\end{tikzpicture} & 
         \Large{$\xrightarrow{\Mt}$} &
        \begin{tikzpicture}[vertex/.style={circle, draw, inner sep=0pt, minimum size=6pt, fill=black}, edge/.style={draw}, baseline=0]
	   \node[vertex,fill=red,color=red] [label=above:$1$] (a) at (0,2) {};
		\node[vertex] (b) at (0,1) {};
		\node[vertex] (c) at (0,0) {};
		\node[vertex] (d) at (0,-1) {};
		\node[vertex] [label=left:$1$] (e) at (-1.5,1) {};
		\node[vertex] (f) at (-1.5,0) {};
		\node[vertex] [label=left:$1$] (g) at (1.5,0.5) {};
		\path[edge] (b) to (c);
		\path[edge] (c) to (d);
		\path[edge,dotted] (a) to[out=180,in=70] (e);
		\path[edge] (e) to (f);
		\path[edge,dotted] (a) to[out=0,in=90] (g);
		\path[edge] (g) to[out=270,in=0] (d);
	\end{tikzpicture} & \Large{$\xrightarrow{\Delta}$} &
        \begin{tikzpicture}[vertex/.style={circle, draw, inner sep=0pt, minimum size=6pt, fill=black}, edge/.style={draw}, baseline=0]
	   \node[vertex,fill=red,color=red] [label=above:$1$] (a) at (0,2) {};
		\node[vertex] (b) at (0,1) {};
		\node[vertex] (c) at (0,0) {};
		\node[vertex] [label=left:$1$] (d) at (0,-1) {};
		\node[vertex] [label=left:$1$] (e) at (-1.5,1) {};
		\node[vertex] (f) at (-1.5,0) {};
		\node[vertex] [label=left:$x$] (g) at (1.5,0.5) {};
		\path[edge] (b) to (c);
		\path[edge] (c) to (d);
		\path[edge,dotted] (a) to[out=180,in=70] (e);
		\path[edge] (e) to (f);
		\path[edge,dotted] (a) to[out=0,in=90] (g);
		\path[edge,dotted] (g) to[out=270,in=0] (d);
	\end{tikzpicture} & \Large{$\xrightarrow{\Mt}$} & 
      \begin{tikzpicture}[vertex/.style={circle, draw, inner sep=0pt, minimum size=6pt, fill=black}, edge/.style={draw}, baseline=0]
	   \node[vertex,fill=red,color=red] [label=above:$1$] (a) at (0,2) {};
		\node[vertex] (b) at (0,1) {};
		\node[vertex] (c) at (0,0) {};
		\node[vertex] [label=left:$1$] (d) at (0,-1) {};
		\node[vertex] [label=left:$1$] (e) at (-1.5,1) {};
		\node[vertex] (f) at (-1.5,0) {};
		\node[vertex] (g) at (1.5,0.5) {};
		\path[edge] (b) to (c);
		\path[edge] (c) to (d);
		\path[edge,dotted] (a) to[out=180,in=70] (e);
		\path[edge] (e) to (f);
		\path[edge] (a) to[out=0,in=90] (g);
		\path[edge,dotted] (g) to[out=270,in=0] (d);
	\end{tikzpicture} & \Large{$\xrightarrow{\Delta}$} &
     \begin{tikzpicture}[vertex/.style={circle, draw, inner sep=0pt, minimum size=6pt, fill=black}, edge/.style={draw}, baseline=0]
	   \node[vertex,fill=red,color=red] [label=above:$x$] (a) at (0,2) {};
		\node[vertex] (b) at (0,1) {};
		\node[vertex] (c) at (0,0) {};
		\node[vertex] [label=left:$1$] (d) at (0,-1) {};
		\node[vertex] [label=left:$1$] (e) at (-1.5,1) {};
		\node[vertex] (f) at (-1.5,0) {};
		\node[vertex] [label=left:$1$] (g) at (1.5,0.5) {};
		\path[edge] (b) to (c);
		\path[edge] (c) to (d);
		\path[edge,dotted] (a) to[out=180,in=70] (e);
		\path[edge] (e) to (f);
		\path[edge,dotted] (a) to[out=0,in=90] (g);
		\path[edge,dotted] (g) to[out=270,in=0] (d);
	\end{tikzpicture}
    \end{tabular}}
    
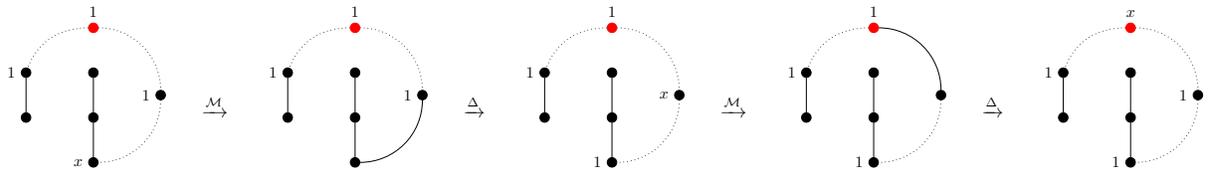
\captionof{figure}{A rooted positive subpath in $T_{4}(LLLdD\Bar{d}\Bar{D}\Bar{D})$ from Table \ref{spanning trees of 8_20}}
   \label{rooted positive subpath example}
   \end{table}

		\begin{proposition}
			 For a given fixed chain of spanning trees $\Ch= (T= T_1 > \cdots > T_n = T')$, $\Gamma(T^+,T'^-)$ can be computed entirely using $M(T_i)$ and $W(T_i)$ for all $i=1, \cdots, n$.
		\end{proposition}
		
		\begin{proof}
            \sloppy As before let $(e_i,f_i)$ be the pair of edge which differs $T_i$ from $T_{i+1}$ and $e_i$ be the dead edge whose change in resolution allows one to move from the $ith$ broken path to $(i+1)th$. For an alternating path $\Pt$ between $S_c^{T^+}$ and $S_c^{T'^-}$, there exists a unique $1 \leq j \leq n$ such that $\Pt$ travels from $\mathcal{H}_\Mt ((\CKh^+(U(T_j)),\partial_{\Kh}^+))$ to $\mathcal{H}_\Mt ((\CKh^-(U(T_j)),\partial_{\Kh}^-))$ or from $\mathcal{H}_\Mt ((\CKh^+(U(T_j)),\partial_{\Kh}^+))$ to $\mathcal{H}_\Mt ((\CKh^-(U(T_{j+1})),\partial_{\Kh}^-))$ for which $1 \leq j \leq n-1$. Proposition \ref{rooted negative subpaths} and \ref{rooted positive subpaths} helps to find such enhanced spanning subgraphs for such tree $T_j$. The rest of the argument for the description of $\Pt$ follows from Proposition \ref{combinatorial description for red complex}.
		\end{proof}
		
	\end{subsection}

	\begin{subsection}{Torsion in alternating links}
		
		A.N. Shumakovitch in \cite{Shumakovitch} conjectured that every non-split link except the trivial knot, the Hopf link, and their connected sums has 2-torsion in their Khovanov homology. Although this conjecture is yet open but some important special cases has been proven by Shumakovitch \cite{Shumakovitch}, and Asaeda and Przytycki in \cite{Asaeda} for alternating and almost-alternating links respectively.
		
		\begin{conjecture}\label{shuma_conj}
			Every non-split link except the trivial knot, Hopf link and their connected sums have $\Z_2$--torsion in their Khovanov homology.
		\end{conjecture}

        Shumakovitch partially answered conjecture \ref{shuma_conj} in affirmative for alternating links which are not connected sum of Hopf links. In this subsection, we provide an alternate proof of this partial answer. In other words, we prove Theorem \ref{torsion theorem} using the spanning tree complex of alternating links and the structure of their associated Tait graph. Before moving on, let us discuss some preliminary results from graph theory.
        
		\begin{definition}
			Let $H$ be a subgraph of a graph $G$. An \textit{ear} of $F$ in $G$ is a non-trivial path whose ends lie in $H$ but its internal vertices do not.
		\end{definition}
		
		\begin{definition}
			A \textit{nested sequence} of graphs is a sequence $(G_0,G_1,\cdots,G_k)$ such that $G_i \subset G_{i+1}$ for all $0 \leq i < k$. An \textit{ear decomposition} of a non-separable graph $G$ is a nested sequence $(G_0,G_1,\cdots,G_k)$ of non-separable subgraphs such that 
			
			\begin{enumerate}
				\item $G_0$ is a cycle,
				\item $G_{i+1} = G_i \cup P_i$ where $P_i$ is an ear of $G_i$ in $G$ for all $0 \leq i <k$,
				\item $G_k=G$.
			\end{enumerate}
		\end{definition}

	\begin{proof}[Proof of Theorem \ref{torsion theorem}]
	 
	 It is enough to prove the theorem for a non-split prime alternating link. Let $\Lk$ be an non-split prime alternating connected link diagram. Assume that all the edges of $G_\Lk$ are positive. As $G_\Lk$ is $2$--connected, we can find an ear decomposition of $G_\Lk$ with $G_0$ being a cycle of length atleast 3 since $\Lk$ is neither unknot nor the connected sum of Hopf links. We provide a specific ordering on edges of $G_\Lk$. First we order all edges of $G_0$, then we sequentially order all the edges of $P_i$ for all $0 \leq i <k$ such that $e_i < e_j$ whenever $e_i \in P_i$, $e_j \in P_j$ and $i<j$. Choose any vertex of $G_0$ to be considered as a root.\\
	 
	 Remove the largest edge in $e \in G_0$, and sequentially remove the largest edge in each $P_i$ to get a tree $T$ such that $|A(H_c,\epsilon_c)|$ for $T$ is minimum among all spanning trees of $G_\Lk$.  All the internal edges of $T$ have activity $L$ and all the external edges have activity $d$. Remove the second largest edge $f \in G_0$ and insert $e$ to get another tree $T'$.\\
	 
	 Since $G_0$ is of length atleast 3, there is still an edge in $G_0 \cap T'$ with activity $L$. Now $a_{T'}(e)=D$ and $a_{T'}(f)=d$, there is only one chain possible starting with $T'$ which is $T' > T$. Since, the differential $\partial_{ST}^+$ in the reduced spanning tree complex $\CST^+(G_\Lk)$ is trivial, so alternating path only exists between $T'^+$ and $T^-$. All possible alternating paths are shown in Figure \ref{alt-trees-path}. Let $|E(G_0)| = n$, then we have $\partial_{ST}(T'^+) = (-1)^{n-1}. 2T^-$. The only thing required to be shown is that $T^-$ is not exact.\\
	 
	 Observe that any tree $\Tilde{T}$ with $i(\Tilde{T}^+)=i(T'^+)$ having chains beginning with $\Tilde{T}^+$ always conatins tree with higher homological grading than $\Tilde{T}^+$, thus the only chain which is feasible in this case is $\Tilde{T} > T$, where $W(\Tilde{T})$ and $W(T)$ exactly differ at two edges. Thus, for any chain element $X=\sum_{i(\Tilde{T})=i(T)-1,\  \Tilde{T}>T} \Tilde{T}$ we will similarly have $\partial_{ST}(X)=2c.T^- $ for some constant $c \in \Z$.
	 
	\end{proof} 

      \begin{table}[ht]
      	\centering
      	\setlength{\tabcolsep}{4mm}
      	\def\arraystretch{1.5}
      	\resizebox{0.85\textwidth}{!}{
      		\begin{tabular}{ccccc}
      			\begin{tikzpicture}[vertex/.style={circle, draw, inner sep=0pt, minimum size=6pt, fill=black}, edge/.style={draw}, baseline=0]
      				\node[vertex,fill=red,color=red] at (0,0) (a) {};
      				\node[vertex] at (0.5,-1) (b) {};
      				\path[edge,dashed] (a) to node[midway,left]  {$1\:L$} (b);
      				\node[vertex] at (1,1) (c) {};
      				\node[vertex] at (2,0.2) (d) {};
      				\node at (1,-0.2) (e) {$G_0$};
      				\node at (-1,0.2) (f) {};
      				\node at (1.4,2) (g) {};
      				\node at (0,-2) (h) {};
      				\node at (3,0.4) (i) {};
      				\path[edge,dashed] (a)--(f) {};
      				\path[edge,dashed] (b)--(h) {};
      				\path[edge,dashed] (c)--(g) {};
      				\path[edge,dashed] (d)--(i) {};
      				\path[edge,dashed] (a) to node[midway,above,xshift=-12pt,yshift=3pt] {$n\: D(e)$} (c);
      				\path[] (c) to node[right,xshift=3pt,yshift=8pt] {$(n-1)\:d(f)$} (d);
      				\path[edge,dashed] (b) to[out=320,in=300] node[below,yshift=-6pt] {$2\:L\:\cdots$} (d);
      				\node at (1.3,-2.2) {$T'^+$};
      			\end{tikzpicture} & \begin{tikzpicture}
      				\draw[->,very thick,color=blue] (0,0) -- node[midway,above,color=black] {$\Delta$} (1,0);
      			\end{tikzpicture} &
      			\begin{tikzpicture}[vertex/.style={circle, draw, inner sep=0pt, minimum size=6pt, fill=black}, edge/.style={draw}, baseline=0]
      				\node[vertex,fill=red,color=red] at (0,0) (a) {};
      				\node[vertex] at (0.5,-1) (b) {};
      				\path[edge,dashed] (a) to node[midway,left]  {$1\:L$} (b);
      				\node[vertex] at (1,1) (c) {};
      				\node[vertex] at (2,0.2) (d) {};
      				\node at (1,-0.2) (e) {$G_0$};
      				\node at (-1,0.2) (f) {};
      				\node at (1.4,2) (g) {};
      				\node at (0,-2) (h) {};
      				\node at (3,0.4) (i) {};
      				\path[edge,dashed] (a)--(f) {};
      				\path[edge,dashed] (b)--(h) {};
      				\path[edge,dashed] (c)--(g) {};
      				\path[edge,dashed] (d)--(i) {};
      				\path[] (a) to node[midway,above,xshift=-12pt,yshift=3pt] {$n\: d(e)$} (c);
      				\path[edge,dashed] (c) to node[right,xshift=3pt,yshift=8pt] {$(n-1)\:L(f)$} (d);
      				\path[edge,dashed] (b) to[out=320,in=300] node[below,yshift=-6pt] {$2\:L\:\cdots$} (d);
      				\node at (1.3,-2.2) {$T^-$}; 
      			\end{tikzpicture} & \begin{tikzpicture}
      			\draw[<-] (0,0) -- node[midway,above,color=black] {$\Delta$} (1,0);
      			\end{tikzpicture} & 
      			\begin{tikzpicture}[vertex/.style={circle, draw, inner sep=0pt, minimum size=6pt, fill=black}, edge/.style={draw}, baseline=0]
      				\node[vertex,fill=red,color=red] at (0,0) (a) {};
      				\node[vertex] at (0.5,-1) (b) {};
      				\path[edge,very thick, color=red] (a) to node[color=black,midway,left]  {$1\:L$} (b);
      				\node[vertex] at (1,1) (c) {};
      				\node[vertex] at (2,0.2) (d) {};
      				\node at (1,-0.2) (e) {$G_0$};
      				\node at (-1,0.2) (f) {};
      				\node at (1.4,2) (g) {};
      				\node at (0,-2) (h) {};
      				\node at (3,0.4) (i) {};
      				\path[edge,dashed] (a)--(f) {};
      				\path[edge,dashed] (b)--(h) {};
      				\path[edge,dashed] (c)--(g) {};
      				\path[edge,dashed] (d)--(i) {};
      				\path[] (a) to node[midway,above,xshift=-12pt,yshift=3pt] {$n\: d(e)$} (c);
      				\path[edge,dashed] (c) to node[right,xshift=3pt,yshift=8pt] {$(n-1)\:L(f)$} (d);
      				\path[edge,dashed] (b) to[out=320,in=300] node[below,yshift=-6pt] {$2\:L\:\cdots$} (d);
      				\node at (1.3,-2.2) {$T^-$}; 
      			\end{tikzpicture} \\
      			&&&& \\
      			\begin{tikzpicture}
      				\draw[->,very thick,color=blue] (0,0) -- node[midway,right,color=black] {$\Delta$} (0,-1);
      			\end{tikzpicture} & & & & \begin{tikzpicture}
      			\draw[->,very thick,color=blue] (0,0) -- node[midway,right,color=black] {$\Mt$} (0,1);
      			\end{tikzpicture} \\ 
      			\begin{tikzpicture}[vertex/.style={circle, draw, inner sep=0pt, minimum size=6pt, fill=black}, edge/.style={draw}, baseline=0]
      				\node[vertex,fill=red,color=red] at (0,0) (a) {};
      				\node[vertex] at (0.5,-1) (b) {};
      				\path[edge,dashed] (a) to node[midway,left]  {$1\:L$} (b);
      				\node[vertex] at (1,1) (c) {};
      				\node[vertex] at (2,0.2) (d) {};
      				\node at (1,-0.2) (e) {$G_0$};
      				\node at (-1,0.2) (f) {};
      				\node at (1.4,2) (g) {};
      				\node at (0,-2) (h) {};
      				\node at (3,0.4) (i) {};
      				\path[edge,dashed] (a)--(f) {};
      				\path[edge,dashed] (b)--(h) {};
      				\path[edge,dashed] (c)--(g) {};
      				\path[edge,dashed] (d)--(i) {};
      				\path[] (a) to node[midway,above,xshift=-12pt,yshift=3pt] {$n\: d(e)$} (c);
      				\path[edge,dashed,very thick,color=red] (c) to node[color=black,right,xshift=3pt,yshift=8pt] {$(n-1)\:L(f)$} (d);
      				\path[edge,dashed] (b) to[out=320,in=300] node[below,yshift=-6pt] {$2\:L\:\cdots$} (d);
      				\node at (1.3,-2.2) {$T^-$}; 
      			\end{tikzpicture} & \begin{tikzpicture}
      				\draw[->] (0,0) -- node[midway,above,color=black] {$\Mt$} (1,0);
      			\end{tikzpicture} &
      			\begin{tikzpicture}[vertex/.style={circle, draw, inner sep=0pt, minimum size=6pt, fill=black}, edge/.style={draw}, baseline=0]
      				\node[vertex,fill=red,color=red] at (0,0) (a) {};
      				\node[vertex] at (0.5,-1) (b) {};
      				\path[edge,dashed] (a) to node[midway,left]  {$1\:L$} (b);
      				\node[vertex] at (1,1) (c) {};
      				\node[vertex] at (2,0.2) (d) {};
      				\node at (1,-0.2) (e) {$G_0$};
      				\node at (-1,0.2) (f) {};
      				\node at (1.4,2) (g) {};
      				\node at (0,-2) (h) {};
      				\node at (3,0.4) (i) {};
      				\path[edge,dashed] (a)--(f) {};
      				\path[edge,dashed] (b)--(h) {};
      				\path[edge,dashed] (c)--(g) {};
      				\path[edge,dashed] (d)--(i) {};
      				\path[] (a) to node[midway,above,xshift=-12pt,yshift=3pt] {$n\: d(e)$} (c);
      				\path[edge,very thick,color=red] (c) to node[color=black,right,xshift=3pt,yshift=8pt] {$(n-1)\:L(f)$} (d);
      				\path[edge,dashed] (b) to[out=320,in=300] node[below,yshift=-6pt] {$2\:L\:\cdots$} (d);
      				\node at (1.3,-2.2) {$T^-$}; 
      			\end{tikzpicture} & $\boldsymbol{\cdots\cdots}$ &
      			\begin{tikzpicture}[vertex/.style={circle, draw, inner sep=0pt, minimum size=6pt, fill=black}, edge/.style={draw}, baseline=0]
      				\node[vertex,fill=red,color=red] at (0,0) (a) {};
      				\node[vertex] at (0.5,-1) (b) {};
      				\path[edge,dashed,very thick,color=red] (a) to node[color=black,midway,left]  {$1\:L$} (b);
      				\node[vertex] at (1,1) (c) {};
      				\node[vertex] at (2,0.2) (d) {};
      				\node at (1,-0.2) (e) {$G_0$};
      				\node at (-1,0.2) (f) {};
      				\node at (1.4,2) (g) {};
      				\node at (0,-2) (h) {};
      				\node at (3,0.4) (i) {};
      				\path[edge,dashed] (a)--(f) {};
      				\path[edge,dashed] (b)--(h) {};
      				\path[edge,dashed] (c)--(g) {};
      				\path[edge,dashed] (d)--(i) {};
      				\path[] (a) to node[midway,above,xshift=-12pt,yshift=3pt] {$n\: d(e)$} (c);
      				\path[edge,dashed] (c) to node[right,xshift=3pt,yshift=8pt] {$(n-1)\:L(f)$} (d);
      				\path[edge,dashed] (b) to[out=320,in=300] node[below,yshift=-6pt] {$2\:L\:\cdots$} (d);
      				\node at (1.3,-2.2) {$T^-$}; 
      			\end{tikzpicture} 
     	\end{tabular}}
     
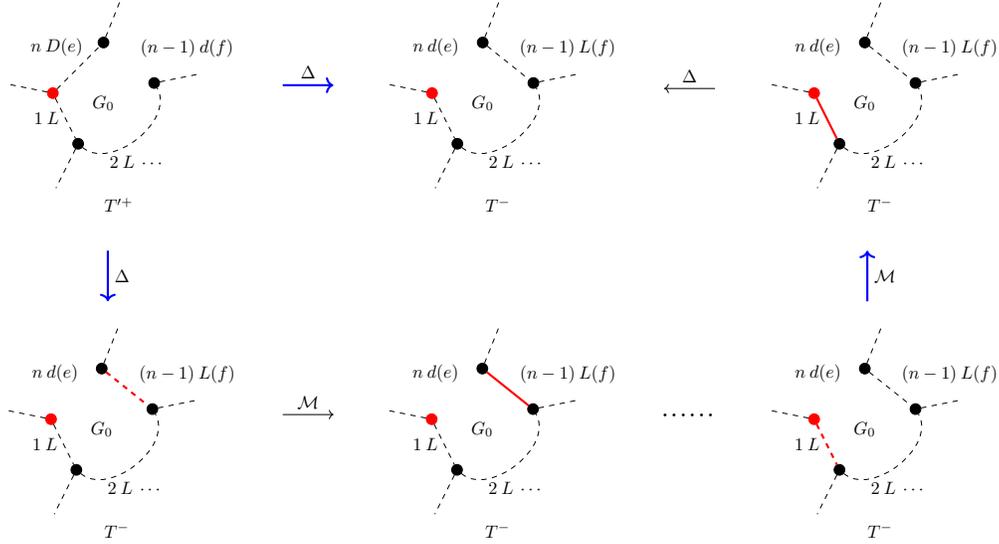
\captionof{figure}{Alternating paths between $T^+$ and $T'^-$}
     \label{alt-trees-path}
      \end{table}	

    \end{subsection}
    
   \end{section}


\begin{section}{Rasmussen's s invariant}	\label{s-invariant}
	
	\begin{subsection}{Acyclic Matching on the unreduced Lee complex}
		
		The generating set for both Khovanov and Lee complex coincides. Thus, we still can use the fact that Lee complex can be partitioned into complexes of the twisted unknots for each spanning tree of the Tait graph. Hence, we can use Theorem \ref{matching_theorem} to provide the same matching to Lee complex because the edge set of the Hasse diagram of Khovanov complex is contained in the edge set of the Hasse diagram of Lee complex.
		
		\begin{proposition}\label{acyclic matching for Lee}
			Let $T$ be a spanning tree of the Tait graph $G_\Lk$ associated to a connected link diagram $\Lk$, then there exists an acyclic near-perfect matching on $\mathcal{H}(({\CLee}(U(T)),\partial_{Lee}))$, where $(\CLee(U(T)),\partial_{Lee})$ is the unreduced Lee complex for $U(T)$. Furthermore, by taking the union of these matchings over all spanning trees, we get an acyclic matching $\Mt^{Lee}_\Lk$ on $\mathcal{H}(({\CLee}(\Lk),\partial_{Lee}))$.
		\end{proposition}
		
		\begin{proof}
		   \sloppy  From the definition of Lee complex, we know that $\mathcal{H}((\CKh(U(T)),\partial_{\Kh}))$ is a subgraph of $\mathcal{H}((\CLee(U(T),\partial_{Lee})))$. Thus, we can take $\Mt^{Lee}_{\Lk}$ to be the same matching as $\Mt^{\Kh}_\Lk$ as defined in Theorem \ref{matching_theorem} and it follows that it is near-perfect as well in Lee complex. For any directed cycle to occur in $\mathcal{H}_\Mt((\CLee(U(T),\partial_{Lee}))$, all the enhanced spanning subgraphs in that cycle must have the same $j$--grading, thus any directed cycle in $\mathcal{H}_\Mt((\CLee(U(T)),\partial_{Lee}))$ is also a directed cycle in $\mathcal{H}_\Mt((\CKh(U(T)),\partial_{\Kh}))$ which is not possible. Thus, $\Mt^{Lee}_\Lk$ is acyclic.   
		\end{proof}

        \begin{definition}
			A \textit{generalized rooted positive subpath} is a rooted positive subpath of $G(T)$ where $\epsilon(\{v_C\}) \in \{1,x\}$ for every connected component $C$ of the rooted positive subpath.
		\end{definition}
		
		\begin{proposition}\label{generalized rooted pos subpath}
			For a given spanning tree $T$ of $G_\Lk$, the following two sets are in one-one correspondence:
			
			\[\begin{tikzcd}[ampersand replacement=\&,cramped]
				{\left\{S_x^{T} \mid \A^\uparrow_{Lee}(S_x^{T},S_c^{T^+})=1, i(S_c^{T^+})=i(S_x^{T})\right\}} \& {\left\{\text{Generalized rooted positive subpaths of }G(T)\right\}}
				\arrow[shorten <=5pt, shorten >=5pt, tail reversed, from=1-1, to=1-2]
			\end{tikzcd}\]

            where $\A^\uparrow_{Lee}(S_x^T,S_y^T)$ is the number of alternating paths from $S_x^T$ to $S_y^T$ in $\mathcal{H}_\Mt((\CLee(U(T)),\partial_{Lee}))$, where the path begins with a matched edge.
		\end{proposition}
		
		\begin{proof}
			We follow the proof of Proposition \ref{rooted positive subpaths} with an extra addition to detail. In case of Lee complex, a split map at an edge $e=\{v_1,v_2\} \in G(T)$ with activity $L/\Bar{l}$ can also occur with both $\epsilon(C_{v_1})=\epsilon(C_{v_2})=1$ (See equation \ref{lee diff split}), where $C_{v_i}$ is the component containing the vertex $v_i$ for $i=1,2$. Thus, the reason for taking generalized rooted positive subpaths is justified.
		\end{proof}
		
		 A similar result also holds for enhanced spanning subgraphs with $\A^\uparrow_{Lee}(S_x^{T},S_c^{T^-})=1$ and moreover, one can verify that 
		 
		 \begin{equation} \label{equal number of paths}
		     \bigg|\left\{S_x^{T} \mid \A^\uparrow_{Lee}(S_x^{T},S_c^{T^+})=1, i(S_c^{T^+})=i(S_x^{T})\right\}\bigg| = \bigg|\left\{S_x^{T} \mid \A^\uparrow_{Lee}(S_x^{T},S_c^{T^-})=1, i(S_c^{T^-})=i(S_x^{T})\right\}\bigg| 
		 \end{equation}

	\end{subsection}   
	
	\begin{subsection}{Orientation preserving tree}
		Given an oriented connected knot diagram $K$, Rasmussen showed that the homology class of orientation preserving resolution $\mathfrak{s}_o$ of $K$ survives under any cobordism of $K$, with no closed components, in the Lee cohomology. We want to express these generators using the critical enhanced spanning subgraphs of some spanning tree of $G_K$ which we call the \textit{orientation preserving tree} and denote it by $\mathfrak{T}_o$. The construction of $\mathfrak{T}_o$ depends on a specific ordering of the edges of $G_K$.
		
	\begin{construction}
		Suppose we have a fixed orientation $o$ on $K$, then consider the Tait graph $G_K$ on top of the checkerboard colored knot diagram $K$. We color an edge red with an arrow pointing towards the orientation $o$ if near its corresponding crossing the local picture looks like type I configuration or else we leave the edge uncolored for a type II configuration (See Figure \ref{directed tait graph}).
		
		\begin{figure}[ht]
			\centering
			\includegraphics[width=0.5\textwidth]{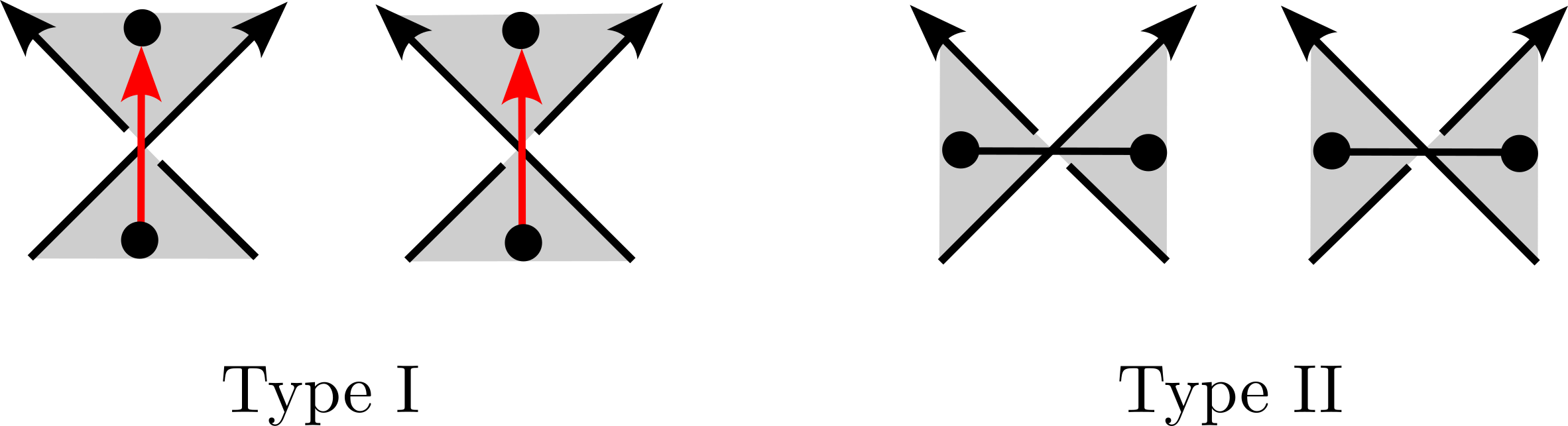}
			\caption{}
			\label{directed tait graph}
		\end{figure}
		
        Here are some key observations which we obtain after this whole decoration of $G_K$.
		
		\begin{enumerate}
			\item \label{observe 1} There does not exist a red edge which is not a part of a red cycle. Suppose this is false, then there exist a red edge $e$ whose one end vertex $v$ has all black edges incident to it. (such a vertex always exists since, vertices of red edges cannot be a leaf of $G_K$). Now in the state $\mathfrak{s}_o$, these black edges will not be contained inside the shaded region and hence the orientation on the boundary of the  shaded region containing $v$ will not match with $o$ chosen for $K$.

            \item \label{observe 2} A minimal cycle (cycle which does not contain any other subgraph inside its face) containing a black edge must contain even number of black edges because the arcs of the circle components passing around the end vertices of a black edge have orientation pointing in the same direction when seen in $\mathfrak{s}_o$. Figure \ref{obs} provides a pictorial representation of the above observation.

         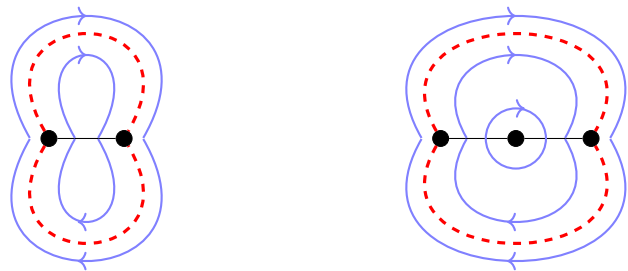
\begin{figure}[ht]
                \centering
                \subfigure[One cannot assign any crossing to the black edge due to the orientation of the inner blue circle.]{\begin{tikzpicture}[vertex/.style={circle, draw, inner sep=0pt, minimum size=6pt, fill=black}, edge/.style={draw}, baseline=0]
    \node[vertex] at (0,0) (a) {};
    \node[vertex] at (1,0) (b) {};
    \draw (a)--(b);
    \path[edge,dashed,red,very thick] (a) to[out=120,in=60,distance=2cm] (b);
    \path[edge,dashed,red,very thick] (a) to[out=240,in=300,distance=2cm] (b);
    \node[inner sep=0pt, minimum size=0pt] at (-0.25,0) (c) {};
    \node[inner sep=0pt, minimum size=0pt] at (0.35,0) (d) {};
    \node[inner sep=0pt, minimum size=0pt] at (0.65,0) (e) {};
    \node[inner sep=0pt, minimum size=0pt] at (1.25,0) (f) {};
    \draw[blue!50,thick,decoration={markings, mark=at position 0.5 with {\arrow{>}}},
        postaction={decorate}] (c) to[out=120,in=60,distance=2.5cm,looseness=1.5] (f);
    \draw[blue!50,thick, decoration={markings, mark=at position 0.5 with {\arrow{<}}},
        postaction={decorate}] (c) to[out=240,in=300,distance=2.5cm,looseness=1.5] (f);
    \draw[blue!50,thick, decoration={markings, mark=at position 0.5 with {\arrow{>}}},
        postaction={decorate}] (d) to[out=120,in=60,distance=1.7cm,looseness=1.5] (e);
    \draw[blue!50,thick, decoration={markings, mark=at position 0.5 with {\arrow{<}}},
        postaction={decorate}] (d) to[out=240,in=300,distance=1.7cm,looseness=1.5] (e);
\end{tikzpicture}}
\qquad\quad
\subfigure[The orientation on each blue circle allows one to provide any suitable crossing for each black edge.]{\begin{tikzpicture}[vertex/.style={circle, draw, inner sep=0pt, minimum size=6pt, fill=black}, edge/.style={draw}, baseline=0]
    \node[vertex] at (0,0) (a) {};
    \node[vertex] at (1,0) (b) {};
    \node[vertex] at (2,0) (c) {};
    \draw (a)--(b);
    \draw (b)--(c);
    \path[edge,dashed,red,very thick] (a) to[out=120,in=60,distance=2cm] (c);
    \path[edge,dashed,red,very thick] (a) to[out=240,in=300,distance=2cm] (c);
    \node[inner sep=0pt, minimum size=0pt] at (-0.25,0) (d) {};
    \node[inner sep=0pt, minimum size=0pt] at (0.35,0) (e) {};
    \node[inner sep=0pt, minimum size=0pt] at (1.65,0) (h) {};
    \node[inner sep=0pt, minimum size=0pt] at (2.25,0) (i) {};
    \draw[blue!50,thick,decoration={markings, mark=at position 0.5 with {\arrow{>}}},
        postaction={decorate}] (d) to[out=120,in=60,distance=2.5cm,looseness=1.5] (i);
    \draw[blue!50,thick, decoration={markings, mark=at position 0.5 with {\arrow{<}}},
        postaction={decorate}] (d) to[out=240,in=300,distance=2.5cm,looseness=1.5] (i);
    \draw[blue!50,thick, decoration={markings, mark=at position 0.5 with {\arrow{>}}},
        postaction={decorate}] (e) to[out=120,in=60,distance=1.7cm,looseness=1.5] (h);
    \draw[blue!50,thick, decoration={markings, mark=at position 0.5 with {\arrow{<}}},
        postaction={decorate}] (e) to[out=240,in=300,distance=1.7cm,looseness=1.5] (h);
    \path[draw=blue!50,thick,decoration={markings, mark=at position 0.25 with {\arrow{<}}},
        postaction={decorate}] (b) circle[radius=0.4];
\end{tikzpicture}}
               \caption{The blue circles around the Tait graph represents the state $\mathfrak{s}_o$}
                \label{obs}
            \end{figure}
			
        \item \label{observe 3} Consider the dual graph $G_K^*$ of $G_K$. The dual edge of each red edge in $G_K$ becomes black in $G_K^*$ while the dual edge of each black edge in $G$ becomes red in $G_K^*$. Thus, a similar phenomenon occurs for black edges in $G_K^*$ which occurs in (\ref{observe 2}) for black edges in $G_K$.
		\end{enumerate}
    
     Let $H_1$ be the subgraph which contains all the red edges of $G_K$. For each connected component $C \in C(H_1)$, remove positive red edges from $C$ until it becomes minimally connected. Then remove negative red edges from $C$ until it becomes minimally connected. This gives a new subgraph $H_2$. Now consider the spanning subgraph $H_3$ which contains all the black edges of $G_K$ and $H_2$. First remove positive black edges until $H_3$ becomes minimally connected and then remove negative black edges until it becomes minimally connected. This gives us the final spanning subgraph $H_4$ which is a spanning tree of $G_K$. Now we order the edges of $G_K$ in the following sequence: 
		
	$$
		 \{\text{red edges not in }H_4\} < \{\text{black edges in } H_4\} < \{\text{red edges in } H_4\} < \{\text{black edges not in }H_4\}
	$$
		
		Thus, we have a spanning tree $\To$ where the removed red edges are externally active, the removed black edges are externally inactive, the red edges in $\To$ are internally inactive and the black edges in $\To$ are internally active. Hence, the critical enhanced spanning subgraphs of $\mathfrak{T}_o$ corresponds to the  resolved states $\mathfrak{s}_o$ due to Proposition \ref{span graph and res state}. An example of this construction is illustrated for the knot $K$ (See Figure \ref{exotic-disk-oriented-tree}) in Figure \ref{exotic_disks} with a chosen orientation.
				
	  \end{construction}

		\begin{figure}[ht]
			\centering
			\includesvg[width=0.5\textwidth,inkscape=force]{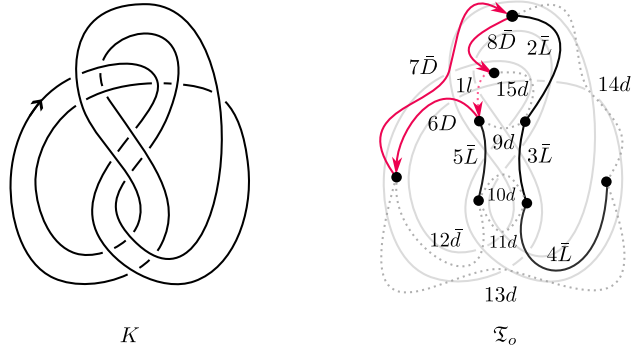}
			\caption{Orientation preserving tree of the knot $K$ in Figure \ref{exotic_disks}}
			\label{exotic-disk-oriented-tree}
		\end{figure}

        Let $\CST_{Lee}(K)$ denote the spanning tree complex corresponding to the Lee complex of $K$. We will now show that $\To$ is a cycle in $\CST_{Lee}(K)$ by showing that any chain $\Ch = \To > T_1 >T_2 \cdots > T_n$ beginning with $\To$ has the property that for every tree $T_k \in \Ch$, $i(T_k)-i(T_{k-1}) \geq 1$ and $j(T_k)-j(T_{k-1}) \geq 2$ which reduces the possibility to $\Ch=\To > T_1$ with $\partial_{ST}(\To^+)=\Gamma(\To^+,T_1^-).T_1^-$, but we will also show that $\Gamma(\To^+,T_1^-)=0$. 

          \begin{proposition}\label{length 1 chain}
              Let $\Ch=\To > T_1$ with $i(T_1)-i(\To)=1$ and $j(T_1)-j(\To)=2$, then $\Gamma(\To^+,T_1^-)=0$.
          \end{proposition}

          \begin{proof}
              Assume that $a_{\To}(f_1)=\Bar{d}$ and $a_{T_1}(f_1)=\Bar{D}$. Observe that, $cyc(\To,f_1)$ contains edges with activity $\Bar{L},D,\Bar{D}$ only and if $a_{\To}(e_1) \in \{D,\Bar{D}\}$, then $a_{T_1}(e_1) \in \{d,\Bar{d}\}$ since, $cyc(\To,f_1)$ contains atleast one black edge other than $f_1$. Hence, $a_{\To}(e_1) = \Bar{L}$. Moreover, $e_1$ has to be the largest black edge in $cyc(\To,f_1) \setminus f_1$, otherwise we will have $i(T_1) -i(\To) > 1$. Thus, the only possible change between $W(\To)$ and $W(T_1)$ is $\Bar{L}\Bar{d} \rightarrow \Bar{d}\Bar{D}$. Now any alternating path between $S_c^{\To^+}$ and $S_c^{T_1^-}$ cannot contain a non-matched directed edge which increases the $j-$grading by 4. So, any alternating path between these two critical enhanced spanning subgraphs can be seen in $\mathcal{H}_\Mt((\CKh(K),\partial_{\Kh}))$.\\
              Now from Proposition \ref{rooted negative subpaths} and \ref{rooted positive subpaths}, we know that $S_0=E_0=S_c^{\To^+}$ and $S_1=E_1=S_c^{T_1^-}$ as spanning subgraphs. Suppose $e_1=\{v_1,v_2\}$ and $P_{v_i}$ be the unique path from $v_d$ to $v_i$ in $G(\To)$. Then in order to move from $\mathcal{H}_\Mt ((\CKh(U(\To^+)),\partial^+_{\Kh}))$ to $\mathcal{H}_\Mt ((\CKh(U(T_1^-),\partial^-_{\Kh}))$ through an alternating path, one of vertices $v_1$ or $v_2$ must have $1$ as its enhancement in $E_0$ and $\epsilon(\{v_d\})=x$. Thus, for finding $\epsilon(E_0)$, we only have two proper rooted negative subpaths, one from each $P_{v_i}$. Now for these two choices of $E_0$ we will have two choices of $S_1$ and thus, one can show that there will be exactly two alternating paths. Suppose $p=|e \mid e \in P_{v_1} \cap P_{v_2}, e \text{ is a positive twist }|$, $q=|e \mid e \in P_{v_1} \cap P_{v_2}, e \text{ is a negative twist }|$, $r_1=|e \mid e \in P_{v_1}\setminus P_{v_2}, \ a_{\To}(e)=\Bar{L}|$ and $r_2=|e \mid e \in P_{v_2}\setminus P_{v_1}, \ a_{\To}(e)=\Bar{L}|$. Then, one can easily verify that

              $$
               \Gamma(\To^+,T_1^-)=(-1)^{s(f_1)+p+q}\left((-1)^{r_1} + (-1)^{r_2}\right)
              $$

              Now by observation (\ref{observe 2}) $cyc(\To,f_1)$ contains even number of black edges hence, $r_1+r_2+1$ is even. So, $\Gamma(\To^+,T_1^-)=0$. When $a_{\To}(f_1)=D$ and $a_{T_1}(f_1)=d$ then, consider the dual chain $\To^* > T_1^*$, where $a_{\To^*}(f_1^*)=\Bar{d}$ and $a_{T_1^*}(f_1^*)=\Bar{D}$ and similarly we have  $\Gamma({\To^*}^+,{T_1^*}^-)=0$ and thus, $\Gamma(\To^+,T_1^-)=0$
              \end{proof}
		
		 \begin{proposition}\label{higher hom than To}
			$\To$ is a cycle in $\CST_{Lee}(K)$ with $i(\To^\pm)=0$.
		\end{proposition}
		
		\begin{proof}
              Suppose $\Ch = \To > T_1 >T_2 \cdots > T_n$ be a chain beginning with $\To$. Any tree $T_i$ in $\Ch$ is obtained from $T_{i-1}$ is obtained by switching a pair of edges $(e_i,f_i)$, where $e_i \in T_{i-1}$ and $f_i \not\in T_{i-1}$ and one of them is dead edge with activity either $\Bar{d}$ or $D$. We treat these two cases separately.\\

              \textbf{case 1:} Suppose $a_{T_{i-1}}(f_i) = \Bar{d}$ and $a_{T_i}(f_i)=\Bar{D}$. Assume that $f_i$ is a red edge then, $a_{\To}(f_i)=\Bar{l} / \Bar{D}$. Now there exists $g \in cyc(T_{i-1},f_i)$ such that $g < f_i$. If $a_{\To}(f_i)=\Bar{l}$, then $a_{\To}(g)=\Bar{D}$ which implies $f_i < g$. Thus, $a_{\To}(f_i) = \Bar{D}$ and this implies that there exists $1 <k < k' < i-1$ such that $f_i$ changes its activity from $\Bar{D}$ to $\Bar{L}$ while moving from $T_{k-1}$ to $T_k$ and $f_i$ changes its activity from $\Bar{L}$ to $\Bar{d}$ while moving from $T_{k'-1}$ to $T_{k'}$ but this cannot occur since $cut(T_{k-1},f_i)$ contains an edge with activity $\Bar{l}$. Another possibility is that $f_i$ changes from $\Bar{D}$ to $\Bar{l}$ and then to $\Bar{d}$ but again this is not possible since, changing from $\Bar{D}$ to $\Bar{l}$ is only possible if there is an edge $g' \in cut(T_{k-1},f_i)$ with activity $\Bar{d}$ which is a black edge and hence there will be another black edge (due to observation (\ref{observe 2})) $g'' \in cyc(T_{k-1},g')$ such that $g'' < f_i$. Thus, $f_i$ has to be a black edge.\\
              Now if $a_{\To}(f_i)=\Bar{d}$ then all the black edges in $cyc(T_{i-1},f_i)$ must be negative and $e_i < f_i$ according to the construction of $\To$ Thus, if $e_i$ is black then $a_{T_{i-1}}(e_i)=\{\Bar{L},\Bar{D}\}$. Hence, $a_{T_i}(e_i)=\{\Bar{l},\Bar{d}\}$. If $e_i$ changes from $\Bar{L}$ to $\Bar{l}$, then $i(T_i) - i(T_{i-1}) \geq 2$ and thus, $j(T_i)-j(T_{i-1}) \geq 4$. If $e_i$ changes from $\Bar{L}$ to $\Bar{d}$ then $i(T_i) - i(T_{i-1}) \geq 1$ and thus, $j(T_i)-j(T_{i-1}) \geq 2$. If $e_i$ changes from $\Bar{D}$ to $\Bar{l}$, then similarly we will have $i(T_i) - i(T_{i-1}) \geq 1$ and $j(T_i)-j(T_{i-1}) \geq 2$. Now, if $e_i$ is red then, $a_{T_{i-1}}(e_i)=\{D,\Bar{D}\}$ and $a_{T_i}(e_i)=\{l,\Bar{l}\}$ but this scenario is not possible since, there is another black edge other than $f_i$ in $cyc(T_{i-1},f_i)$ which is smaller than $e_i$. Now if $a_{\To}(f_i)=\Bar{L}$, then there is a black edge $e_i' \in cyc(T_{i-1},f_i)$ such that $e_i'< f_i$ and there is an edge in $cut(\To,f_i)$ with activity $\Bar{d}$ because $a_{T_{i-1}}(f_i)=\Bar{d}$. Thus, again all the black edges in $cyc(T_{i-1},f_i)$ must be negative and hence, the rest of the argument follows similarly.\\
			
			\textbf{case 2:} Suppose $a_{T_{i-1}}(e_i) = D$ and $a_{T_i}(e_i)=d$. Then we have $a_{T^*_{i-1}}(e_i^*)=\Bar{d}$ and $a_{T^*_i}(e_i^*)=\Bar{D}$ since, $cut(T,e)=cyc(T^*,e^*)$. Now it follows from case 1 and observation (\ref{observe 3}) that $i(T_{i-1}^*) - i(T_i^*) \geq 1$ and $j(T_i^*)-j(T_{i-1}^*) \geq 2$ in all scenarios. Thus, these inequalities also holds for $T_{i-1}$ and $T_i$ since, $T_{i-1} < T_i$ if and only if $T_{i-1}^* < T_i^*$.\\
			
			Thus, in both the cases, $i(T_{i-1}) < i(T_i)$ and $j(T_{i-1})<j(T_i)$ for all $i=1,\cdots,n$ with $T_0 = \To$. For chains $\Ch = \To > T_1$, the proof follows from Proposition \ref{length 1 chain}. $i(\To^\pm)=0$ follows from Figure \ref{directed tait graph}.
		\end{proof}

        Recall from \cite{rus} that $s_{\max}(K)$ and $s_{\min}(K)$ are the $j-$gradings of the two copies of $\Q$ surviving in the $E_\infty$ term of the spectral sequence defined by Lee in \cite{LEE2005554} and Rasmussen showed that $s_{\max}(K)=s_{\min}(K)+2$ and finally defined the knot invariant
        $$
         s(K)=s_{max}(K)-1 = s_{\min}(K)+1
        $$
        Denote $s$ to be the $j-$grading in $H^*(\CLee(K))$ induced from the usual $j-$grading in $\CLee(K)$. Then, Rasmussen showed that the s grading of the generators $[\mathfrak{s}_o \pm \mathfrak{s}_{\Bar{o}}]$ attain either of the values in $\{s_{\max}(K),s_{\min}(K)\}$. We now show that the generators $[\mathfrak{s}_o \pm \mathfrak{s}_{\Bar{o}}] \in H^*(\CLee(K))$ corresponds to $[\To^\pm] \in \CST_{Lee}(K)$. 

        \begin{proposition}\label{oriented tree generators}
			$g(\mathfrak{s}_o \pm \mathfrak{s}_{\Bar{o}})$ is a non-zero multiple of $ \mathfrak{T}_o^{\pm}$, where $g$ is the map in equation \ref{g map in dmt} for the Lee complex.
		\end{proposition}
		
		\begin{proof}  
			
			For $T \neq \mathfrak{T}_o$, $\Gamma(y,S_c^{T^{\pm}})=0$ for all $y \in \mathfrak{s}_o \subset \CLee(U(\To))$ due to the fact that any chain starting with $\To$ contains trees with $i-$grading higher than $\To$ as was shown in Proposition \ref{higher hom than To}. Now for $T=\To$, we need to find all the states $s \in \mathfrak{s}_o \subset \CLee(U(\To))$ such that $\A^{\uparrow} (s,S_c^{\To^\pm})=1$. From Proposition \ref{generalized rooted pos subpath}, we get all such states $s \in \mathfrak{s}_o$ for which $\A^{\uparrow}(s,S_c^{\To^\pm})=1$. Now suppose $\Pt^\pm_s$ be the alternating path from $s$ to $S_c^{\To^\pm}$ in $\mathcal{H}_\Mt((\CLee(U(\To)),\partial_{Lee}))$ and let $M(\Pt_s) = \{e \in \Pt_s \mid e \in \Mt_\To\}$, where $\Mt_\To$ is the matching in $\mathcal{H}_\Mt((\CLee(U(\To)),\partial_{Lee}))$, then $\text{sign}(s) = (-1)^{|M(\Pt_s)|}.\text{sign}\left(S_c^{\To^\pm}\right)$ and $w(\Pt^\pm_s) = (-1)^{|M(\Pt_s)|}.\text{sign}(s) = \text{sign}\left(S_c^{\To^\pm}\right)$. Thus, $\Gamma(s,S_c^{\To^\pm})=\text{sign}\left(S_c^{\To^\pm}\right)$ for all $s \in \mathfrak{s}_o$ and hence, from equation \ref{equal number of paths} we have
			
			$$
			 g(\mathfrak{s}_o) = \sum_{s \in \mathfrak{s}_o} \left(\Gamma(s,S_c^{\To^+}).\To^+ + \Gamma(s,S_c^{\To^-}).\To^-\right) = \text{sign}(S_c^{\To^\pm}).\bigg|\left\{s \mid \A^{\uparrow}\left(s,S_c^{\To^\pm}\right)=1\right\}\bigg|.\left(\To^+ + \To^-\right)
			$$
			
			When the orientation is $\Bar{o}$ then,
			
			$$
			 g(\mathfrak{s}_{\Bar{o}}) = \sum_{s \in \mathfrak{s}_{\Bar{o}}} \left(\Gamma(s,S_c^{{\mathfrak{T}_{\Bar{o}}}^+}).{\mathfrak{T}_{\Bar{o}}}^+ + \Gamma(s,S_c^{{\mathfrak{T}_{\Bar{o}}}^-}).{\mathfrak{T}_{\Bar{o}}}^-\right) = \text{sign}(S_c^{{\mathfrak{T}_{\Bar{o}}}^\pm}).\bigg|\left\{s \mid \A^{\uparrow}\left(s,S_c^{{\mathfrak{T}_{\Bar{o}}}^\pm}\right)=1\right\}\bigg|.\left({\mathfrak{T}_{\Bar{o}}}^+ - {\mathfrak{T}_{\Bar{o}}}^-\right)
			$$ 
        \end{proof}
		
	\end{subsection} 
	
	\begin{subsection}{Filtration on the spanning tree complex and s-invariant}
		We have an analogous filtration on $\CST_{Lee}(K)$ as compared to the filtration $\fil^{Lee}_p(\CLee(K))$ defined in \cite{LEE2005554}. Define a submodule of $\CST_{Lee}^{i,*}(K)$ for each $i$ as
		
		$$
		\fil^{ST}_p(\CST_{Lee}^{i,*}(K)) = \Q \langle T^+,T^- \mid T^{\pm} \in \CST_{Lee}^{i,*}(K) \text{ and } j(T^{\pm}) \geq p \rangle 
		$$
		
		The collection of this submodules clearly defines a filtration on $\CST_{Lee}(K)$. Now the map $g$ defined from $\CLee(K)$ to $\CST_{Lee}(K)$ (See equation \ref{g map in dmt}) is a filtered chain homotopy equivalence, since $g$ is grading preserving map and homotopy equivalence follows from Lemma \ref{chain homotopic maps}. Thus, we have the following commutative diagram:
		
		\begin{equation}\label{cd for inc and gpm}
			\begin{tikzcd}[ampersand replacement=\&,cramped]
				{H^*(\mathcal{F}_p^{Lee}(\CLee(K)))} \&\& {H^*(\CLee(K))} \\
				\\
				{H^*(\mathcal{F}_p^{ST}(\CST_{Lee}(K)))} \&\& {H^*(\CST_{Lee}(K))}
				\arrow["{i_p^{*{Lee}}}", from=1-1, to=1-3]
				\arrow["{g^*}", from=1-3, to=3-3]
				\arrow["{i_p^{*ST}}", from=3-1, to=3-3]
				\arrow["{g^*}", from=1-1, to=3-1]
				\arrow["\cong"', from=1-1, to=3-1]
				\arrow["\cong"', from=1-3, to=3-3]
			\end{tikzcd}
		\end{equation} 
		
		For a given knot $K$, let us define two numerical quantities corresponding to $s_{\max}(K)$ and $s_{\min}(K)$. We denote $s_{ST}$ to be the $j-$grading on $H^*(\CST_{Lee}(K))$ induced by the $j$-grading on $\CST_{Lee}(K)$.
		
		\begin{definition}
			$$
			\begin{aligned}
				&	s^{ST}_{\max}(K) = \max\{s_{ST}(T) \mid [T] \in \CST_{Lee}(K), [T] \neq 0\}\\
				& 	s^{ST}_{\min}(K) = \min\{s_{ST}(T) \mid [T] \in \CST_{Lee}(K), [T] \neq 0\}
			\end{aligned}
			$$
		\end{definition}
		
		Then, $s^{ST}_{\max}(K)$ and $s^{ST}_{\min}(K)$ are the $j-$gradings of the two copies of $\Q$ surviving in the $E_\infty$ term of the spectral sequence obtained from the filtration $\fil^{ST}_p(\CST_{Lee}^{i,*}(K))$. We are now ready prove Theorem \ref{s-invariant thm}.
		
		\begin{proof}[Proof of Theorem \ref{s-invariant thm}.]
			The $s$ and $s_{ST}$ gradings of the generators in $H^*(\CLee(K))$ and $H^*(\CST_{Lee}(K))$ are respectively given by
			
			\begin{equation} \label{s-grad eq}
				\begin{aligned}
					& s([\mathfrak{s}_o \pm \mathfrak{s}_{\Bar{o}}]) = \underset{p}{\max}\{p \mid [\mathfrak{s}_o \pm \mathfrak{s}_{\Bar{o}}] \in \im(i^{*Lee}_p)\} \\
					& s_{ST}([\mathfrak{T}_o^{\pm}]) = \underset{p}{\max}\{p \mid [\mathfrak{T}_o^{\pm}] \in \im(i^{*ST}_p)\}
				\end{aligned}
			\end{equation} 
			
			Now by the above commutative diagram \ref{cd for inc and gpm} and Proposition \ref{oriented tree generators} we have,
			$$
			\{s([\mathfrak{s}_o + \mathfrak{s}_{\Bar{o}}]), s([\mathfrak{s}_o - \mathfrak{s}_{\Bar{o}}])\}=\{s_{ST}([\mathfrak{T}_o^{+}]),s_{ST}([\mathfrak{T}_o^{-}])\}
			$$
			But we know that $s([\mathfrak{s}_o + \mathfrak{s}_{\Bar{o}}])+s([\mathfrak{s}_o - \mathfrak{s}_{\Bar{o}}]) = s_{\max}(K) + s_{\min}(K) = 2s(K)$ and thus, we have our desired result.
		\end{proof}
		
		\begin{proof}[Proof of Corollary \ref{s-inv lower bound}.]
		    Since, $s_{ST}(\To^\pm) \geq j(\To^\pm)$ coming from equation \ref{s-grad eq} and $j(\To^+)=j(\To^-)+2$ hence, the result then follows from Theorem \ref{s-invariant thm}.
		\end{proof}

        \begin{remark} \label{lobb remark}
            Let $D$ be a connected knot diagram of a knot $K$. One can associate a decorated graph to $D$ known as the \textit{Seifert graph}, denoted by $T(D)$. Take the orientation preserving resolution of $D$ and for each circle, we have a vertex for $T(D)$ and we provide an edge with positive sign between two vertices if they share a positive crossing in $D$ otherwise we assign a negative edge. We form two subgraphs of $T(D)$, $T^+(D)$ (remove all the negative edges from $T(D)$) $T^-(D)$ (remove all positive edges from $T(D)$). Lobb proved the inequality \ref{lobb inequality} in \cite{Lobb}. It is important to note that this inequality depends on the diagram $D$. Now, observe that $G(\To)$ can be seen as a spanning tree of $T(D)$ and hence we have,

            $$
             \begin{aligned}
             j(\To^+)-1 & = i(\To^+)+w(D)+(\# L + \# \Bar{l} + 1 - \#\Bar{L} - \# l)-1 \\
             & = w(D) - (\# L + \# \Bar{l} + \#\Bar{L} + \# l + 1) + 2(\#L + \# \Bar{l} + 1) - 1 \\
             & = w(D) - \#\text{nodes}(T(D)) + 2\#\text{components}(T^+(D)) -1 
             \end{aligned}
            $$
        \end{remark}
		
	\end{subsection}
	
\end{section}


\begin{section}{Infinite family of exotic disks} \label{exotic}
	We now detect the exotic phenomena exhibited by the slice disks bounded by the knot $K$ previously mentioned in Figure \ref{exotic_disks}. Our goal will be to produce a chain element in $\CST(K)$ which would differentiate the two slice disks $D_1$ and $D_2$ in Figure \ref{exotic_disks} in terms of smooth isotopy rel boundary. But before that we want to investigate whether both of them are topologically isotopic rel boundary or not. We use the following theorem to do so,
	
	\begin{figure}[ht]
		\centering
		\includesvg[width=0.5\textwidth,inkscape=force]{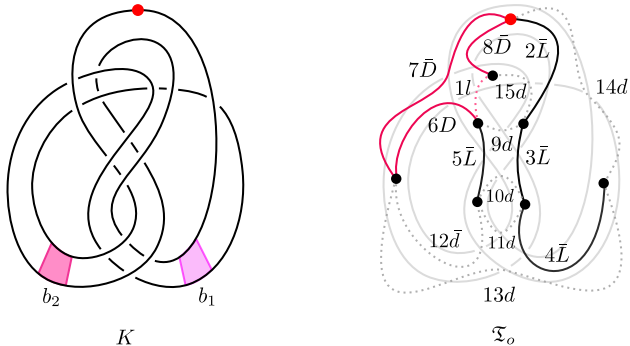}
		\caption{The knot $K$ with two bands $b_1$ and $b_2$ along which surgery results in discs $D_1$ and $D_2$, together with a distinguishing element $\mathfrak{T}_o \in \CST(K)$}
		\label{knot with bands and their distinguishing element}
	\end{figure}
	
	\begin{theorem}[\cite{Conway_powell}]\label{cp_top_isotopy}
		Let $F_1$ and $F_2$ be two locally flat, embedded, compact, orientable genus $g \neq 1,2$ surfaces inside $B^4$, both of which have a common boundary knot $K$ in $S^3$, where $K$ has Alexander polynomial one. Then $F_1$ and $F_2$ are topologically ambiently isotopic rel boundary if and only if $\pi_1(B^4 \setminus F_1) \cong \pi_1(B^4 \setminus F_2) \cong \Z$.
	\end{theorem}
	
	Given a ribbon knot $K \in S^3$, let $D$ be the ribbon disk it bounds inside $B^4$. We use the receipe for drawing the Kirby diagram for the ribbon disk complement $B^4 \setminus D$ explained in \cite{Gompf_Stipsicz}. For each $0$-handle of $D$, we obtain a $1$-handle (dotted circle) for $B^4 \setminus D$. Similarly, for each $1$-handle of $F$, we have a $2$-handle attached to the $1$-handles of $B^4 \setminus D$. We do not consider higher dimensional handles since they can be attached uniquely. The attaching curve for a $2$-handle are two parallel copies running along the core of the ribbon band which merges to form a knot and attached along the two parallel copies of the core of the $1$-handle. The framing of the $2$-handles attached are zero. Following this procedure, we draw the Kirby diagrams of the disk complement for both of the slice disks $D_1$ and $D_2$ in Figure \ref{exotic_disks} (See Figure \ref{kirby diagrams}).
	
	\begin{figure}[ht]
		\centering
		\includesvg[width=0.65\textwidth]{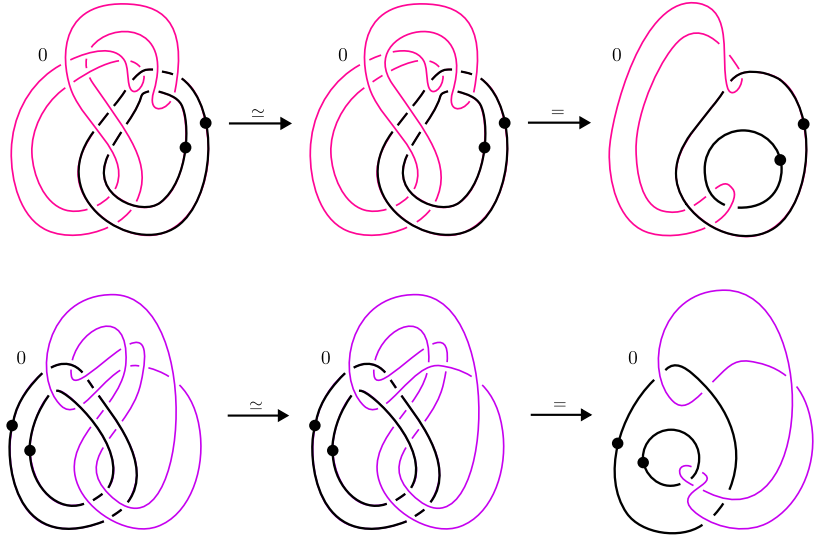}
		\caption{Homotopy of the 0-framed 2-handle produces a cancelling 1-2 handle pair for both $B^4 \setminus D_1$ and $B^4 \setminus D_2$}
		\label{kirby diagrams}
	\end{figure}
	
	\begin{proposition}\label{top-iso-prop}
		The slice disks $D_1$ and $D_2$ in Figrue \ref{exotic_disks} are topologically isotopic rel boundary.
	\end{proposition}
	
	\begin{proof}
		Figure \ref{kirby diagrams} shows that both of the slice disks complement are homotopy equivalent to $S^1 \times B^3$. Thus the proposition follows from Theorem \ref{cp_top_isotopy}.
	\end{proof}
	
	\begin{theorem}[\cite{jac}]\label{jacobsson thm}
		Let $\Sigma_1$ and $\Sigma_2$ be two cobordisms between two links $\Lk_1$ and $\Lk_2$ such that $\Sigma_1$ and $\Sigma_2$ are smoothly isoptic rel. boundary, then $\Kh(\Sigma_1)$ and $\Kh(\Sigma_2)$ are same upto overall sign.
	\end{theorem}
	
	\begin{figure}[!tpbh]
		\centering
		\includesvg[width=0.6\textwidth,inkscape=force]{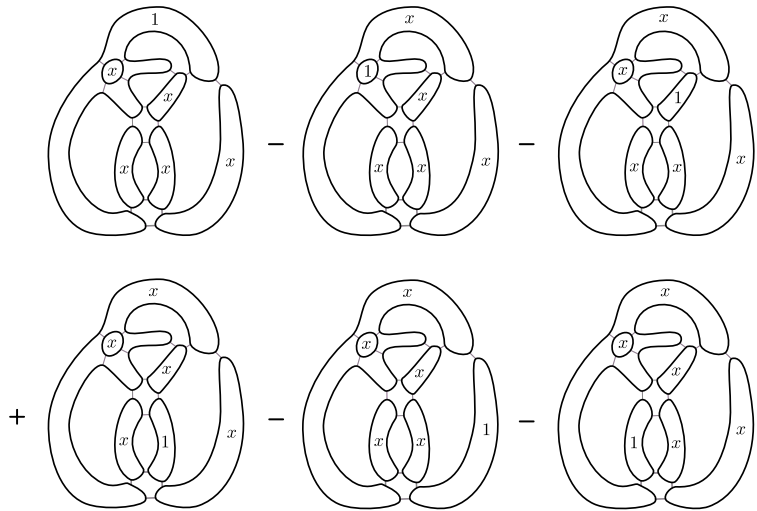}
		\caption{The distinguished oriented resolution cycle $\phi \in \CKh(K)$ distinguishing $D_1$ and $D_2$}
		\label{khovanov dist cycle}
	\end{figure}
	
	\begin{theorem}
		The slice disks $D_1$ and $D_2$ induce distinct maps on homology of the spanning tree complex and hence are not smoothly isotopic rel boundary.
	\end{theorem}
	
	\begin{proof}
		The oriented tree $\To^+$ in Figure \ref{knot with bands and their distinguishing element} is a cycle in $\CST(K)$ and by Proposition \ref{rooted negative subpaths}, we obtain $\phi$ in Figure \ref{khovanov dist cycle} from $f(\To^+)$, where $f$ is the map in equation \ref{f map in dmt} for $\CKh(K)$. The movie description for the fourth summand of $\phi$ for the cobordism due to the band surgery $b_1$ is shown in Figure \ref{movie chain map}. One can easily check that the rest of the summands maps to zero for the band surgery $b_1$ and $\phi$ maps to zero for the band surgery $b_2$. Thus $D_1$ and $D_2$ are not smoothly isotopic rel. boundary due to Theorem \ref{jacobsson thm}.
	\end{proof}

    \begin{figure}[!tbph]
		\centering
		\includesvg[width=0.6\textwidth]{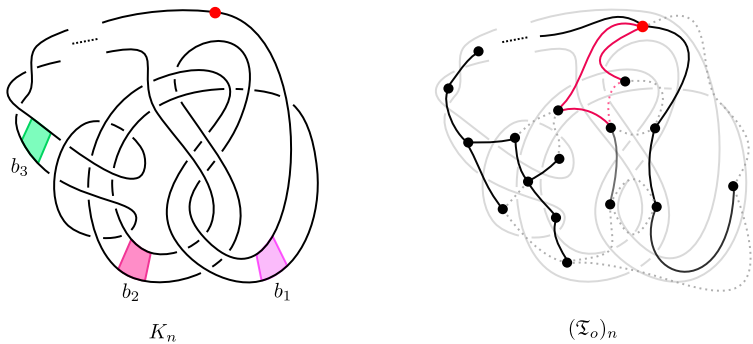}
		\caption{An infinite family of pairwise exotic slice disks $D_n$ and $D_n'$, and their corresponding distinguishing element in the spanning tree complex}
		\label{infinite family exotic disks}
	\end{figure}
	
	\begin{figure}[!tbph]
		\centering
		\includesvg[width=\textwidth,inkscape=force]{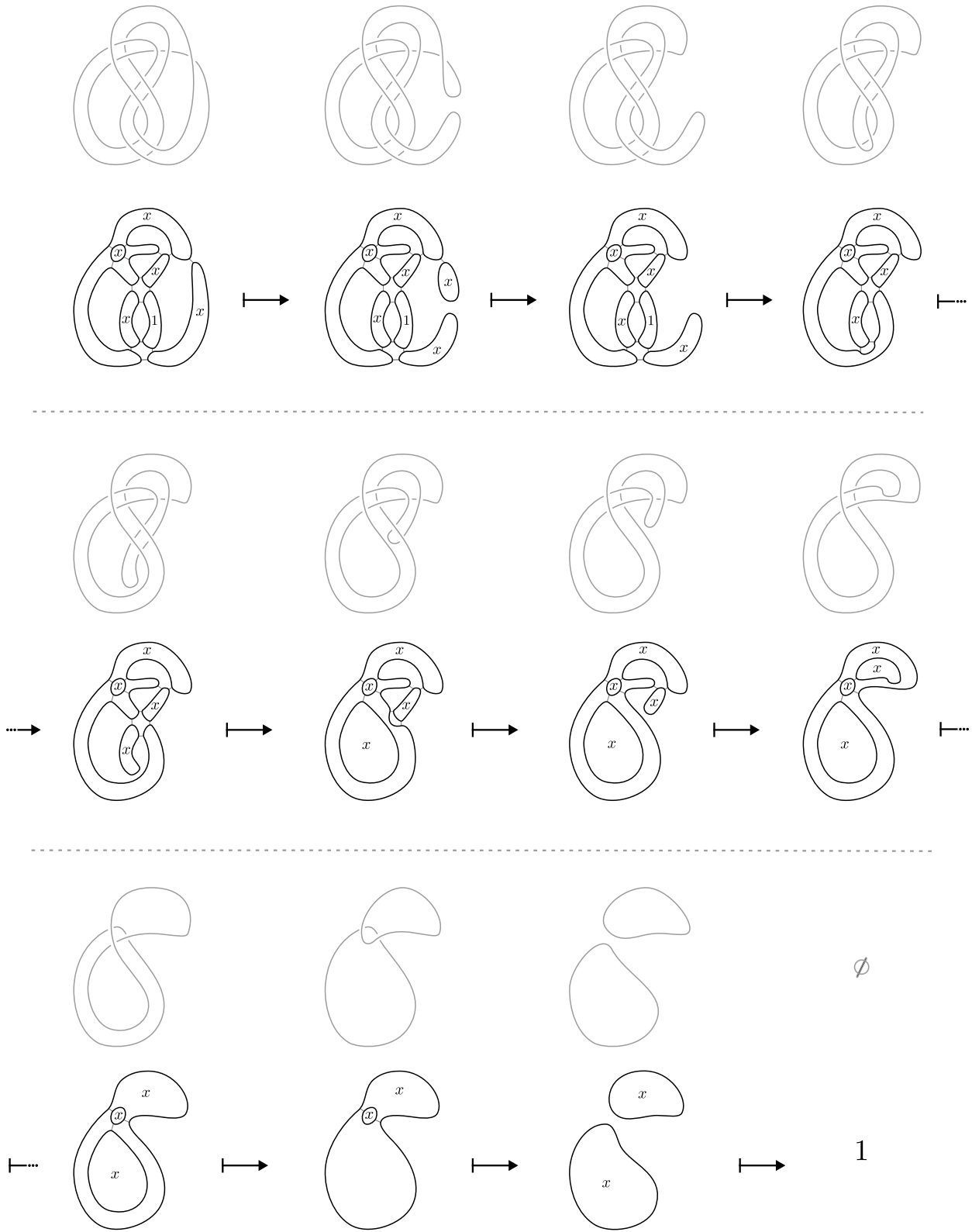}
		\caption{A movie description of the slice disk and the mapping of a summand of $\phi$ under the cobordism due to the band move $b_1$}
		\label{movie chain map}
	\end{figure}
	
	\begin{corollary}
		The slice disks $D_1$ and $D_2$ forms an exotic pair rel boundary.
	\end{corollary}
	
	The fact that the infinite family $K_n$ in Figure \ref{infinite family exotic disks} bounds a pair of slice disks for each $n$ which are exotically knotted rel. boundary follows from the corollary \ref{ribbon conc inj cst} which is a consequence of the following theorem:
	
	\begin{theorem}[\cite{ribbon-inj}]
		If $C$ is a ribbon concordance from $\Lk_1$ to $\Lk_2$ ,
		the induced map $Kh(C) : Kh(\Lk_1) \rightarrow Kh(\Lk_2)$ is injective, with left inverse induced by the reverse of $C$ , viewed as a cobordism from $\Lk_2$ to $\Lk_1$ .
	\end{theorem}
	
	\begin{corollary} \label{ribbon conc inj cst}
		Let $\Sigma_1$ and $\Sigma_2$ be cobordisms from $\Lk_0$ to $\Lk_1$ and let $C$ be a ribbon concordance from $\Lk_1$ to $\Lk_2$. If $\Sigma_1$ and $\Sigma_2$ induce distinct maps $\Kh(\Lk_0) \rightarrow \Kh(\Lk_1)$, then the cobordisms $C \circ \Sigma_1$ and $C \circ \Sigma_2$ induce distinct maps $\Kh(\Lk_0) \rightarrow \Kh(\Lk_2)$.\\
		
		Similarly, if the reverses of $\Sigma_1$ and $\Sigma_2$ induce distinct maps $\Kh(\Lk_1) \rightarrow \Kh(\Lk_2)$, then the reverses of $C \circ \Sigma_1$ and $C \circ \Sigma_2$ induce distinct map $\Kh(\Lk_2) \rightarrow \Kh(\Lk_0)$.
	\end{corollary}
	
	\begin{proof}
		Suppose $\Sigma_1$ and $\Sigma_2$ induce distinct maps on $\Kh(\Lk_0) \rightarrow \Kh(\Lk_1)$, then there exists $S \in \Kh(\Lk_0)$ such that $\Kh(\Sigma_1)(S) \neq \Kh(\Sigma_2)(S)$. Since $C$ induces an injective map $\Kh(\Lk_1) \rightarrow \Kh(\Lk_2)$, we have
		
		$$
		 \Kh(C \circ \Sigma_1)(S) - \Kh(C \circ \Sigma_2)(S) = \Kh(C)(\Kh(\Sigma_1)(S) - \Kh(\Sigma_2)(S)) \neq 0
		$$
		
		An analogous argument applies to the reversed cobordisms, appealing instead to the surjectivity of the map $\Kh(\Lk_2) \rightarrow \Kh(\Lk_1)$ induced by the reverse of $C$.
	\end{proof}

	\begin{proof}[Proof of Theorem \ref{application in exotic disks}]
		The disks $D_n$ and $D_n'$ are obtained by extending the disks $D_1$ and $D_2$ respectively using the concordance from $K$ to $K_n$ which implies that $D_n$ and $D_n'$ are topologically isotopic rel. boundary as $D_1$ and $D_2$ are from Proposition \ref{top-iso-prop}. The band surgery due to $b_3$ in Figure \ref{infinite family exotic disks} produces a ribbon concordance between $K_n$ and $K$ for each $n$ and since $K$ bounds a pair of slice disks which are exotically knotted rel. boundary, the rest follows from corollary \ref{ribbon conc inj cst}.
	\end{proof}
	
	\begin{remark}
		Theorem \ref{application in exotic disks} can be made even stronger by removing the rel. boundary condition. If we have a knot $K$ whose diffeomorphism symmetry group is trivial and if $K$ bounds surfaces $\Sigma$ and $\Sigma'$ in $B^4$ that are ambiently isotopic, then $\Sigma$ and $\Sigma'$ are also ambiently isoptopic rel. boundary. From this above result, it readily follows that the disks bounded by the knots $K_n$ in Figure \ref{infinite family exotic disks} are not smoothly isotopic under any ambient diffeomorphism once we prove that $S^3 \setminus K_n$ has trivial isometry group. We direct the interested reader to section \ref{iso-grp-sec} for a proof of the above result.
	\end{remark}
	
\end{section}

\appendix

\begin{section}{The isometry group of \texorpdfstring{$K_n$}{Kn}}\label{iso-grp-sec}

\begin{theorem}
	The link complement $S^3 \setminus \Lk$ is hyperbolic with trivial isometry group.
\end{theorem}

\begin{proof}

We used SnapPy's link editor to draw the link $\Lk$ and obtained its DT code. We then went to the Sage terminal where we use the following code to verify that the complement of $\Lk$ has trivial isometry group.

\begin{figure}[ht]
	\centering
	\includesvg[width=0.28\textwidth]{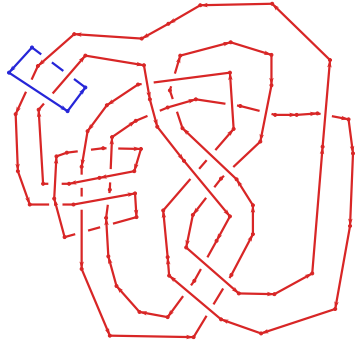}
	\caption{A 3--component link $\Lk$ with 1 unknotted component along which we perform a Dehn felling}
\end{figure}

\begin{verbatim}
	sage: import snappy
	sage: L=snappy.ManifoldHP('DT: [(56,-8,24,-42,14,-26,40,-58,38,34,-52,6,-12,-18,
	....: -36,-54,22,-46,48,10,-4,-32,50,-28,30,-20,44),(-16,2)], [1,0,0,0,0,0,0,1,1
	....: ,0,0,1,1,0,0,0,1,0,1,1,1,1,0,0,1,1,1,0,0]')
	sage: R=L.canonical_retriangulation(verified=True)
	sage: len(R.isomoprhisms_to(R))
	1
\end{verbatim}

Thus, the size of the isometry group is 1, so the identity is the unique isometry of $S^3\setminus \Lk$. One can also verify that $S^3 \setminus \Lk$ is hyperbolic using the  command \verb|L.verify_hyperbolicity()|.

\end{proof}

\begin{proposition}
	The knot complement $S^3 \setminus K_n$ is hyperbolic with trivial isometry group.
\end{proposition}

\begin{proof}
	One can obtain the knot complement $S^3 \setminus K_n$ by Dehn filling along the unknotted component in $\Lk$ with slope $1/n$. We then use the standard technique due to Thurston’s hyperbolic Dehn surgery theorem that $S^3 \setminus K_n$ is hyperbolic and has trivial isometry group since $S^3 \setminus \Lk$ has trivial isometry group. More details about this proof can be found in \cite{kh-isaac-hayden}.
\end{proof}

\end{section}

\begin{section}{Khovanov homology of \texorpdfstring{$8_{20}$}{820}}\label{example computation}

The Table \ref{spanning trees of 8_20} below lists all the spanning of the Tait graph $G_\Lk$ in Figure \ref{example 8_20}, Table \ref{chain complex table for 8_20} provides the spanning tree complex of $G_\Lk$, and Table \ref{incidences of example} provides all the incidences for the differential in the spanning tree complex of $G_\Lk$ together with the Khovanov homology groups of the knot $8_{20}$.

	\setlength{\tabcolsep}{5mm}
	\def\arraystretch{1.25}
	
	\begin{table}[!tbph]
		\centering
		\resizebox{\textwidth}{!}{ 
			\begin{tabular}{|c|c|c|c|c|c|c|}
				\hline
				\begin{tikzpicture}[vertex/.style={circle, draw, inner sep=0pt, minimum size=6pt, fill=black}, edge/.style={draw}, baseline=0]
					\node[vertex,fill=red,color=red] (a) at (0,2) {};
					\node[vertex] (b) at (0,1) {};
					\node[vertex] (c) at (0,0) {};
					\node[vertex] (d) at (0,-1) {};
					\node[vertex] (e) at (-1.5,1) {};
					\node[vertex] (f) at (-1.5,0) {};
					\node[vertex] (g) at (1.5,0.5) {};
					\path[edge] (b) to (c);
					\path[edge] (c) to (d);
					\path[edge] (e) to (f);
					\path[edge] (f) to[out=290,in=180] (d);
					\path[edge] (a) to[out=0,in=90]  (g);
					\path[edge] (g) to[out=270,in=0] (d);
					\node at (0,-2) {$T_1$ ($lDDDD\Bar{d}\Bar{D}\Bar{D}$)};
				\end{tikzpicture} &
				\begin{tikzpicture}[vertex/.style={circle, draw, inner sep=0pt, minimum size=6pt, fill=black}, edge/.style={draw}, baseline=0]
					\node[vertex,fill=red,color=red] (a) at (0,2) {};
					\node[vertex] (b) at (0,1) {};
					\node[vertex] (c) at (0,0) {};
					\node[vertex] (d) at (0,-1) {};
					\node[vertex] (e) at (-1.5,1) {};
					\node[vertex] (f) at (-1.5,0) {};
					\node[vertex] (g) at (1.5,0.5) {};
					\path[edge] (b) to (c);
					\path[edge] (c) to (d);
					\path[edge] (a) to[out=180,in=70] (e);
					\path[edge] (e) to (f);
					\path[edge] (f) to[out=290,in=180] (d);
					\path[edge] (g) to[out=270,in=0] (d);
					\node at (0,-2) {$T_2$ ($LdDDD\Bar{d}\Bar{D}\Bar{D}$)};
				\end{tikzpicture} &
				\begin{tikzpicture}[vertex/.style={circle, draw, inner sep=0pt, minimum size=6pt, fill=black}, edge/.style={draw}, baseline=0]
					\node[vertex,fill=red,color=red] (a) at (0,2) {};
					\node[vertex] (b) at (0,1) {};
					\node[vertex] (c) at (0,0) {};
					\node[vertex] (d) at (0,-1) {};
					\node[vertex] (e) at (-1.5,1) {};
					\node[vertex] (f) at (-1.5,0) {};
					\node[vertex] (g) at (1.5,0.5) {};
					\path[edge] (b) to (c);
					\path[edge] (c) to (d);
					\path[edge] (a) to[out=180,in=70] (e);
					\path[edge] (e) to (f);
					\path[edge] (f) to[out=290,in=180] (d);
					\path[edge] (a) to[out=0,in=90] (g);
					\node at (0,-2) {$T_3$ ($LLdDD\Bar{d}\Bar{D}\Bar{D}$)};
				\end{tikzpicture} &
				\begin{tikzpicture}[vertex/.style={circle, draw, inner sep=0pt, minimum size=6pt, fill=black}, edge/.style={draw}, baseline=0]
					\node[vertex,fill=red,color=red] (a) at (0,2) {};
					\node[vertex] (b) at (0,1) {};
					\node[vertex] (c) at (0,0) {};
					\node[vertex] (d) at (0,-1) {};
					\node[vertex] (e) at (-1.5,1) {};
					\node[vertex] (f) at (-1.5,0) {};
					\node[vertex] (g) at (1.5,0.5) {};
					\path[edge] (b) to (c);
					\path[edge] (c) to (d);
					\path[edge] (a) to[out=180,in=70] (e);
					\path[edge] (e) to (f);
					\path[edge] (a) to[out=0,in=90] (g);
					\path[edge] (g) to[out=270,in=0] (d);
					\node at (0,-2) {$T_4$ ($LLLdD\Bar{d}\Bar{D}\Bar{D}$)};
				\end{tikzpicture} &
				\begin{tikzpicture}[vertex/.style={circle, draw, inner sep=0pt, minimum size=6pt, fill=black}, edge/.style={draw}, baseline=0]
					\node[vertex,fill=red,color=red] (a) at (0,2) {};
					\node[vertex] (b) at (0,1) {};
					\node[vertex] (c) at (0,0) {};
					\node[vertex] (d) at (0,-1) {};
					\node[vertex] (e) at (-1.5,1) {};
					\node[vertex] (f) at (-1.5,0) {};
					\node[vertex] (g) at (1.5,0.5) {};
					\path[edge] (b) to (c);
					\path[edge] (c) to (d);
					\path[edge] (a) to[out=180,in=70] (e);
					\path[edge] (f) to[out=290,in=180] (d);
					\path[edge] (a) to[out=0,in=90] (g);
					\path[edge] (g) to[out=270,in=0] (d);
					\node at (0,-2) {$T_5$ ($LLLLd\Bar{d}\Bar{D}\Bar{D}$)};
				\end{tikzpicture} &
				\begin{tikzpicture}[vertex/.style={circle, draw, inner sep=0pt, minimum size=6pt, fill=black}, edge/.style={draw}, baseline=0]
					\node[vertex,fill=red,color=red] (a) at (0,2) {};
					\node[vertex] (b) at (0,1) {};
					\node[vertex] (c) at (0,0) {};
					\node[vertex] (d) at (0,-1) {};
					\node[vertex] (e) at (-1.5,1) {};
					\node[vertex] (f) at (-1.5,0) {};
					\node[vertex] (g) at (1.5,0.5) {};
					\path[edge] (a) to (b);
					\path[edge] (c) to (d);
					\path[edge] (e) to (f);
					\path[edge] (f) to[out=290,in=180] (d);
					\path[edge] (a) to[out=0,in=90] (g);
					\path[edge] (g) to[out=270,in=0] (d);
					\node at (0,-2) {$T_6$ ($lDDDD\Bar{L}\Bar{d}\Bar{D}$)};
				\end{tikzpicture} &
				\begin{tikzpicture}[vertex/.style={circle, draw, inner sep=0pt, minimum size=6pt, fill=black}, edge/.style={draw}, baseline=0]
					\node[vertex,fill=red,color=red] (a) at (0,2) {};
					\node[vertex] (b) at (0,1) {};
					\node[vertex] (c) at (0,0) {};
					\node[vertex] (d) at (0,-1) {};
					\node[vertex] (e) at (-1.5,1) {};
					\node[vertex] (f) at (-1.5,0) {};
					\node[vertex] (g) at (1.5,0.5) {};
					\path[edge] (a) to (b);
					\path[edge] (c) to (d);
					\path[edge] (a) to[out=180,in=70] (e);
					\path[edge] (e) to (f);
					\path[edge] (f) to[out=290,in=180] (d);
					\path[edge] (g) to[out=270,in=0] (d);
					\node at (0,-2) {$T_7$ ($LdDDD\Bar{L}\Bar{d}\Bar{D}$)};
				\end{tikzpicture} \\
                \hline
				\begin{tikzpicture}[vertex/.style={circle, draw, inner sep=0pt, minimum size=6pt, fill=black}, edge/.style={draw}, baseline=0]
					\node[vertex,fill=red,color=red] (a) at (0,2) {};
					\node[vertex] (b) at (0,1) {};
					\node[vertex] (c) at (0,0) {};
					\node[vertex] (d) at (0,-1) {};
					\node[vertex] (e) at (-1.5,1) {};
					\node[vertex] (f) at (-1.5,0) {};
					\node[vertex] (g) at (1.5,0.5) {};
					\path[edge] (a) to (b);
					\path[edge] (c) to (d);
					\path[edge] (a) to[out=180,in=70] (e);
					\path[edge] (e) to (f);
					\path[edge] (f) to[out=290,in=180] (d);
					\path[edge] (a) to[out=0,in=90] (g);
					\node at (0,-2) {$T_8$ ($LLdDD\Bar{L}\Bar{d}\Bar{D}$)};
				\end{tikzpicture} &
				\begin{tikzpicture}[vertex/.style={circle, draw, inner sep=0pt, minimum size=6pt, fill=black}, edge/.style={draw}, baseline=0]
					\node[vertex,fill=red,color=red] (a) at (0,2) {};
					\node[vertex] (b) at (0,1) {};
					\node[vertex] (c) at (0,0) {};
					\node[vertex] (d) at (0,-1) {};
					\node[vertex] (e) at (-1.5,1) {};
					\node[vertex] (f) at (-1.5,0) {};
					\node[vertex] (g) at (1.5,0.5) {};
					\path[edge] (a) to (b);
					\path[edge] (c) to (d);
					\path[edge] (a) to[out=180,in=70] (e);
					\path[edge] (e) to (f);
					\path[edge] (a) to[out=0,in=90] (g);
					\path[edge] (g) to[out=270,in=0] (d);
					\node at (0,-2) {$T_9$ ($LLLdD\Bar{L}\Bar{d}\Bar{D}$)};
				\end{tikzpicture} &
				\begin{tikzpicture}[vertex/.style={circle, draw, inner sep=0pt, minimum size=6pt, fill=black}, edge/.style={draw}, baseline=0]
					\node[vertex,fill=red,color=red] (a) at (0,2) {};
					\node[vertex] (b) at (0,1) {};
					\node[vertex] (c) at (0,0) {};
					\node[vertex] (d) at (0,-1) {};
					\node[vertex] (e) at (-1.5,1) {};
					\node[vertex] (f) at (-1.5,0) {};
					\node[vertex] (g) at (1.5,0.5) {};
					\path[edge] (a) to (b);
					\path[edge] (c) to (d);
					\path[edge] (a) to[out=180,in=70] (e);
					\path[edge] (f) to[out=290,in=180] (d);
					\path[edge] (a) to[out=0,in=90] (g);
					\path[edge] (g) to[out=270,in=0] (d);
					\node at (0,-2) {$T_{10}$ ($LLLLd\Bar{L}\Bar{d}\Bar{D}$)};
				\end{tikzpicture} &
				\begin{tikzpicture}[vertex/.style={circle, draw, inner sep=0pt, minimum size=6pt, fill=black}, edge/.style={draw}, baseline=0]
					\node[vertex,fill=red,color=red] (a) at (0,2) {};
					\node[vertex] (b) at (0,1) {};
					\node[vertex] (c) at (0,0) {};
					\node[vertex] (d) at (0,-1) {};
					\node[vertex] (e) at (-1.5,1) {};
					\node[vertex] (f) at (-1.5,0) {};
					\node[vertex] (g) at (1.5,0.5) {};
					\path[edge] (a) to (b);
					\path[edge] (b) to (c);
					\path[edge] (e) to (f);
					\path[edge] (f) to[out=290,in=180] (d);
					\path[edge] (a) to[out=0,in=90] (g);
					\path[edge] (g) to[out=270,in=0] (d);
					\node at (0,-2) {$T_{11}$ ($lDDDD\Bar{L}\Bar{L}\Bar{d}$)};
				\end{tikzpicture} &
				\begin{tikzpicture}[vertex/.style={circle, draw, inner sep=0pt, minimum size=6pt, fill=black}, edge/.style={draw}, baseline=0]
					\node[vertex,fill=red,color=red] (a) at (0,2) {};
					\node[vertex] (b) at (0,1) {};
					\node[vertex] (c) at (0,0) {};
					\node[vertex] (d) at (0,-1) {};
					\node[vertex] (e) at (-1.5,1) {};
					\node[vertex] (f) at (-1.5,0) {};
					\node[vertex] (g) at (1.5,0.5) {};
					\path[edge] (a) to (b);
					\path[edge] (b) to (c);
					\path[edge] (a) to[out=180,in=70] (e);
					\path[edge] (e) to (f);
					\path[edge] (f) to[out=290,in=180] (d);
					\path[edge] (g) to[out=270,in=0] (d);
					\node at (0,-2) {$T_{12}$ ($LdDDD\Bar{L}\Bar{L}\Bar{d}$)};
				\end{tikzpicture} &
				\begin{tikzpicture}[vertex/.style={circle, draw, inner sep=0pt, minimum size=6pt, fill=black}, edge/.style={draw}, baseline=0]
					\node[vertex,fill=red,color=red] (a) at (0,2) {};
					\node[vertex] (b) at (0,1) {};
					\node[vertex] (c) at (0,0) {};
					\node[vertex] (d) at (0,-1) {};
					\node[vertex] (e) at (-1.5,1) {};
					\node[vertex] (f) at (-1.5,0) {};
					\node[vertex] (g) at (1.5,0.5) {};
					\path[edge] (a) to (b);
					\path[edge] (b) to (c);
					\path[edge] (a) to[out=180,in=70] (e);
					\path[edge] (e) to (f);
					\path[edge] (f) to[out=290,in=180] (d);
					\path[edge] (a) to[out=0,in=90] (g);
					\node at (0,-2) {$T_{13}$ ($LLdDD\Bar{L}\Bar{L}\Bar{d}$)};
				\end{tikzpicture} &
				\begin{tikzpicture}[vertex/.style={circle, draw, inner sep=0pt, minimum size=6pt, fill=black}, edge/.style={draw}, baseline=0]
					\node[vertex,fill=red,color=red] (a) at (0,2) {};
					\node[vertex] (b) at (0,1) {};
					\node[vertex] (c) at (0,0) {};
					\node[vertex] (d) at (0,-1) {};
					\node[vertex] (e) at (-1.5,1) {};
					\node[vertex] (f) at (-1.5,0) {};
					\node[vertex] (g) at (1.5,0.5) {};
					\path[edge] (a) to (b);
					\path[edge] (b) to (c);
					\path[edge] (a) to[out=180,in=70] (e);
					\path[edge] (e) to (f);
					\path[edge] (a) to[out=0,in=90] (g);
					\path[edge] (g) to[out=270,in=0] (d);
					\node at (0,-2) {$T_{14}$ ($LLLdD\Bar{L}\Bar{L}\Bar{d}$)};
				\end{tikzpicture} \\
                \hline
				\begin{tikzpicture}[vertex/.style={circle, draw, inner sep=0pt, minimum size=6pt, fill=black}, edge/.style={draw}, baseline=0]
					\node[vertex,fill=red,color=red] (a) at (0,2) {};
					\node[vertex] (b) at (0,1) {};
					\node[vertex] (c) at (0,0) {};
					\node[vertex] (d) at (0,-1) {};
					\node[vertex] (e) at (-1.5,1) {};
					\node[vertex] (f) at (-1.5,0) {};
					\node[vertex] (g) at (1.5,0.5) {};
					\path[edge] (a) to (b);
					\path[edge] (b) to (c);
					\path[edge] (a) to[out=180,in=70] (e);
					\path[edge] (f) to[out=290,in=180] (d);
					\path[edge] (a) to[out=0,in=90] (g);
					\path[edge] (g) to[out=270,in=0] (d);
					\node at (0,-2) {$T_{15}$ ($LLLLd\Bar{L}\Bar{L}\Bar{d}$)};
				\end{tikzpicture} &
				\begin{tikzpicture}[vertex/.style={circle, draw, inner sep=0pt, minimum size=6pt, fill=black}, edge/.style={draw}, baseline=0]
					\node[vertex,fill=red,color=red] (a) at (0,2) {};
					\node[vertex] (b) at (0,1) {};
					\node[vertex] (c) at (0,0) {};
					\node[vertex] (d) at (0,-1) {};
					\node[vertex] (e) at (-1.5,1) {};
					\node[vertex] (f) at (-1.5,0) {};
					\node[vertex] (g) at (1.5,0.5) {};
					\path[edge] (a) to (b);
					\path[edge] (b) to (c);
					\path[edge] (c) to (d);
					\path[edge] (e) to (f);
					\path[edge] (f) to[out=290,in=180] (d);
					\path[edge] (g) to[out=270,in=0] (d);
					\node at (0,-2) {$T_{16}$ ($llDDD\Bar{D}\Bar{D}\Bar{D}$)};
				\end{tikzpicture} &
				\begin{tikzpicture}[vertex/.style={circle, draw, inner sep=0pt, minimum size=6pt, fill=black}, edge/.style={draw}, baseline=0]
					\node[vertex,fill=red,color=red] (a) at (0,2) {};
					\node[vertex] (b) at (0,1) {};
					\node[vertex] (c) at (0,0) {};
					\node[vertex] (d) at (0,-1) {};
					\node[vertex] (e) at (-1.5,1) {};
					\node[vertex] (f) at (-1.5,0) {};
					\node[vertex] (g) at (1.5,0.5) {};
					\path[edge] (a) to (b);
					\path[edge] (b) to (c);
					\path[edge] (c) to (d);
					\path[edge] (e) to (f);
					\path[edge] (f) to[out=290,in=180] (d);
					\path[edge] (a) to[out=0,in=90] (g);
					\node at (0,-2) {$T_{17}$ ($lLdDD\Bar{D}\Bar{D}\Bar{D}$)};
				\end{tikzpicture} &
				\begin{tikzpicture}[vertex/.style={circle, draw, inner sep=0pt, minimum size=6pt, fill=black}, edge/.style={draw}, baseline=0]
					\node[vertex,fill=red,color=red] (a) at (0,2) {};
					\node[vertex] (b) at (0,1) {};
					\node[vertex] (c) at (0,0) {};
					\node[vertex] (d) at (0,-1) {};
					\node[vertex] (e) at (-1.5,1) {};
					\node[vertex] (f) at (-1.5,0) {};
					\node[vertex] (g) at (1.5,0.5) {};
					\path[edge] (a) to (b);
					\path[edge] (b) to (c);
					\path[edge] (c) to (d);
					\path[edge] (a) to[out=180,in=70] (e);
					\path[edge] (e) to (f);	
					\path[edge] (g) to[out=270,in=0] (d);
					\node at (0,-2) {$T_{18}$ ($lDDDD\Bar{d}\Bar{D}\Bar{D}$)};
				\end{tikzpicture} &
				\begin{tikzpicture}[vertex/.style={circle, draw, inner sep=0pt, minimum size=6pt, fill=black}, edge/.style={draw}, baseline=0]
					\node[vertex,fill=red,color=red] (a) at (0,2) {};
					\node[vertex] (b) at (0,1) {};
					\node[vertex] (c) at (0,0) {};
					\node[vertex] (d) at (0,-1) {};
					\node[vertex] (e) at (-1.5,1) {};
					\node[vertex] (f) at (-1.5,0) {};
					\node[vertex] (g) at (1.5,0.5) {};
					\path[edge] (a) to (b);
					\path[edge] (b) to (c);
					\path[edge] (c) to (d);
					\path[edge] (a) to[out=180,in=70] (e);		   	
					\path[edge] (f) to[out=290,in=180] (d);		   
					\path[edge] (g) to[out=270,in=0] (d);
					\node at (0,-2) {$T_{19}$ ($LlDLd\Bar{D}\Bar{D}\Bar{D}$)};
				\end{tikzpicture} &
				\begin{tikzpicture}[vertex/.style={circle, draw, inner sep=0pt, minimum size=6pt, fill=black}, edge/.style={draw}, baseline=0]
					\node[vertex,fill=red,color=red] (a) at (0,2) {};
					\node[vertex] (b) at (0,1) {};
					\node[vertex] (c) at (0,0) {};
					\node[vertex] (d) at (0,-1) {};
					\node[vertex] (e) at (-1.5,1) {};
					\node[vertex] (f) at (-1.5,0) {};
					\node[vertex] (g) at (1.5,0.5) {};
					\path[edge] (a) to (b);
					\path[edge] (b) to (c);
					\path[edge] (c) to (d);
					\path[edge] (a) to[out=180,in=70] (e);
					\path[edge] (e) to (f);			   	
					\path[edge] (a) to[out=0,in=90] (g);
					\node at (0,-2) {$T_{20}$ ($LLddD\Bar{D}\Bar{D}\Bar{D}$)};	
				\end{tikzpicture} &
				\begin{tikzpicture}[vertex/.style={circle, draw, inner sep=0pt, minimum size=6pt, fill=black}, edge/.style={draw}, baseline=0]
					\node[vertex,fill=red,color=red] (a) at (0,2) {};
					\node[vertex] (b) at (0,1) {};
					\node[vertex] (c) at (0,0) {};
					\node[vertex] (d) at (0,-1) {};
					\node[vertex] (e) at (-1.5,1) {};
					\node[vertex] (f) at (-1.5,0) {};
					\node[vertex] (g) at (1.5,0.5) {};
					\path[edge] (a) to (b);
					\path[edge] (b) to (c);
					\path[edge] (c) to (d);
					\path[edge] (a) to[out=180,in=70] (e);
					\path[edge] (f) to[out=290,in=180] (d);
					\path[edge] (a) to[out=0,in=90] (g);
					\node at (0,-2) {$T_{21}$ ($LLdLd\Bar{D}\Bar{D}\Bar{D}$)};
				\end{tikzpicture} \\
				\hline
		\end{tabular}}
		\vspace{0.2cm}
		\caption{Spanning trees of Tait graph of $8_{20}$}
           \label{spanning trees of 8_20}
	\end{table}
	\vspace{-0.5cm}
	
	\begin{table}[!tbph]
		\setlength{\tabcolsep}{2mm}
		\def\arraystretch{1.5}
		\centering
        \resizebox{\textwidth}{!}{
		\begin{tabular}{|c|c|c|c|c|c|c|c|c|c|}
			\hline
			\backslashbox{$j$}{$i$} & $-5$ & $-4$ & $-3$ & $-2$ & $-1$ & $0$ & $1$ & $2$ & $3$ \\ \hline
			$5$  &  &  &  & & & & & $T_5^+$ & $T_{21}^+$ \\ \hline
			$3$ & & & & & & & $T_{10}^+, T_{4}^+$ & $T_{20}^+, T_5^-$ & $T_{21}^-$ \\ \hline
			$1$ & & & & & &
			\begin{tabular}{c}
				$T_{15}^+, T_9^+$ \\
				$T_3^+$
			\end{tabular}  & 
			\begin{tabular}{c}
				$T_{19}^+, T_{10}^-$\\
				$ T_4^-$
			\end{tabular} 
			& $T_{20}^-$ &  \\ \hline
			$-1$ & & & & & \begin{tabular}{c}
				$T_{14}^+, T_8^+$ \\
				$T_2^+$
			\end{tabular} &
			\begin{tabular}{c}
				$T_{17}^+,T_{15}^-$\\
				$ T_{18}^+,T_9^-$\\
				$ T_3^-$
			\end{tabular}  
			& $T_{19}^-$ & & \\ \hline
			$-3$ & & & & $T_{13}^+,T_7^+$ & \begin{tabular}{c}
				$T_{14}^-, T_8^-$ \\
				$T_2^-$
			\end{tabular} & $T_{17}^-,T_{18}^-$ & & &  \\ \hline
			$-5$ & & & $T_{12}^+, T_1^+$ &
			\begin{tabular}{c}
				$T_{16}^+, T_{13}^-$\\
				$T_7^-$
			\end{tabular}  & & & & & \\ \hline 
			$-7$ & & $T_6^+$ & $T_{12}^-, T_1^-$ & $T_{16}^-$ & & & & & \\ \hline
			$-9$ & $T_{11}^+$ & $T_6^-$ & & & & & & & \\ \hline
			$-11$ & $T_{11}^-$ & & & & & & & & \\ \hline
		\end{tabular}}
		\vspace{0.1cm}
		\caption{Spanning tree complex of $8_{20}$}
		\label{chain complex table for 8_20}
	\end{table}
	
	\begin{table}[ht]
	    \centering
	    \setlength{\tabcolsep}{2mm}
		\def\arraystretch{1.5}
        \resizebox{0.45\textwidth}{!}{
	    \begin{tabular}{|l|l|}
	    	\hline
	    	$T_{11}^+ \mapsto 2T_6^-$ & $T_{17}^+ \mapsto 0$\\
	    	\hline
	    	$T_6^+ \mapsto 0$ & $T_{18}^+ \mapsto -2T_{19}^-$\\
	    	\hline
	    	$T_{12}^- \mapsto 0$ & $T_{15}^- \mapsto T_{19}^-$\\
	    	\hline
	    	$T_{12}^+ \mapsto 2T_{13}^-$ & $T_9^- \mapsto T_{19}^-$\\
	    	\hline 
	    	$T_1^+ \mapsto T_{16}^+$  & $T_3^- \mapsto T_{19}^-$\\
	    	\hline
	    	$T_{13}^+ \mapsto 0$ & $T_{15}^+ \mapsto T_{19}^+ - 2T_{10}^-$\\
	    	\hline
	    	$T_7^+ \mapsto -2T_8^- - 2T_2^-$ & $T_9^+ \mapsto T_{19}^+ -2T_{10}^-$\\
	    	\hline
	    	$T_{14}^- \mapsto T_{18}^-$ & $T_3^+ \mapsto T_{19}^+$\\
	    	\hline
	    	$T_8^- \mapsto -T_{17}^-$ & $T_{19}^+ \mapsto 0$\\
	    	\hline
	    	$T_2^- \mapsto T_{17}^-$ & $T_{10}^- \mapsto 0$\\
	    	\hline
	    	$T_{14}^+ \mapsto -2T_{15}^- + 2T_9^- + T_{18}^+$ & $T_4^- \mapsto T_{20}^-$\\
	    	\hline
	    	$T_8^+ \mapsto -T_{17}^+ + 2T_3^-$ & $T_{10}^+ \mapsto 0$ \\
	    	\hline
	    	$T_2^+ \mapsto T_{17}^+ - 2T_3^-$ & $T_4^+ \mapsto T_{20}^+ - 2T_5^-$\\
	    	\hline
	    	$T_{20}^+ \mapsto -2T_{21}^-$ & $T_5^\pm \mapsto -T_{21}^\pm$\\
	    	\hline
	    \end{tabular}}
        \hfill
            \resizebox{0.47\textwidth}{!}{
		\begin{tabular}{|c|c|c|c|c|c|c|c|}
			\hline
			\backslashbox{$j$}{$i$} & $-5$ & $-4$ & $-3$ & $-2$ & $-1$ & $0$ & $1$ \\ \hline
			$3$ & & & & & & & $\Z$ \\ \hline
			$1$ & & & & & & $\Z$ & $\Z_2$  \\ \hline
			$-1$ & & & & & $\Z$ & $\Z \oplus \Z$ &\\ \hline
			$-3$ & & & & $\Z$ & $\Z_2$ & &   \\ \hline
			$-5$ & & & & $\Z \oplus \Z_2$ & & & \\ \hline 
			$-7$ & & $\Z$ & $\Z$ & & & & \\ \hline
			$-9$ & & $\Z_2$ & & & & &  \\ \hline
			$-11$ & $\Z$ & & & & & &  \\ \hline
		\end{tabular}}
	    \vspace{0.2cm}
	    \caption{Incidences of spanning trees and the Khovanov homology groups derived from the differential}
          \label{incidences of example}
	\end{table}
	
\end{section}	
\clearpage


\bibliographystyle{alpha}
\bibliography{citations}

\end{document}